\documentclass[12 pt,reqno]{amsart}
\usepackage{xargs}
\usepackage[normalem]{ulem}

\usepackage{pdfpages}
\usepackage{amsmath,amsthm,amssymb,amscd}
\usepackage{setspace}
\usepackage{upgreek}
\usepackage[bookmarksopen,bookmarksdepth=3]{hyperref}
\usepackage{cite}
\usepackage{url}
\usepackage{mathtools}
\usepackage{fullpage}
\usepackage{float}
\usepackage{comment}
\usepackage{graphics}

\usepackage{cleveref}	

\usepackage{enumitem}

\input xy
\xyoption{all}

\usepackage{microtype}

\allowdisplaybreaks[1]

\newtheorem{thm}{Theorem}

\newtheorem{prop}{Proposition}[section]
\crefname{prop}{Proposition}{Propositions}		

\newtheorem{lm}[prop]{Lemma}
\crefname{lm}{Lemma}{Lemmas}				

\newtheorem{cor}[prop]{Corollary}

\theoremstyle{definition}
\newtheorem{dfn}[prop]{Definition}

\theoremstyle{remark}
\newtheorem{rem}[prop]{Remark}

\DeclareMathOperator{\im}{Im}

\DeclareMathOperator{\rdim}{rel\,dim}

\newcommand{\lrarr}{\longrightarrow}

\newcommand{\R}{\mathbb{R}}

\newcommand{\Z}{\mathbb{Z}}

\newcommand{\M}{\mathcal{M}}

\renewcommand{\P}{\mathbb{C}P}
\renewcommand{\L}{\Lambda}

\newcommand{\mI}{\mathcal{I}}

\newcommand{\Hh}{\widehat{H}}

\newcommand{\p}{\mathfrak{p}}
\newcommand{\pkl}{\p_{k,l}}

\newcommand{\q}{\mathfrak{q}}

\newcommand{\m}{\mathfrak{m}}
\renewcommand{\d}{\partial}

\newcommand{\mR}{\mathfrak{R}}

\newcommand{\at}{\tilde{\alpha}}

\newcommand{\etat}{\tilde{\eta}}

\newcommand{\Mt}{\widetilde{\M}}

\newcommand{\bt}{\tilde b}
\newcommand{\gt}{\tilde\gamma}

\newcommand{\qt}{\tilde\q}

\newcommand{\mC}{\mathfrak{C}}

\newcommand{\evbt}{\widetilde{evb}}
\newcommand{\evit}{\widetilde{evi}}
\newcommand{\evt}{\widetilde{ev}}

\newcommand{\mg}{\m^{\gamma}}

\renewcommand{\ll}{\langle\!\langle}
\renewcommand{\gg}{\rangle\!\rangle}

\newcommand{\RP}{\mathbb{R}P}

\newcommand{\sly}{\Pi}
\newcommand{\pr}{\varpi}

\newcommand{\lp}{{\prec}}

\renewcommand{\a}{\alpha}

\newcommand{\pg}{\p^{\gamma}}
\newcommand{\pbg}{\p^{b,\gamma}}

\newcommand{\qbg}{\q^{b,\gamma}}
\newcommand{\qg}{\q^{\gamma}}
\newcommand{\qtbg}{\qt^{\bt,\gt}}

\newcommand{\ptbg}{\pt^{\bt,\gt}}
\newcommand{\ptg}{\pt^{\gt}}

\newcommand{\V}{\Vert}
\newcommand{\A}{\mathcal{A}}

\newcommand{\lst}{\a}
\newcommand{\pt}{\tilde{\p}}
\newcommandx{\ptkl}[2][1=k,2=l]{\pt_{#1,#2}}
\newcommand{\lstt}{\tilde{\lst}}
\newcommand{\alt}{\a}
\newcommandx{\aslt}[1]{{\boldsymbol{\a}}_{(#1)}}

\newcommandx{\atslt}[1]{\boldsymbol{\at}_{(#1)}}
\newcommand{\glt}{\gamma}
\newcommandx{\gslt}[1]{\boldsymbol{\gamma}_{(#1)}}

\newcommandx{\gtslt}[1]{\boldsymbol{\gt}_{(#1)}}

\newcommand{\s}{\mathfrak{s}}

\newcommandx{\ngl}[1]{\langle #1 \rangle}
\newcommandx{\rgl}[1]{( #1 )}

\usepackage{xargs}              
\usepackage{xifthen}            
\usepackage{cancel}	 	

\let\originalleft\left
\let\originalright\right
\renewcommand{\left}{\mathopen{}\mathclose\bgroup\originalleft}
\renewcommand{\right}{\aftergroup\egroup\originalright}

\DeclareMathOperator{\pslsym}{PSL}
\newcommandx{\psl}[2][2=\mathbb{C}]{\pslsym_{#1}(#2)}
\DeclareMathOperator{\autsym}{Aut}
\newcommandx{\aut}[1]{\autsym(#1)}
\DeclareMathOperator{\glsym}{GL}
\newcommandx{\gl}[2][2=\mathbb{C}]{\glsym_{#1}(#2)}

\newcommand{\norm}[1]{\left|{#1}\right|}

\newcommandx{\fspace}[7][3=, 4=, 5=, 6=, 7={,}]{
        \ifthenelse{\isempty{#5}}
            {\ifthenelse{\isempty{#6}}
                {\ifthenelse{\isempty{#3}}                  
                    {{#1}_{#4}^{#2}}
                    {{#1}_{#4}^{#2 #7 #3}}}
                {\fspace{#1}{#2}[#3][#4][#6][]}}            
            {\ifthenelse{\isempty{#6}}
                {\ifthenelse{\isempty{#3}}                  
                        {{#1}_{#4}^{#2} (#5)}               
                        {{#1}_{#4}^{#2, #3} (#5)}}          
                {\ifthenelse{\isempty{#3}}                  
                        {{#1}_{#4}^{#2} (#5 #7 #6)}           
                        {{#1}_{#4}^{#2, #3} (#5 #7 #6)}}}     
}

\newcommandx{\slist}[4][2=, 3=, 4={,}]{
  \ifthenelse{\isempty{#2} \AND \isempty{#3}}
  	{#1}
	{\ifthenelse{\isempty{#3}}
		{#1#4 #2}
		{#1#4 #2#4 #3}}}

\newcommandx{\fspacei}[9][2=,3=,4=,5=,6=,7=,8=,9=]{
  \ifthenelse{\isempty{#7}}
  	{{#1}^{\slist{#2}[#3][][,]}_{\slist{#4}[#5][#6][,]}}
  	{{#1}^{\slist{#2}[#3][][,]}_{\slist{#4}[#5][#6][,]}
	\left( \slist{{\slist{{#7}}[{#8}][][,]}}[{#9}][][;] \right)}
}

\newcommandx{\moduli}[6][1=k,2=\ell,3=\beta,4=X,5=L,6=J]
{\fspacei{\mathcal{M}}[#3][][#1][#2][][#4][#5][#6]}

\newcommandx{\modulis}[3][1=k+1,2=\ell,3=\beta]
{\fspacei{\mathcal{M}}[#3][][#1][#2][][][][]}

\newcommandx{\bc}[7][1=i,2=k_1,3=\ell_1,4=\beta_1,5=k_2,6=\ell_2,7=\beta_2]
{\mathcal{B}_{#1}\left( \left. #2, #3, #4 \, \right| \, #5, #6, #7 \right)}

\newcommandx{\qkl}[3][1=k,2=l,3=\beta]{\mathfrak{q}_{#1,#2}^{#3}}

\newcommandx{\df}[4][2=,3=,4=*]{\fspacei{A}[#4][][][][][#1][#3][{#2}]}

\newcommandx{\cz}[2][2=]{\fspacei{C}[][][0][][][{#1}][{#2}]} 
\newcommandx{\ck}[3][3=]{\fspacei{C}[#1][][][][][{#2}][{#3}]} 
\newcommandx{\lps}[2]{\fspacei{\ell}[#1][][][][][{#2}]} 

\newcommandx{\Set}[2][2=]{
    \ifthenelse{\isempty{#2}}
        {\left\{ {#1} \right\}}
        {\left\{ {#1}  \, \middle| \, {#2} \right\}}
}

\newcommandx{\degree}[1][1=]{\operatorname{deg} #1}

\newcommandx{\homo}[4][3=, 4=\mathbb{Z}]{\fspacei{H}[][][#1][][][#2][#3][#4]}

\newcommandx{\inner}[2][1={\cdot}, 2={\cdot}]{\left<{#1},{#2}\right>}

\newcommandx{\nnorm}[1][1=\cdot]{\left|\left|{#1}\right|\right|}
\newcommandx{\matring}[2][2=\mathbb{R}]{M_{#1}(#2)}

\newcommandx{\complexst}[2][2=]{\fspace{\mathcal{J}}{}[][][#1][#2]}
\newcommandx{\Hom}[4][3=,4=]{\fspace{\mathrm{Hom}}{#3}[][{#4}][{#1}][{#2}]}
\newcommandx{\complexify}[1]{{\fspace{{#1}}{}[][\mathbb{C}][][]}}

\newcommandx{\realpart}[2][2=]{\fspace{{#1}}{{#2}}[][\mathbb{R}][][]}
\newcommandx{\Grass}[3]{\fspace{\mathrm{Gr}}{}[][#1][#2][#3]}
\newcommandx{\TotallyReal}[2]{\fspace{\mathrm{TR}}{}[][][#1][#2]}
\newcommandx{\GLVec}[2]{\fspace{\glsym}{}[][#2][#1][]}

\newcommandx{\Ainf}{A_{\infty}}
\newcommandx{\RG}{\mathbb{R}_{\geq 0}}

\newcommandx{\intr}[1]{#1^{\circ}}

\newcommandx{\hilb}[1][1=]{\fspacei{\mathcal{H}}[][][][][][{#1}][]}

\newcommandx{\supnorm}[2][1=\cdot,2=]{\nnorm[#1]_{\infty,#2}}
\newcommandx{\lpnorm}[3]{\nnorm[#1]_{\lp{#2}{#3}}}

\newcommand{\ZZ}{\mathbb{Z}}      
\newcommand{\NN}{\mathbb{N}}      

\newcommandx{\cp}[1][1=n]{\mathbb{CP}^{#1}}	

\newcommandx{\Nov}[1][1=]{\fspace{\Lambda}{}[][0,\mathrm{nov}][{#1}][][]}

\newcommandx{\comoduliev}[5][1=k+1, 2=l, 3=X, 4=L, 5=\beta]{\comoduli[#1][#2][#3][#4][#5]}

\newcommandx{\coker}[1]{\operatorname{coker} \left( #1 \right)}
\newcommandx{\hsnorm}[1]{\nnorm[#1]_{\text{HS}}}
\newcommandx{\trnorm}[1]{\nnorm[#1]_{\text{tr}}}

\newcommand{\catname}[1]{{\normalfont\textbf{#1}}}
\newcommandx{\CVec}[1]{\fspace{\catname{Vec}}{}[][#1][][]}

\newcommandx{\evb}[4][1=j,2=k,3=\ell,4=\beta]{\mathrm{evb}_{#1}^{#2,#3,#4}}
\newcommandx{\evbs}[2][1=j,2=\beta]{\mathrm{evb}_{#1}^{#2}}
\newcommandx{\evi}[4][1=j,2=k,3=\ell,4=\beta]{\mathrm{evi}_{#1}^{#2,#3,#4}}
\newcommandx{\evis}[2][1=j,2=\beta]{\mathrm{evi}_{#1}^{#2}}

\newcommandx{\sign}[2][1=\alpha,2=\gamma]{\varepsilon(#1;#2)}

\newcommandx{\fbp}[2][1={\evbs[i][\beta_1]},2={\evbs[0][\beta_2]}]{\tensor*[_{#1}]{\times}{_{#2}}}

\newcommandx{\mc}[1]{\mathcal{#1}}

\newcommandx{\conn}[4][1=j,2=k,3=\ell,4=\beta]{\nabla_{#1}^{#2,#3,#4}}

\newcommandx{\tlb}[4][1=j,2=k,3=\ell,4=\beta]{\mc{L}_{#1}^{#2,#3,#4}}

\newcommandx{\xrightarrowdbl}[2][1=,2=]{%
  \xrightarrow[#2]{#1}\mathrel{\mkern-14mu}\rightarrow
}

\newcommandx{\tens}[1]{T \left( #1 \right)}		
\newcommandx{\tensr}[1]{\overline{T} \left( #1 \right)} 
\newcommand{\ootimes}{\overline{\otimes}}		

\newcommandx{\degr}[1]{\left| {#1} \right|}		
\newcommandx{\degb}[1]{| {#1} |}			
\newcommandx{\degs}[1]{\Vert {#1} \Vert}		

\newcommandx{\gen}[1]{\left< {#1} \right>}

\newcommandx{\eqcl}[1]{\left[ {#1} \right]}		

\newcommandx{\choch}[2][1,2=*]{{CH}_{#2} \left( #1 \right)}
\newcommandx{\hhoch}[2][1,2=*]{HH_{#2} \left( #1 \right)}
\newcommandx{\dhoch}[1][1=]{\fspacei{\partial_{\textrm{hoch}}}[][][][][][#1][][]}
\newcommandx{\chochnorm}[2][1,2=*]{\widetilde{CH}_{#2} \left( #1 \right)}
\newcommandx{\cconnes}[2][1,2=*]{{C}_{#2}^{\lambda} \left( #1 \right)}
\newcommandx{\hcyc}[2][1,2=*]{HC_{#2} \left( #1 \right)}
\renewcommandx{\t}[1][1=]{\fspacei{\tau}[][][][][][#1][][]}
\newcommandx{\cyccl}[1]{\left[ {#1} \right]}
\newcommandx{\ceconnes}[2][1,2=*]{{C}_{#2}^{\lambda,+} \left( #1 \right)}
\newcommandx{\hecyc}[2][1,2=*]{HC_{#2}^{+} \left( #1 \right)}
\newcommandx{\crconnes}[2][1,2=*]{\overline{C}_{#2}^{\lambda} \left( #1 \right)}
\newcommandx{\hrconnes}[2][1,2=*]{\overline{HC}_{#2} \left( #1 \right)}
\newcommandx{\cerconnes}[2][1,2=*]{\overline{C}_{#2}^{\lambda,+} \left( #1
\right)}
\newcommandx{\herconnes}[2][1,2=*]{\overline{HC}_{#2}^{+} \left( #1 \right)}

\newcommandx{\cdiff}[4][2=*,3=,4=,]{\fspacei{A}[#2][][][][][#1][#3][#4]}
\newcommandx{\hderham}[3][3=]{\fspacei{H}[#1][][\textrm{dR}][][][#2][][#3]}
\newcommandx{\ccur}[4][2=*,3=,4=]{\fspacei{\mathcal{A}}[#2][][#3][][][#1][][#4]}
\newcommandx{\hcur}[4][2=*,3=,4=]{\fspacei{\mathcal{H}}[#2][][#3][][][#1][][#4]}

\newcommandx{\dfzero}[4][4=]{\fspacei{A}[#1][][0][][][#2][#3][#4]}
\newcommandx{\hdfzero}[4][4=]{\fspacei{H}[#1][][0][][][#2][#3][#4]}

\title{Differential forms, open-closed maps, and Gromov-Witten axioms}

\author[P. Giterman]{Pavel Giterman}
\address{Institute of Mathematics\\ Hebrew University, Givat Ram\\Jerusalem, 9190401, Israel } \email{pavel.giterman@mail.huji.ac.il}
\author[J. Solomon]{Jake P. Solomon}
\address{Institute of Mathematics\\ Hebrew University, Givat Ram\\Jerusalem, 9190401, Israel } \email{jake@math.huji.ac.il}
\author[S. Tukachinsky]{Sara B. Tukachinsky}
\address{School of Mathematical Sciences\\ Tel Aviv University\\Tel Aviv, 6997801, Israel }\email{sarabt1@gmail.com}

\date{July 2026}

\begin{document}

\begin{abstract}
We construct open-closed maps on various versions of Hochschild and cyclic homology of the Fukaya $A_\infty$ algebra of a Lagrangian submanifold modeled on differential forms. The $A_\infty$ algebra may be curved. Properties analogous to Gromov-Witten axioms are verified.
The paper is written with applications in mind to gravitational descendants and obstruction theory.
\end{abstract}

\maketitle

\setcounter{tocdepth}{2}
\tableofcontents

\section{Introduction}

Open-closed maps~\cite{Abouzaid,FOOO,Seidel_ICM,Seidel_2009} take data that encodes information about Lagrangian submanifolds to data about the ambient symplectic manifold. For Liouville domains, this means that the Hochschild homology of the wrapped Fukaya category is mapped to the symplectic homology of the domain; for closed symplectic manifolds, the Hochschild homology of the Fukaya category is mapped to the quantum cohomology of the manifold.
Such maps can be used to derive information about the appropriate Fukaya category, e.g., verifying its non-triviality~\cite{RitterSmith} or giving a split-generation criterion~\cite{Abouzaid, RitterSmith, Sheridan, Sheridan25,AbouzaidFukayaOhOhtaOno}. They can also be used to derive information about the quantum/symplectic cohomology, e.g., to verify non-vanishing~\cite{RitterSmith} or study nilpotency of $c_1$~\cite{RitterSmith}.

The open-closed map is typically defined on chain level, on the Hochschild chain complex, and is then shown to descend to other related chain complexes --- most notably, the cyclic complex.
This cyclic version has further applications. For example, it is used to study obstruction classes in Lagrangian Floer theory~\cite{FOOO}, to verify that homological mirror symmetry implies enumerative mirror symmetry~\cite{GanatraSheridan}, to verify the isomorphism of the TE structures on Lagrangian Floer vs. quantum cohomology~\cite{Hugtenburg1}, and to find a cyclic $A_\infty$ minimal model of the Fukaya category~\cite{Ganatra2}.

The open-closed map can be restricted to sub-categories, and can further be specialized to the ``baby'' case where we fix one Lagrangian submanifold and take the Fukaya $A_\infty$ algebra associated to it. This last specialized version is the object of this paper. Specifically, we carry out the construction of an open-closed map via differential forms and currents.
The closest to our formalism in the existing literature is~\cite{Hugtenburg1}; the main difference is that, in order to circumvent the need for currents, the open-closed map in~\cite{Hugtenburg1} is defined by dualizing the closed-open map rather than by an explicit formula.

In our construction, the cyclic complex comes up naturally, as the operations that define the open-closed map have an obvious cyclic symmetry (Proposition~\ref{prop:p_sgn}).
Thus, we get chain-level open-closed maps, from various versions of the Hochschild and cyclic homology of the Fukaya $A_\infty$ algebra of a Lagrangian submanifold to the quantum cohomology.
For a curved $A_\infty$ algebra, the curvature term defines a cocycle in the cyclic complex. To make it exact,
we offer a new version of cyclic homology, the \textit{extended} cyclic homology.
We verify that the open-closed map descends to this homology as well.
A detailed discussion of the various types of homologies is found in Section~\ref{sec:Hochschild complexes}.

We proceed to verify that, on chain level, our open-closed maps satisfy properties analogous to the Gromov-Witten axioms. We conclude with a discussion of pseudo-isotopies of open-closed maps, which we include for the sake of future applications.

The machinery developed here is intended to be used to define open gravitational descendants and establish their properties in a future work~\cite{GitermanSolomon}.
It also has applications to the study of the space of bounding cochains, which in turn is useful in defining genus zero open Gromov-Witten invariants~\cite{T19}.

\subsection{Setting}\label{ssec:setting}

Consider a closed symplectic manifold $(X,\omega)$ with $\dim_{\R}X=2n$, and a connected closed Lagrangian submanifold $L$ with relative spin structure~$\s =\s_L$~\cite{FOOO,WehrheimWoodward}. Let $J$ be an $\omega$-tame~\cite{MS0} almost complex structure on $X$. Denote by $\mu:H_2(X,L;\Z) \to \Z$ the Maslov index~\cite{CieliebakGoldstein}.
Let $\sly$ be a quotient of $H_2(X,L;\Z)$ by a possibly trivial subgroup contained in the kernel of the homomorphism $\omega \oplus \mu : H_2(X,L;\Z) \to \R \oplus \Z.$ Thus the homomorphisms $\omega,\mu,$ descend to $\sly.$ Denote by $\beta_0$ the zero element of $\sly.$
We use the Novikov ring $\L$ which is a completion of a subring of the group ring of $\sly$.
The precise definition follows. Denote by $T^\beta$ the element of the group ring corresponding to $\beta \in \sly$, so $T^{\beta_1}T^{\beta_2} = T^{\beta_1 + \beta_2}.$
Then,
\[
\L=\left\{\sum_{i=0}^\infty a_iT^{\beta_i}\bigg|a_i\in\R,\beta_i\in \sly,\omega(\beta_i)\ge 0,\; \lim_{i\to \infty}\omega(\beta_i)=\infty\right\}.
\]
A grading is defined on $\L$ by declaring $T^\beta$ to be of degree $\mu(\beta).$ In particular, since the relative spin structure $\s$ on $L$ includes orientation, $T^\beta$ is of even degree for any $\beta$.

For $k,l\ge  0,$ denote by $\M_{k,l+1}(\beta)$ the moduli space of genus zero $J$-holomorphic open stable maps to $(X,L)$ of degree $\beta \in \sly$ with one boundary component, $k$ boundary marked points, and $l+1$ interior marked points. The boundary points are labeled according to their cyclic order. Denote by
$evb_i^\beta:\M_{k,l+1}(\beta)\to L,$ and
$evi_j^\beta:\M_{k,l+1}(\beta)\to X,$ the boundary and interior evaluation maps respectively, where $ i=1,\ldots,k,$ and $j=0,\ldots, l$.
Assume that $\M_{k,l+1}(\beta)$ is a smooth orbifold with corners. Then it carries a natural orientation induced by the relative spin structure on $(X,L)$, as in~\cite[Chapter 8]{FOOO}.
Examples of $(X,L)$ that satisfy our assumptions include $(\P^n, \RP^n)$ with the standard complex and symplectic structures or, more generally, flag varieties, Grassmannians, and various other homogeneous spaces, as detailed in~\cite[Example 1.5, Remark 1.6]{ST1}.

For a compact orbifold $M$, possibly with corners, denote by $A^*(M)$ the algebra of smooth differential forms on $M$ with coefficients in $\R$.
Denote by $\ccur{M}[*][][]$ the dual module of currents on $M$ with coefficients in $\R$, equipped with the cohomological grading, so the inclusion $A^*(M) \to \ccur{M}[*][][]$ preserves degree. Denote by $\hcur{M}[*][][]$ the cohomology of $\ccur{M}[*][][],$ which by~\cite[Theorem 14]{deRham} is isomorphic to the de Rham cohomology $H^*(M)$ of differential forms. On chain level, we need to use currents in order to push forward differential forms along maps that are not submersions, as in equation~\eqref{eq:pdef_intro} below. Working with currents as opposed to forms does not add much technical difficulty once their basic properties are established. Indeed, we avoid pulling back currents or taking products of currents, which can be delicate. See Section~\ref{ssec:currents} for background on currents on orbifolds.

Let $t_0,\ldots,t_N$ be formal variables with degrees in $\Z$.
Define the graded-commutative ring
\begin{equation}\label{eq:R}
R:=\L[[t_0,\ldots,t_N]],
\end{equation}
thought of as a differential graded algebra with trivial differential.

Define a valuation
$
\nu:R\lrarr \R
$
by
\[
\nu\left(\sum_{j=0}^\infty a_jT^{\beta_j}\prod_{i=0}^Nt_i^{l_{ij}} \right)
= \inf_{\substack{j\\a_j\ne 0}} \left(\omega(\beta_j)+\sum_{i=0}^N l_{ij}\right).
\]
Denote the positive valuation ideal in $R$ by $\mI_R: = \{\alpha\in R\,|\,\nu(\alpha)>0\}$.

Abbreviate $A^*(M;R)$ and $\ccur{M}[*][][R]$ for $A^*(M)\otimes R$ and $\ccur{M}[*][][]\otimes R$, respectively, where $\otimes$ is understood as the completed tensor product of differential graded-commutative algebras and modules respectively.
The gradings on $A^*(M;R)$ and $\A^*(M;R)$ take into account the degrees of $t_j, T^\beta,$ and the degree of differential forms and currents.
We denote the degree of an element $a$ of $A^*(M;R)$ or $\A^*(M;R)$ by $|a|$.
The valuation induced by $\nu$ on $A^*(M;R)$ and $\A^*(M;R)$ will still be denoted by $\nu$.
From now on, whenever tensor products and direct sums are taken in the text, we implicitly complete them with respect to $\nu$.

For the following particularly important special case, we use the notation
\[
C:= A^*(L;R).
\]
The $R$-algebra $C$ is endowed with a family of \textbf{$A_\infty$ structures} $\mg$, parameterized by certain elements $\gamma\in A^*(X;R)$. The operators $\mg$ are defined using moduli spaces of $J$-holomorphic stable disk maps, where $\gamma$ represents constraints imposed on interior marked points on the disk. See Section~\ref{ssec:q} for details. For a fixed $\gamma$, we say that $(C,\mg)$ is an \textbf{$A_\infty$-algebra}.
Furthermore, the constant function $1_L\in A^0(L)\subset C$ is a (strong) unit of the $A_\infty$ structure.

\subsection{Statement of results}

Fix $\gamma \in A^*(X;R)$ with $|\gamma|=2$, $d\gamma=0$, and $\nu(\gamma)>0$.
For $\beta\in \sly$ and $k,l\ge 0$, consider $evb_j^\beta:\M_{k,l+1}(\beta)\to L$ and
$evi_j^\beta:\M_{k,l+1}(\beta)\to X$ and
define
\begin{equation}\label{eq:p_intro}
\pg_k,\p^{\gamma,\beta}_k:C^{\otimes k}\lrarr \ccur{X}[*][][R]
\end{equation}
by
\begin{equation}\label{eq:pdef_intro}
\p^{\gamma,\beta}_k(\alpha_1,\ldots,\alpha_k):= (-1)^{\sum_{j=1}^k(n+j)(|\a_j|+1)}
\sum_{l\ge0}\frac{1}{l!}{evi_0^\beta}_* (\bigwedge_{j=1}^l(evi_j^\beta)^*\gamma\wedge \bigwedge_{j=1}^k (evb_j^\beta)^*\alpha_j).
\end{equation}
Note that $\pbg$ is defined using push-forward of a differential form on an orbifold with corners; the meaning of such operations is discussed thoroughly in~\cite{ST4} and is reviewed in Section~\ref{ssec:currents}.

Set
\begin{equation}\label{eq:pkg}
\pg_k :=\sum_{\beta\in\sly}T^\beta\p^{\gamma,\beta}_k.
\end{equation}
The same $\gamma$ can be used to define an $A_\infty$ structure $\mg$ on $C$.
We show that the operators $\pg_k$ induce maps from various versions of the Hochschild and cyclic homologies of $(C, \mg)$ into the appropriate versions of quantum cohomology, as detailed below.
The differentials of the various Hochschild and cyclic complexes depend on
$\gamma$ via the algebra structure $\mg$, but as we have fixed $\gamma$,
we omit it from our notation for the complexes and their homologies.

The \textbf{Hochschild chain complex} of $C$ is the shifted, reduced (i.e., starting with $j=1$) tensor algebra of $C$, so
\[
\choch[C] := \bigoplus_{j=1}^\infty C[1]^{\otimes j},
\]
equipped with the coboundary operator $\dhoch$, which is defined in terms of the $A_\infty$ structure on $C$ and increases degree by one. The fact that  $\partial^2_{\textrm{hoch}}=0$ is a direct consequence of the $A_\infty$ relations. For full details see Section~\ref{ssec:hoch}. The \textbf{Hochschild homology} is the cohomology of the Hochschild chain complex
\[
\hhoch[C] : = H^*(\choch[C][\bullet],\dhoch).
\]
Set
\begin{equation*}
  \p^{\gamma} = \sum_{k = 1}^{\infty} \p_k^{\gamma} \colon \choch[C]
  \rightarrow  \ccur{X}[*+n+1][][R].
\end{equation*}
The degree shift above is justified in Proposition~\ref{prop:p_deg}.

\begin{thm} \label{thm:pg-hoch}
  The map $\pg$ is a chain map of degree $n + 1$, i.e., it satisfies
  \begin{equation}
    d \circ \pg-(-1)^{n+1}\pg \circ \dhoch=0.
  \end{equation}
  Thus, it induces a map $\pg \colon \hhoch[C] \rightarrow \hcur{X}[* + n +
  1][][R]$
  between the Hochschild homology of $C$ and the de Rham currents cohomology of $X$.
  Alternatively, by de Rham's theorem, we may think of
$\pg$ as a map $\pg \colon \hhoch[C] \rightarrow H^{* + n +  1}(X;R)$.
\end{thm}

The \textbf{normalized Hochschild} chain complex $\chochnorm[C]$ is constructed
from $\choch[C]$ by modding out by those elements $\a_1\otimes\cdots\otimes\a_n$ such that $\a_j=1_L$ for some $j\ge 2$. See Section~\ref{sssec:normNred} for details.

\begin{thm}\label{thm:pg-normhoch}
  The map $\pg$ descends to the normalized Hochschild complex and gives
  us a chain map $\pg \colon \chochnorm[{C}] \rightarrow \ccur{X}[* + n + 1][][R]$.
\end{thm}

We work with a version of cyclic homology based on Connes' construction. The \textbf{Connes chain complex} $\cconnes[{C}]$ is obtained from $\choch[C]$ by identifying pure tensors that agree after a cyclic permutation.
Intuitively, we can think of elements of $\choch[C]$ as elements of $C$ arranged in a list, while elements of $\cconnes[{C}]$ are elements of $C$ arranged in a circle. The \textbf{cyclic homology} $\hcyc[C]$ is the cohomology of $\cconnes[{C}].$
See Section~\ref{sssec:connes} for details.

\begin{thm}\label{thm:pg-connes}
  The map $\pg$ descends to the Connes cyclic complex and gives us a chain map $\pg
  \colon \cconnes[{C}] \rightarrow \A^{* + n + 1}(X;R)$. Thus, it induces a map
  $\pg \colon \hcyc[C] \rightarrow \hcur{X}[* + n + 1][][R]$.
\end{thm}

The \textbf{reduced} Connes complex $\crconnes[C]$, analogously to the normalized Hochschild complex, is obtained by modding out by those elements $\a_1\otimes\cdots\otimes\a_n$ of $\cconnes[C]$ such that $\a_j=1_L$ for some $j\ge 1$.
See Section~\ref{sssec:normNred} for details. Denote the cohomology of $\crconnes[C]$ by $\hrconnes[C]$.
The element $1_L$ (in the case $n=j=1$) is not mapped to zero by $\pg$, but rather to the current of integration on $L$. Therefore, in order to descend $\pg$ to $\hrconnes[C]$, we have to quotient the codomain by the subspace spanned by that current. We denote the resulting chain complex by $\ccur{X}[*][L][]$, and the resulting cohomology
by $\hcur{X}[*][L][R]$. See Section~\ref{ssec:currents} for full details. For example, in the case when $L$ is a rational homology sphere, $\hcur{X}[*][L][R]$ is isomorphic to $H_*(X\setminus L;R)$ for $* \neq n.$

\begin{thm}\label{thm:pg-redconnes}
  The map $\pg$ descends to a chain map $\pg \colon \crconnes[C] \rightarrow
  \ccur{X}[* + n + 1][L][R]$.
  Thus, it induces a map
  \begin{equation*}
    \pg \colon \hrconnes[C] \rightarrow \hcur{X}[* + n + 1][L][R]
  \end{equation*}
  between the reduced cyclic homology of $C$ and $\hcur{X}[*][L][R]$.
\end{thm}

The \textbf{extended cyclic} complex $\ceconnes[C]$ is obtained from the above-mentioned $\cconnes[{C}]$ by adding a generator that corresponds to the empty list. This is equivalent to taking the underlying chain complex to be the \textit{full} tensor algebra $\tens{C[1]}:=\bigoplus_{j=0}^\infty C[1]^{\otimes j}$ quotiented by the cyclic action as before. We denote the new generator by $1$ (as opposed to $1_L$, which was the unit in $C$). Such an extension would not work on $\choch[C]$, because the natural choice for $\dhoch[1]$ is not closed in $\choch[C]$, but it is in the cyclic complex.
See Section~\ref{sssec:extend} for a detailed discussion.
Here again we have a reduced version: The \textbf{extended and reduced} complex $\cerconnes[C]$ is obtained by modding out by elements of $\ceconnes[C]$ that have the unit in one of their components. See Section~\ref{sssec:normNred}.

As discussed, the new generator added to form the extended complex corresponds to the empty list. Thus, we need to extend the $\pg$ operator by adding the term
\[
\pg_0 : R\to \ccur{X}[n + 1][L][R]
\]
defined above~\eqref{eq:pdef_intro} by pushing forward along the evaluation maps $evi_0:\M_{0,l+1}(\beta)\to X$. Extended thusly, $\pg$ is well defined on the extended and extended reduced complexes. However, it does not immediately give a chain map.
The geometric reason for the violation of the chain map property is the degeneration of disks where the boundary collapses to a point, which is now possible since no marked points appear on the boundary of the stable maps involved in defining $\pg_0$.

The resulting extra contribution can be balanced out if the Lagrangian $L$ is homologically trivial inside the ambient manifold $X$. In this case, the current $\zeta_L=i_*1_L\in \A^*(X;R)$ is exact, where $i:L\hookrightarrow X$ is the inclusion. Specifically, the choice of a singular chain $S$ such that $\d S= -L$ corresponds to a current $\eta$ such that $d\eta = -\zeta_L$. See Section~\ref{ssec:currents}.
Then we can balance out the contribution of disks with potentially contractible boundary by adding contributions from spheres that pass through $S$. A similar idea appears in~\cite{PandharipandeSolomonWalcher,Joyce,ST3}. Concretely, we define an operator
\[
\qg_{\emptyset,1}: \A^*(X;R)\to \A^*(X;R)
\]
analogous to $\pg$ but defined using spaces of $J$-holomorphic \emph{spheres} rather than disks, see Section~\ref{ssec:qemptyset}.
Then we have the following.

\begin{thm}\label{thm:pg-extended}
Assume that $i_{*} ([L]) = 0$ in $H_n(X;\R)$.
Choose
a current $\eta$ with $d\eta=-\zeta_L$. Consider the operator
  $\mathcal{P}_{\eta} = \mathcal{P} \colon \tens{C[1]}^{*} \rightarrow
  \ccur{X}[*+n+1][][R]$ defined by
  \begin{align*}
    \mathcal{P} \left( \alpha_1, \dots, \alpha_k \right) &= \pg \left( \alpha_1,
    \dots, \alpha_k \right), \,\,\, k \geq 1 \\
    \mathcal{P} (1) &= \pg_0(1) + \q_{\emptyset,1}^{\gamma} \left( \eta
    \right).
  \end{align*}
  The map $\mathcal{P}$ is symmetric with respect to the cyclic group action and
  so it descends to a chain map $\mathcal{P} \colon \ceconnes[C] \rightarrow
  \ccur{X}[*+n+1][][R]$. Thus, it induces a map
  \[
  \mathcal{P} \colon \hecyc[C] \rightarrow
  \hcur{X}[*+n+1][][R]
  \]
  between the extended cyclic homology of $C$ and the
  cohomology of $X$. In addition, $\mathcal{P}$ descends to a map $\mathcal{P}
  \colon \cerconnes[C] \rightarrow \ccur{X}[*+n+1][L][R]$ and thus induces a map
  \[
  \mathcal{P} \colon \herconnes[C] \rightarrow \hcur{X}[*+n+1][L][R]
  \]
  between the extended and reduced cyclic homology $\herconnes[C]$ and
  $\hcur{X}[*+n+1][L][R]$. The map $\mathcal{P}_{\eta}$ depends on the choice
  of $\eta$ but the induced maps on homology depend only on the homology class
  $\eqcl{\eta}$ inside $\hcur{X}[n-1][L][R]$.
\end{thm}

The map $\mathcal{P}_\eta$ depends on the choice of $\gamma$ as well. In fact, even the domain of $\mathcal{P}_\eta$ depends on $\gamma$. We expect that the $\pt$ maps for pseudoisotopies treated in Section~\ref{sec:isot} can be used to show that cohomologous choices of $\gamma$ give quasi-isomorphic chain complexes in the domain, and that the induced $\mathcal{P}$ agree on homology. Similarly, the domain chain complex should be independent of $J$ up to quasi-isomorphism, and the induced $\mathcal{P}$ on homology should be independent of $J.$ See Section~\ref{ssec:ocps} for further discussion.

The property of inducing open-closed maps on Hochschild homology and the subsequent variants is a consequence of $\pg$ satisfying the structure equations given in Proposition~\ref{prop:pstructure}. These equations are similar in flavor to the structure equations of an $A_\infty$-algebra, in that they describe codimension-1 behavior of the moduli spaces involved.
We further show that, again similarly to the $A_\infty$ operations (cf.~\cite[Theorem 3]{ST1}), the maps $\pg$ satisfy properties reminiscent of the Gromov-Witten axioms:
\begin{thm}\label{thm:pgprop}
Suppose $\d_{t_0}\gamma=1\in A^0(X)\otimes R$ and $\d_{t_1}\gamma=\gamma_1\in A^2(X)\otimes R$.
Then the operations $\pg_k$ satisfy the following properties.
\begin{enumerate}
	\item\label{pprop1} (Fundamental class)
		$\d_{t_0}\pg_k=0$.
	\item\label{pprop2} (Divisor)
		$\d_{t_1}\p^{\gamma,\beta}_k=\int_\beta\gamma_1 \cdot\p^{\gamma,\beta}_k$, assuming $i^*\gamma_1=0$, where $i:L\hookrightarrow X$ is the inclusion.
	\item\label{pprop3} (Energy zero)
		The operations $\pg_k$ are deformations of the usual push-forward of differential forms in the sense that
		\[
		\p^{\gamma,\beta_0}_1(\alpha)= (-1)^{(n+1)(|\a|+1)}i_*\a,
\qquad
    \p^{\gamma,\beta_0}_k=0, \quad k \neq 1.
		\]
\end{enumerate}
\end{thm}
\Cref{thm:pgprop}~\eqref{pprop3} is used in the construction of open Gromov-Witten invariants in~\cite{T19}. We expect that \Cref{thm:pgprop}~\eqref{pprop1} and~\eqref{pprop2} will be useful in proving the fundamental class and divisor axioms of those invariants.

Additionally, in Section~\ref{sec:isot} we discuss a version of $\p$ operators defined on pseudoisotopies. Pseudoisotopies arise, for example, from varying the underlying data like $\gamma$ or $J$.
The discussion is carried out under regularity assumptions on the family moduli spaces similar to those already assumed for $\M_{k,l+1}(\beta).$

\subsubsection{Regularity assumptions}\label{ssec:reg}

As mentioned, we proceed with the regularity assumptions set in~\cite{ST1}, namely, that moduli spaces are smooth orbifolds with corners and the evaluation maps $evb_0$, $ev_0$, when defined, are proper submersions.
In~\cite[Example 4.1-Remark 1.5]{ST1} we show that the regularity assumptions hold
for homogeneous spaces.
In particular, $(\P^n,\RP^n)$ with the standard symplectic and complex structures, or more generally, Grassmannians, flag varieties and products thereof, satisfy our regularity assumptions.
Using the theory of the virtual fundamental class from ~\cite{Fukaya,Fukaya2,FOOOtoricI,FOOOtoricII,FOOO1},~\cite{HoferWysockiZehnder,HoferWysockiZehnder1,HoferWysockiZehnder2,HoferWysockiZehnder3,LiWehrheim}, or~\cite{AbouzaidMcLeanSmith,AbouzaidMcLeanSmith2,HirschiHugtenburg}, our results are expected to extend to general target manifolds.

\subsection{Outline of the paper}

In Section~\ref{sec:Hochschild complexes} we give a construction of the various versions of Hochschild and cyclic homology of an arbitrary curved $A_\infty$-algebra.
In Section~\ref{sec:background} we establish notation that will be used throughout the subsequent text and cite previously proven results. Notably, we cite the construction and properties of closed operators (operators modeled on spaces of stable sphere maps) and closed-open operators (operators modeled on spaces of stable disk maps, with an output at the boundary).
In Section~\ref{sec:p} we construct the $\p$ operators and prove their basic properties, in the model of differential forms and currents.
In Section~\ref{sec:OChoch}, we take the geometric realization of the homologies defined in Section~\ref{sec:Hochschild complexes} that comes from the Fukaya $A_\infty$-algebra of a Lagrangian submanifold, and verify that the $\p$ operators from Section~\ref{sec:p} descend to maps on those homologies.
Finally, in Section~\ref{sec:isot} we verify properties for a version of the $\p$ operators defined on pseudoisotopies.

\subsection{Acknowledgements}
\leavevmode

P.~G., J.~S., and S.~T., were partially supported by ERC starting grant~337560. P.~G. and J.~S. were partially supported by ISF grant~569/18. J.~S. was partially supported by ISF grant~1127/22 and the Miriam and Julius Vinik Chair in Mathematics.
S.~T. was partially supported by NSF grant DMS-1638352, ISF grant~2793/21, and the Colton Foundation.

\section{Hochschild and cyclic homologies of a curved \texorpdfstring{$A_\infty$}{A-infty} algebra}
\label{sec:Hochschild complexes}

In this section, we give a general, algebraic construction of the types of Hochschild and cyclic homologies of a curved $A_\infty$ algebra $(A,\mu)$ that we need. Subsequently, in Section~\ref{sec:OChoch}, we apply these constructions to the particular case of the $A_\infty$-algebra $(C,\mg)$ defined in the introduction.
When $A$ corresponds to a unital DGA, our construction of the Hochschild, cyclic, and reduced cyclic complexes given in Sections~\ref{ssec:hoch}, \ref{sssec:connes}, and~\ref{sssec:normNred}, respectively, are isomorphic to the standard complexes described for example in~\cite{Loday}.
We refer to
\cite[Appendix C]{GitermanSolomon2} for more details.
The construction of the extended cyclic homology in Section~\ref{sssec:extend} is new, and is designed to fix the issues that arise from non-zero curvature, as occurs in our application.

\subsection{Notation}\label{ssec:notlist}
In what follows, we will always work in the category of $\mathbb{Z}$-graded algebras
and modules. Fix a field $k$ of characteristic zero and let $R$ be
a graded-commutative $k$-algebra. Let $\nu$ be a valuation on $R$ with respect to which $R$ is complete.

Let $A$ be a $\mathbb{Z}$-graded left $R$-module endowed with a valuation also denoted by $\nu$, with respect to which $A$ is complete. All undecorated tensor products will be
taken over $R$. We will denote by
\begin{equation*}
  \tens{A} := \bigoplus_{i=0}^{\infty} A^{\otimes i}, \qquad \tensr{A} :=
  \bigoplus_{i=1}^{\infty} A^{\otimes i}
\end{equation*}
the tensor algebra and the reduced tensor algebra, respectively.
The tensors products and direct sums above are implicitly completed with respect to $\nu$. Both algebras carry two natural gradings, the \textbf{weight grading} and the \textbf{degree grading}. An elementary tensor $l = a_1 \otimes \dots \otimes a_k$ has weight $k$ and degree $\degb{a_1}
+ \dots + \degb{a_k}$. We will often think informally of such an elementary
tensor $l$ as representing a list $(a_1,\dots,a_k)$ of $k$ elements from $A$.
When $k = 0$, the empty tensor product
$a_1 \otimes \dots \otimes a_k$ is to be interpreted as $1_R \in \tens{A}$ and
thought of as the empty list. In what follows, we will often need to start
with $l$, split it into several consecutive lists and apply operations to
certain parts of the splitting. To do that, it will be convenient to introduce
the following notation, taken from and used extensively in \cite{GitermanSolomon2}:
\begin{enumerate}
  \item \textbf{Splitting}. Given $r \in \NN$, we denote by $l_{(1)} \ootimes
    l_{(2)} \ootimes \cdots \ootimes l_{(r)}$ the element in $\tens{A}^{\ootimes
    r}$ which is the \textbf{sum} of all possible splittings of $l$ into $r$
    consecutive, possibly empty, lists. For example, if $l = a_1 \otimes a_2
    \otimes a_3$ and $r = 2$ then
    \begin{equation*}
      l_{(1)} \ootimes l_{(2)} = 1 \ootimes (a_1 \otimes a_2) + a_1 \ootimes a_2 +
      (a_1 \otimes a_2) \ootimes 1
    \end{equation*}
    where we use the symbol $\ootimes$ to denote the ``external'' tensor product
    so that we won't confuse it with the internal tensor product appearing in the definition of $\tens{A}$.
    We will also need to
    iterate this construction so we will write expressions such as
    \begin{equation}\label{eq:sweedler}
      l_{(1)} \ootimes l_{(2)} \ootimes l_{(3)} = l_{(11)} \ootimes l_{(12)} \ootimes
      l_{(2)} = l_{(1)} \ootimes l_{(21)} \ootimes l_{(22)}
    \end{equation}
    which are equal.
        \begin{rem}
    The map which sends $l$ to $l_{(1)} \ootimes \dots \ootimes
    l_{(r)}$ is precisely the iterated deconcatenation coproduct
    $\Delta^{r-1}$ on $\tens{A}$, see~\cite[Section 1.2]{LodayVallette}. This notation
    is sometimes called \textbf{Sweedler's notation}. The equality in~\eqref{eq:sweedler} expresses the coassociativity of $\Delta$.
    \end{rem}
  \item \textbf{Application}. Given a map $\psi \colon \tens{A} \rightarrow A$,
    we denote by
    \begin{equation*}
      l_{(1)} \otimes \dots \otimes l_{(i-1)} \otimes \psi \left( l_{(i)} \right)
      \otimes l_{(i+1)} \otimes l_{(r)}
    \end{equation*}
    the element of $\tens{A}$ which is the sum of all elements which are
    obtained by splitting $l$ into $r$ consecutive, possibly empty, lists,
    applying $\psi$ to the $i$-th list and taking the product \textbf{inside}
    $\tens{A}$. For example, if $l = a_1 \otimes a_2 \otimes a_3, r = 2$ and $i
    = 1$ we have
    \begin{align*}
      \psi \left( l_{(1)} \right) \otimes l_{(2)} &=
      \psi(1) \otimes \left( a_1 \otimes a_2 \right) + \psi(a_1) \otimes a_2 +
      \psi(a_1 \otimes a_2) \otimes 1 \\
      &= \psi(1) \otimes a_1 \otimes a_2 + \psi(a_1) \otimes a_2 +
	\psi(a_1,a_2) \in \tens{A}.
    \end{align*}
    Note that we use only the ``internal'' tensor product and identify $\psi
    \left( a_1 \otimes a_2 \right) \otimes 1$ with $\psi \left( a_1 \otimes a_2
    \right) = \psi \left( a_1, a_2 \right)$. We will also allow our notation to
    include signs as in the following expression:
    \begin{equation*}
      \begin{aligned}
        (-1)^{\degb{l_{(1)}}} \psi \left( l_{(1)} \right) \otimes l_{(2)} ={}&
	\psi(1) \otimes a_1 \otimes a_2 +
        (-1)^{\degb{a_1}} \psi(a_1) \otimes a_2\\
		&+ (-1)^{\degb{a_1} + \degb{a_2}} \psi(a_1 \otimes a_2).
      \end{aligned}
    \end{equation*}
\end{enumerate}

From now on, we work with the (reduced) tensor algebra of the shifted module $A[1]$, with the naturally induced grading. In particular, given an elementary tensor
$l = a_1 \otimes \dots \otimes a_k \in \tens{A[1]}$,
the degree $\degb{l}$ differs from the degree of the corresponding element in $\tens{A}$ by $k$.

The \textbf{structure of an $\Ainf$-algebra} on $A$ is given by a map
\[
\mu \colon \tens{A[1]} \rightarrow
A[1]
\]
of degree one such that $\nu(\mu(l)) \geq \nu(l)$ and $\nu(\mu(1)) > 0$, which satisfies
\[
  (-1)^{\degb{l_{(1)}}} \mu \left( l_{(1)} \otimes \mu(l_{(2)}) \otimes l_{(3)} \right) = 0.
\]
The equation above is called \textbf{the $A_\infty$ relations}.
Denote by $\mu_k \colon A[1]^{\otimes k} \rightarrow A[1]$
the composition of the inclusion of $A[1]^{\otimes k} \hookrightarrow
\tens{A[1]}$ with $\mu$. The \textbf{hat extension} of $\mu$ is the map
\[
\hat{\mu} \colon \tens{A[1]} \rightarrow \tens{A[1]}
\]
which is given by
\begin{equation}
  \hat{\mu}(l) = (-1)^{\degb{l_{(1)}}} l_{(1)} \otimes \mu(l_{(2)}) \otimes
  l_{(3)}.
\end{equation}
The hat extension $\hat{\mu}$ is the unique coderivation of the tensor coalgebra
whose natural projection onto $A[1]$ coincides with $\mu$. The $\Ainf$ relations
can equivalently be written as $\hat{\mu} \circ \hat{\mu} = 0$. Note that we allow
$\mu$ to satisfy $\mu_0 \neq 0$, i.e, the $\Ainf$-algebra can be curved.

\subsection{Hochschild homology of an \texorpdfstring{$\Ainf$}{A-infty} algebra}\label{ssec:hoch}

Let $(A,\mu)$ be an $\Ainf$-algebra. The Hochschild homology of $A$ (with
coefficients in $A$) is the homology of the following cochain complex. The
Hochschild cochain complex, as a graded $R$-module, is just the reduced tensor
algebra on $A[1]$ with the natural induced grading. We denote it by
$\choch[A] = \tensr{A[1]}$. The differential on the Hochschild cochain complex is
given by
\[
  \begin{aligned}
    \dhoch[x \otimes l] ={}& (-1)^{\degb{x} + \degb{l_{(1)}}}
    x \otimes l_{(1)} \otimes \mu(l_{(2)}) \otimes l_{(3)}
    \\
    &+ (-1)^{\degb{l_{(3)}} \cdot \left( \degb{x} + \degb{l_{(1)}} +
    \degb{l_{(2)}} \right)}
    \mu \left( l_{(3)} \otimes x \otimes l_{(1)} \right) \otimes l_{(2)}
    \\
    ={}& (-1)^{\degb{x}} x \otimes \hat{\mu}(l) +
    (-1)^{\degb{l_{(3)}} \cdot \left( \degb{x} + \degb{l_{(1)}} + \degb{l_{(2)}} \right)}
    \mu \left( l_{(3)} \otimes x \otimes l_{(1)} \right) \otimes l_{(2)}.
  \end{aligned}
\]
Here, $x \in A[1]$ and $l \in \tens{A[1]}$. It is standard that the $\Ainf$
relations imply $\d^2_{\textrm{hoch}} = 0$. We show this in the lemma below, in order to illustrate the use of our notation conventions.
\begin{lm}
  The map $\dhoch$ satisfies $\d^2_{\textrm{hoch}} = 0$.
\end{lm}
\begin{proof}
  We have
  \[
    \begin{aligned}
       \left( \dhoch \circ \dhoch \right)  \left( x \otimes l \right)
       ={}&
       \dhoch \left( (-1)^{\degb{x}} x \otimes \hat{\mu}(l) \right)    \\
    &+ \dhoch \left( (-1)^{\degb{l_{(3)}} \cdot \left( \degb{x} + \degb{l_{(1)}}
    + \degb{l_{(2)}} \right)} \mu(l_{(3)} \otimes x \otimes l_{(1)}) \otimes
    l_{(2)} \right).
    \end{aligned}
  \]
  Let's start by computing $\dhoch[(-1)^{\degb{x}} x \otimes \hat{\mu}(l)]$. Recall that the $A_\infty$ relations translate to $\hat\mu\circ\hat\mu=0$. So,
  \begin{equation*}
    \begin{aligned}
    \dhoch &\left( (-1)^{\degb{x}} x \otimes \hat{\mu}(l) \right) =
     (-1)^{\degb{x} + \degb{x}} x \otimes
    {\hat{\mu} \left( \hat{\mu} \left( l \right) \right)} +
    \\
    &(-1)^{\degb{x} + \degb{l_{(1)}} + \degb{l_{(32)}} \cdot \left( \degb{x} +
    \degb{l_{(1)}} + \degb{\mu(l_{(2)})} + \degb{l_{(31)}} \right)}
    \mu \left( l_{(32)} \otimes x \otimes l_{(11)} \right) \otimes l_{(12)} \otimes
    \mu(l_{(2)}) \otimes l_{31} +
    \\
    & (-1)^{\degb{x} + \degb{l_{(1)}} + \degb{l_{(33)}} \cdot \left( \degb{x} +
    \degb{l_{(1)}} + \degb{\mu(l_{(2)})} + \degb{l_{(31)}} + \degb{l_{(32)}} \right)}
	    \mu \left( l_{(33)} \otimes x \otimes l_{(1)} \otimes \mu(l_{(2)})
	    \otimes l_{(31)} \right)
	    \otimes l_{(32)} +
    \\
    & (-1)^{\degb{x} + \degb{l_{(1)}} + \left( \degb{l_{(13)}} +
    \degb{\mu(l_{(2)}} +
    \degb{l_{(3)}} \right) \cdot \left( \degb{x} + \degb{l_{(11)}} +
    \degb{l_{(12)}} \right)}
	    \mu \left( l_{(13)} \otimes \mu(l_{(2)}) \otimes l_{(3)} \otimes x
	    \otimes l_{(11)}
	    \right) \otimes l_{(12)}.
    \end{aligned}
  \end{equation*}
  All three terms in the expression above involve splitting the list $l$ into
  five consecutive lists. In the first term, we first split $l$ into $l_{(1)}
  \otimes l_{(2)} \otimes l_{(3)}$ and then both $l_{(1)}$ and $l_{(3)}$ into
  two lists. In the second term, we first split $l$ into $l_{(1)} \otimes
  l_{(2)} \otimes l_{(3)}$ and then split $l_{(3)}$ into three consecutive lists
  and so on. Since the specific order in which we perform the splitting doesn't
  matter, we can rewrite the expressions above as
  \begin{equation*} \begin{aligned}
     \dhoch & \left( (-1)^{\degb{x}} x \otimes \hat{\mu}(l) \right) =
     \\
     ={}& (-1)^{\degb{x} + \degb{l_{(1)}} + \degb{l_{(2)}} + \degb{l_{(5)}} \cdot
             \left( \degb{x} + \degb{l_{(1)}} + \degb{l_{(2)}} + \degb{\mu(l_{(3)})} +
             \degb{l_{(4)}} \right)}
     \mu \left( l_{(5)} \otimes x \otimes l_{(1)} \right) \otimes l_{(2)}
     \otimes \mu(l_{(3)}) \otimes l_{(4)} +
     \\
     & (-1)^{\degb{x} + \degb{l_{(1)}} + \degb{l_{(5)}} \cdot \left( \degb{x} +
	     \degb{l_{(1)}} + \degb{\mu(l_{(2)})} + \degb{l_{(3)}} +
	     \degb{l_{(4)}} \right)}
     \mu \left( l_{(5)} \otimes x \otimes l_{(1)} \otimes \mu(l_{(2)}) \otimes
     l_{(3)} \right) \otimes l_{(4)} +
     \\
     & (-1)^{\degb{x} + \degb{l_{(1)}} + \degb{l_{(2)}} + \degb{l_{(3)}} +
	     \left( \degb{l_{(3)}} + \degb{\mu(l_{(4)})} + \degb{l_{(5)}}
	     \right) \cdot
	     \left( \degb{x} + \degb{l_{(1)}} + \degb{l_{(2)}} \right)}
     \mu \left( l_{(3)} \otimes \mu(l_{(4)}) \otimes l_{(5)} \otimes x \otimes
     l_{(1)} \right) \otimes l_{(2)}
     \\
     ={}& (-1)^{{\degb{l_{(5)}} + \degb{x} + \degb{l_{(1)}} + \degb{l_{(2)}}} +
	     \degb{l_{(5)}} \cdot \left( \degb{x} + \degb{l_{(1)}} + \degb{l_{(2)}} +
	     \degb{l_{(3)}} + \degb{l_{(4)}} \right)}
     \mu \left( {l_{(5)} \otimes x \otimes l_{(1)}} \right) \otimes
     {l_{(2)}} \otimes \mu(l_{(3)}) \otimes l_{(4)} +
     \\
     & (-1)^{{\degb{l_{(5)}} + \degb{x} + \degb{l_{(1)}}}
	     + \degb{l_{(5)}} \cdot \left( \degb{x} + \degb{l_{(1)}} +
	     \degb{l_{(2)}} + \degb{l_{(3)}} + \degb{l_{(4)}} \right)}
     \mu \left({ l_{(5)} \otimes x \otimes l_{(1)}} \otimes
     \mu(l_{(2)}) \otimes l_{(3)} \right) \otimes l_{(4)} +
     \\
     & (-1)^{{\degb{l_{(3)}}} + \left( \degb{x} + \degb{l_{(1)}} +
             \degb{l_{(2)}} \right) \cdot \left( \degb{l_{(3)}} + \degb{l_{(4)}} +
	     \degb{l_{(5)}} \right)} \mu \left( {l_{(3)}} \otimes
     \mu(l_{(4)}) \otimes l_{(5)} \otimes x \otimes l_{(1)} \right)
     \otimes l_{(2)}.
  \end{aligned} \end{equation*}
  Next we will compute the term $\dhoch[ (-1)^{\degb{l_{(3)}} \cdot \left(
    \degb{x} + \degb{l_{(1)}} + \degb{l_{(2)}} \right)} \mu(l_{(3)} \otimes x
    \otimes l_{(1)}) \otimes l_{(2)}]$. We have
  \begin{align*}
    \dhoch & \left(
    (-1)^{\degb{l_{(3)}} \cdot \left( \degb{x} + \degb{l_{(1)}} + \degb{l_{(2)}}
          \right)}
    \mu(l_{(3)} \otimes x \otimes l_{(1)}) \otimes l_{(2)} \right) =
    \\
   ={} & (-1)^{\degb{l_{(3)}} \cdot \left( \degb{x} + \degb{l_{(1)}} + \degb{l_{(21)}} +
	    \degb{l_{(22)}} + \degb{l_{(23)}} \right) + \degb{\mu(l_{(3)}
	    \otimes x \otimes l_{(1)})} + \degb{l_{(21)}}} \\
&\hspace{15em} \cdot
    \mu(l_{(3)} \otimes x \otimes l_{(1)}) \otimes l_{(21)} \otimes
    \mu(l_{(22)}) \otimes l_{(23)} +
    \\
    & (-1)^{\degb{l_{(3)}} \cdot \left( \degb{x} + \degb{l_{(1)}} + \degb{l_{(21)}} +
    	    \degb{l_{(22)}} + \degb{l_{(23)}} \right) + \degb{l_{(23)}} \cdot \left(
	    \degb{\mu(l_{(3)} \otimes x \otimes l_{(1)})} + \degb{l_{(21)}} +
	    \degb{l_{(22)}}
	    \right)}\\
    &\hspace{15em}\cdot \mu \left( l_{(23)} \otimes \mu(l_{(3)} \otimes x \otimes l_{(1)}) \otimes
    l_{(21)} \right) \otimes l_{(22)} \\
    ={}& (-1)^{{\degb{l_{(3)}} + \degb{x} + \degb{l_{(1)}} + \degb{l_{(21)}}} +
      	    \degb{l_{(3)}} \cdot \left( \degb{x} + \degb{l_{(1)}} + \degb{l_{(21)}} +
    	    \degb{l_{(22)}} + \degb{l_{(23)}} \right) + 1}\\
     &\hspace{15em}\cdot\mu( {l_{(3)} \otimes x \otimes l_{(1)}}) \otimes
     {l_{(21)}}
     \otimes \mu(l_{(22)}) \otimes l_{(23)} +
    \\
    & (-1)^{{\degb{l_{(23)}}} + \left( \degb{l_{(23)}} + \degb{l_{(3)}}
	    \right) \cdot \left( \degb{x} + \degb{l_{(1)}} + \degb{l_{(21)}} +
	    \degb{l_{(22)}} \right)}
     \mu \left( {l_{(23)}} \otimes \mu(l_{(3)} \otimes x \otimes l_{(1)})
     \otimes l_{(21)} \right) \otimes l_{(22)}.
  \end{align*}
  Again, we have here two expressions which involve splitting the list $l$ into
  five consecutive lists. Thus, we can rewrite the expressions above as
  \begin{align*}
    \dhoch & \left(
    (-1)^{\degb{l_{(3)}} \cdot \left( \degb{x} + \degb{l_{(1)}} + \degb{l_{(2)}} \right)}
    \mu(l_{(3)} \otimes x \otimes l_{(1)}) \otimes l_{(2)} \right) =
    \\
   = {} & (-1)^{{\degb{l_{(5)}} + \degb{x} + \degb{l_{(1)}} + \degb{l_{(2)}}} +
      	    \degb{l_{(5)}} \cdot \left( \degb{x} + \degb{l_{(1)}} + \degb{l_{(2)}} +
    	    \degb{l_{(3)}} + \degb{l_{(4)}} \right) + 1}
     \mu( {l_{(5)} \otimes x \otimes l_{(1)}}) \otimes {l_{(2)}}
     \otimes \mu(l_{(3)}) \otimes l_{(4)} +
    \\
    & (-1)^{{\degb{l_{(4)}}} + \left( \degb{l_{(4)}} + \degb{l_{(5)}} \right)
	    \cdot \left( \degb{x} + \degb{l_{(1)}} + \degb{l_{(2)}} + \degb{l_{(3)}}
	    \right)}
     \mu \left( {l_{(4)}} \otimes \mu(l_{(5)} \otimes x \otimes l_{(1)})
     \otimes l_{(2)} \right) \otimes l_{(3)}.
  \end{align*}
  Combining both expressions and canceling, we are left with
  \begin{align*}
    ( \dhoch &
    \circ \dhoch )  \left( x \otimes l \right) =
    \\
    ={}& (-1)^{{\degb{l_{(5)}} + \degb{x} + \degb{l_{(1)}}} +
    \degb{l_{(5)}} \cdot \left( \degb{x} + \degb{l_{(1)}} + \degb{l_{(2)}} +
    \degb{l_{(3)}} + \degb{l_{(4)}} \right)} \mu \left( { l_{(5)}
    \otimes x \otimes l_{(1)}} \otimes \mu(l_{(2)}) \otimes l_{(3)} \right)
     \otimes l_{(4)} +
    \\
    & (-1)^{{\degb{l_{(3)}}} + \left( \degb{x} + \degb{l_{(1)}} +
    \degb{l_{(2)}} \right) \cdot \left( \degb{l_{(3)}} + \degb{l_{(4)}} +
    \degb{l_{(5)}} \right)} \mu \left( {l_{(3)}} \otimes \mu(l_{(4)})
    \otimes l_{(5)} \otimes x \otimes l_{(1)} \right)
     \otimes l_{(2)} +
    \\
    & (-1)^{{\degb{l_{(4)}}} + \left( \degb{l_{(4)}} + \degb{l_{(5)}}
    \right) \cdot \left( \degb{x} + \degb{l_{(1)}} + \degb{l_{(2)}} +
    \degb{l_{(3)}} \right)}
     \mu \left( {l_{(4)}} \otimes \mu(l_{(5)} \otimes x \otimes l_{(1)})
     \otimes l_{(2)} \right) \otimes l_{(3)}.
  \end{align*}
  To see that this is zero, let's start with $x \otimes l$, split $l$ into three
  parts and rotate one part so that it appears before $x$. With the appropriate
  signs, we get the expression
  \begin{equation*}
    (-1)^{\degb{l_{(3)}} \cdot \left( \degb{x} + \degb{l_{(1)}} +
    \degb{l_{(2)}} \right)} {l_{(3)} \otimes x \otimes l_{(1)}} \otimes l_{(2)}.
  \end{equation*}
  We   can apply $\mu  \circ \hat{\mu}$ to the first part $l_{(3)} \otimes x \otimes l_{(1)}$ of the
  expression and tensor the result with $l_{(2)}$. Doing so, we get
  \begin{align*}
    0 ={} & (-1)^{\degb{l_{(3)}} \cdot \left( \degb{x} + \degb{l_{(1)}} + \degb{l_{(2)}} \right)}
    \left( \mu \circ \hat{\mu} \right) \left( l_{(3)} \otimes x \otimes l_{(1)} \right) \otimes l_{(2)} =
    \\
  = {} &  (-1)^{\degb{l_{(3)}} \cdot \left( \degb{x} + \degb{l_{(11)}} + \degb{l_{(12)}} +
	     \degb{l_{(13)}} + \degb{l_{(2)}} \right) + {\degb{l_{(3)}} +
	     \degb{x} + \degb{l_{(11)}}}}
     \mu \left( {l_{(3)} \otimes x \otimes l_{(11)}} \otimes \mu(l_{(12)})
     \otimes l_{(13)} \right) \otimes l_{(2)} +
    \\
    &  (-1)^{\left( \degb{l_{(31)}} + \degb{l_{(32)}} \right) \cdot \left( \degb{x} +
	     \degb{l_{(11)}} + \degb{l_{(12)}} + \degb{l_{(2)}} \right) +
	     {\degb{l_{(31)}}}}
     \mu \left({l_{(31)}} \otimes \mu(l_{(32)} \otimes x \otimes l_{(11)})
     \otimes l_{(12)} \right) \otimes l_{(2)} +
    \\
    &  (-1)^{\left( \degb{l_{(31)}} + \degb{l_{(32)}} + \degb{l_{(33)}} \right) \cdot
    	     \left( \degb{x} + \degb{l_{(1)}} + \degb{l_{(2)}} \right) +
	     {\degb{l_{(31)}}}}
     \mu \left( {l_{(31)}} \otimes \mu(l_{(32)}) \otimes l_{(33)}
     \otimes x \otimes l_{(1)} \right) \otimes l_{(2)}.
  \end{align*}
  Here we again have expressions involving splitting $l$ into five consecutive
  lists and so ``performing the change of variables'', we get
  \begin{align*}
    0
    ={} &  (-1)^{{\degb{l_{(5)}} + \degb{x} + \degb{l_{(1)}}} + \degb{l_{(5)}} \cdot
	     \left( \degb{x} + \degb{l_{(1)}} + \degb{l_{(2)}} + \degb{l_{(3)}} + \degb{l_{(4)}}
	     \right)}
     \mu \left({l_{(5)} \otimes x \otimes l_{(1)}} \otimes \mu(l_{(2)}) \otimes
     l_{(3)}
     \right) \otimes l_{(4)} +
    \\
    &  (-1)^{{\degb{l_{(4)}}} + \left( \degb{l_{(4)}} + \degb{l_{(5)}} \right)
	     \cdot \left( \degb{x} + \degb{l_{(1)}} + \degb{l_{(2)}} + \degb{l_{(3)}}
	     \right)}
     \mu \left( {l_{(4)}} \otimes \mu(l_{(5)} \otimes x \otimes l_{(1)})
     \otimes l_{(2)} \right) \otimes l_{(3)} +
    \\
    &  (-1)^{{\degb{l_{(3)}}} + \left( \degb{l_{(3)}} + \degb{l_{(4)}} +
	     \degb{l_{(5)}} \right) \cdot \left( \degb{x} + \degb{l_{(1)}} + \degb{l_{(2)}}
	     \right)}
     \mu \left({l_{(3)}} \otimes \mu(l_{(4)}) \otimes l_{(5)} \otimes x \otimes l_{(1)}
     \right) \otimes l_{(2)}
  \end{align*}
  which is precisely the expression we got for $\left( \dhoch \circ \dhoch
  \right) \left( x \otimes l \right)$.
\end{proof}
The cohomology of the complex $\choch[A]$ is called the \textbf{Hochschild homology}
of $A$ (with coefficients in $A$) and is denoted by $\hhoch[A]$. Note that we use
cohomological grading on the Hochschild homology.

\subsection{Cyclic homology of an \texorpdfstring{$\Ainf$}{A-infty} algebra}

In general, there are several different constructions of chain complexes whose
cohomology gives us the cyclic homology of an algebra. In this section we will
verify that one of the constructions, Connes' complex, once interpreted
correctly, works almost verbatim not only for an algebra but also for an $\Ainf$-algebra, possibly with a non-zero curvature term.

\subsubsection{Connes' Complex for cyclic homology}\label{sssec:connes}

The construction of Connes' complex is based on the following basic operator.
Let $\t \colon \choch[A] \rightarrow \choch[A]$ be given
by

\[
  \t[x_1 \otimes \dots \otimes x_n] =
  (-1)^{\degr{x_n} \cdot \left( \degr{x_1} + \dots + \degr{x_{n-1}} \right)} x_n \otimes
  x_1 \otimes \dots \otimes x_{n-1}.
\]
Using our list notation, we have
\[
  \t[l \otimes x] = (-1)^{\degr{x} \cdot \degr{l}} x \otimes l.
\]
This operator has degree zero with respect to the grading on $\choch[A]$.

\begin{lm} \label{lm:tlemma} We have the identity $\dhoch \circ (1 - \t) = (1 - \t) \circ
  \hat{\mu}$.
\end{lm}
\begin{proof}
  The proof is a lengthy but straightforward calculation.  We'll first verify the
  identity for (elementary) tensors of weight greater than or equal to two. Such
  tensors can be written in our notation as $x \otimes l \otimes z$ when $x,z
  \in A[1]$ and $l = y_1 \otimes \dots y_k$ for $k \geq 0$ (when $k = 0$ this
  means that $l = 1$ and we are working with $x \otimes l \otimes z = x \otimes
  z$). We have
  \begin{align*}
    \dhoch \left(x \otimes l \otimes z \right) ={}
        &\mu(x \otimes l \otimes z) +
    \\
    &\mu(x \otimes l_{(1)}) \otimes l_{(2)} \otimes z +
    \\
    &(-1)^{\degb{x} + \degb{l_{(1)}}} x \otimes l_{(1)} \otimes \mu(l_{(2)}) \otimes l_{(3)}
    \otimes z +
    \\
    &(-1)^{\degb{x} + \degb{l_{(1)}}} x \otimes l_{(1)} \otimes \mu(l_{(2)} \otimes z) +
    \\
    &(-1)^{\degb{x} + \degb{l} +\degb{z}} x \otimes l \otimes z \otimes \mu(1) +
    \\
    &(-1)^{\left( \degb{l_{(3)}} + \degb{z} \right) \cdot \left( \degb{x} +
    \degb{l_{(1)}} + \degb{l_{(2)}} \right)} \mu \left( l_{(3)} \otimes z \otimes x \otimes
    l_{(1)} \right) \otimes l_{(2)}.
  \end{align*}
  Similarly, we have
  \begin{align*}
    \left( \dhoch \circ \t \right) & \left( x \otimes l \otimes z \right) =
    (-1)^{\degb{z} \cdot \left( \degb{x} + \degb{l} \right)}
    \dhoch[z \otimes x \otimes l]
    \\
    ={} & (-1)^{\degb{z} \cdot \left( \degb{x} + \degb{l} \right) + \degb{z}}
    z \otimes \mu(1) \otimes x \otimes l +
    \\
    &(-1)^{\degb{z} \cdot \left( \degb{x} + \degb{l} \right) + \degb{z}}
    z \otimes \mu(x \otimes l_{(1)}) \otimes l_{(2)} +
    \\
    &(-1)^{\degb{z} \cdot \left( \degb{x} + \degb{l}
    \right) + \degb{z} + \degb{x} + \degb{l_{(1)}}}
    z \otimes x \otimes l_{(1)} \otimes \mu(l_{(2)}) \otimes l_{(3)} +
    \\
    &(-1)^{\degb{z} \cdot \left( \degb{x} + \degb{l} \right) +
    \degb{l_{(2)}} \cdot \left( \degb{z} +\degb{x} + \degb{l_{(1)}} \right)}
    \mu(l_{(2)} \otimes z) \otimes x \otimes l_{(1)} +
    \\
    &(-1)^{\degb{z} \cdot \left( \degb{x} + \degb{l_{(1)}} + \degb{l_{(2)}} + \degb{l_{(3)}}
    \right) + \degb{l_{(3)}} \cdot \left( \degb{z} + \degb{x} + \degb{l_{(1)}} +
    \degb{l_{(2)}} \right)}
    \mu(l_{(3)} \otimes z \otimes x \otimes l_{(1)}) \otimes l_{(2)} +
    \\
    &(-1)^{\degb{z} \cdot \left( \degb{x} + \degb{l} \right) +
    \left( \degb{x} + \degb{l} \right) \cdot \degb{z}}
    \mu(x \otimes l \otimes z)
    \\
    ={}& (-1)^{\degb{z} \cdot \left( \degb{x} + \degb{l} + 1 \right)}
    z \otimes \mu(1) \otimes x \otimes l +
    \\
    & (-1)^{\degb{z} \cdot \left( \degb{x} + \degb{l} + 1 \right)}
    z \otimes \mu(x \otimes l_{(1)}) \otimes l_{(2)} +
    \\
    & (-1)^{\degb{z} \cdot \left( \degb{z} + \degb{l} + 1 \right) + \degb{x} +
      \degb{l_{(1)}}}
    z \otimes x \otimes l_{(1)} \otimes \mu(l_{(2)}) \otimes l_{(3)} +
    \\
    & (-1)^{\left( \degb{x} + \degb{l_{(1)}} \right) \cdot \left( \degb{z} +
      \degb{l_{(2)}} \right)}
    \mu(l_{(2)} \otimes z) \otimes x \otimes l_{(1)} +
    \\
    & (-1)^{\left( \degb{l_{(3)}} + \degb{z} \right) \cdot \left( \degb{x} +
        \degb{l_{(1)}} + \degb{l_{(2)}} \right)}
    \mu(l_{(3)} \otimes z \otimes x \otimes l_{(1)}) \otimes l_{(2)} +
    \\
    & \mu(x \otimes l \otimes z).
  \end{align*}
  Each expression above is a sum of six terms (where some of the terms are
  themselves sums of expressions of a certain type). Subtracting and canceling identical
  terms, we get
  \begin{align*}
    \left( \dhoch \circ (1 - \t) \right)  \left( x \otimes l \otimes z \right) ={}
    &\mu(x \otimes l_{(1)}) \otimes l_{(2)} \otimes z +
    \\
    &(-1)^{\degb{x} + \degb{l_{(1)}}} x \otimes l_{(1)} \otimes \mu(l_{(2)}) \otimes l_{(3)}
    \otimes z +
    \\
    &(-1)^{\degb{x} + \degb{l_{(1)}}} x \otimes l_{(1)} \otimes \mu(l_{(2)} \otimes z) +
    \\
    &(-1)^{\degb{x} + \degb{l} +\degb{z}} x \otimes l \otimes z \otimes \mu(1) +
    \\
    & (-1)^{\degb{z} \cdot \left( \degb{x} + \degb{l} + 1 \right) + 1}
    z \otimes \mu(1) \otimes x \otimes l +
    \\
    &(-1)^{\degb{z} \cdot \left( \degb{x} + \degb{l} + 1 \right) + 1}
    z \otimes \mu(x \otimes l_{(1)}) \otimes l_{(2)} +
    \\
    &(-1)^{\degb{z} \cdot \left( \degb{x} + \degb{l} + 1 \right) + \degb{x} +
      \degb{l_{(1)}} + 1} z \otimes x \otimes l_{(1)} \otimes \mu(l_{(2)}) \otimes l_{(3)} +
    \\
    &(-1)^{\left( \degb{x} + \degb{l_{(1)}} \right) \cdot \left( \degb{z} +
    \degb{l_{(2)}} \right) + 1}
    \mu(l_{(2)} \otimes z) \otimes x \otimes l_{(1)}.
  \end{align*}
  Next, we have
  \begin{align*}
    \hat{\mu} \left( x \otimes l \otimes z \right) ={}
    & \mu(1) \otimes x \otimes l \otimes z +
    \\
    & \mu(x \otimes l_{(1)}) \otimes l_{(2)} \otimes z +
    \\
    & \mu(x \otimes l \otimes z) +
    \\
    & (-1)^{\degb{x} + \degb{l_{(1)}}} x \otimes l_{(1)} \otimes \mu(l_{(2)}) \otimes l_{(3)}
    \otimes z +
    \\
    & (-1)^{\degb{x} + \degb{l_{(1)}}} x \otimes l_{(1)} \otimes \mu(l_{(2)} \otimes z)
    \\
    & (-1)^{\degb{x} + \degb{l} + \degb{z}} x \otimes l \otimes z \otimes \mu(1),
  \end{align*}
  which implies that
  \begin{align*}
    \left( \t \circ \hat{\mu} \right)  \left( x \otimes l \otimes z \right) ={}
    &(-1)^{\degb{z} \cdot \left( \degb{\mu(1)} + \degb{x} + \degb{l} \right)}
    z \otimes \mu(1) \otimes x \otimes l +
    \\
    &(-1)^{\degb{z} \cdot \left( \degb{\mu(x \otimes l_{(1)})} + \degb{l_{(2)}} \right)}
    z \otimes \mu(x \otimes l_{(1)}) \otimes l_{(2)} +
    \\
    &\mu(x \otimes l \otimes z) +
    \\
    &(-1)^{\degb{x} + \degb{l_{(1)}} + \degb{z} \cdot \left( \degb{x} + \degb{l_{(1)}} +
    \degb{\mu(l_{(2)})} + \degb{l_{(3)}} \right)}
    z \otimes x \otimes l_{(1)} \otimes \mu(l_{(2)}) \otimes l_{(3)} +
    \\
    &(-1)^{\degb{x} + \degb{l_{(1)}} + \degb{\mu(l_{(2)} \otimes z)} \cdot \left(
    \degb{x} + \degb{l_{(1)}} \right)}
    \mu(l_{(2)} \otimes z) \otimes x \otimes l_{(1)} +
    \\
    &(-1)^{\degb{x} + \degb{l} + \degb{z} + \degb{\mu(1)} \cdot \left( \degb{x}
    + \degb{l} + \degb{z} \right)}
    \mu(1) \otimes x \otimes l \otimes z
    \\
    ={}&(-1)^{\degb{z} \cdot \left( \degb{x} + \degb{l} + 1 \right)}
    z \otimes \mu(1) \otimes x \otimes l +
    \\
    &(-1)^{\degb{z} \cdot \left( \degb{x} + \degb{l} + 1 \right)}
    z \otimes \mu(x \otimes l_{(1)}) \otimes l_{(2)} +
    \\
    &\mu(x \otimes l \otimes z) +
    \\
    &(-1)^{\degb{z} \cdot \left( \degb{x} + \degb{l} + 1 \right)  +
    \degb{x} + \degb{l_{(1)}}}
    z \otimes x \otimes l_{(1)} \otimes \mu(l_{(2)}) \otimes l_{(3)} +
    \\
    &(-1)^{\left( \degb{x} + \degb{l_{(1)}} \right) \cdot \left(\degb{z} +
    \degb{l_{(2)}} \right)}
    \mu(l_{(2)} \otimes z) \otimes x \otimes l_{(1)} +
    \\
    &\mu(1) \otimes x \otimes l \otimes z.
  \end{align*}
  Again, each expression above is a sum of six terms. Subtracting and
  canceling the two identical terms, we get
  \begin{align*}
    \left( \left( 1 - \t \right) \circ \hat{\mu} \right) \left( x \otimes l
    \otimes z \right) ={}
    & \mu(x \otimes l_{(1)}) \otimes l_{(2)} \otimes z +
    \\
    & (-1)^{\degb{x} + \degb{l_{(1)}}} x \otimes l_{(1)} \otimes \mu(l_{(2)}) \otimes l_{(3)}
    \otimes z +
    \\
    & (-1)^{\degb{x} + \degb{l_{(1)}}} x \otimes l_{(1)} \otimes \mu(l_{(2)} \otimes z) +
    \\
    & (-1)^{\degb{x} + \degb{l} + \degb{z}}
    x \otimes l \otimes z \otimes \mu(1) +
    \\
    &(-1)^{\degb{z} \cdot \left( \degb{x} + \degb{l} + 1 \right) + 1}
    z \otimes \mu(1) \otimes x \otimes l +
    \\
    &(-1)^{\degb{z} \cdot \left( \degb{x} + \degb{l} + 1 \right) + 1}
    z \otimes \mu(x \otimes l_{(1)}) \otimes l_{(2)} +
    \\
    &(-1)^{\degb{z} \cdot \left( \degb{x} + \degb{l} + 1 \right) +
    \degb{x} + \degb{l_{(1)}} + 1}
    z \otimes x \otimes l_{(1)} \otimes \mu(l_{(2)}) \otimes l_{(3)} +
    \\
    &(-1)^{\left( \degb{x} + \degb{l_{(1)}} \right) \cdot \left(\degb{z} +
    \degb{l_{(2)}} \right) + 1}
    \mu(l_{(2)} \otimes z) \otimes x \otimes l_{(1)}.
  \end{align*}
  This shows the identify for elementary tensors of weight greater than or equal
  to two. To check the remaining case, we note that if $x \in A[1]$ then $(1 -
  \t)(x) = 0$  and so $\left( \dhoch \circ (1 - \t) \right)(x) = 0$. Meanwhile,
  \begin{align*}
    \hat{\mu}(x) = \mu(x) + \mu(1) \otimes x + (-1)^{\degb{x}} x \otimes \mu(1),
  \end{align*}
  which implies that
  \begin{align*}
    \left( \t \circ \hat{\mu} \right)(x) &= \mu(x) + (-1)^{\degb{x} \cdot
    \degb{\mu(1)}} x \otimes \mu(1) + (-1)^{\degb{x} + \degb{\mu(1)} \cdot
    \degb{x}} \mu(1) \otimes x \\
    &= \mu(x) + (-1)^{\degb{x}} x \otimes \mu(1) + \mu(1) \otimes x.
  \end{align*}
  Subtracting both expressions, we get that $\left( \dhoch \circ (1 - \t)
  \right)(x) = \left( \left( 1 - \t \right) \circ \hat{\mu} \right)(x) = 0$.
\end{proof}
\Cref{lm:tlemma} is a generalization of Lemma 2.1.1 from \cite{Loday}, where the
identity is shown for associative algebras.
We can interpret \cref{lm:tlemma} as stating that the map
\begin{equation*}
  1 - \t \colon \left( \tensr{A[1]}, \hat{\mu} \right) \rightarrow \left( \choch[A][], \dhoch \right)
\end{equation*}
is a morphism of chain complexes. It follows that the Hochschild differential $\dhoch$ descends to the quotient $\choch[A][] / \im (1 - \t)$. The quotient chain complex is called the \textbf{Connes complex} and we will
denote it by $\cconnes[A]$. The cohomology of the complex will be denoted by
$\hcyc[A]$ and is called the \textbf{cyclic homology} of $A$.

\subsubsection{Extended cyclic homology}\label{sssec:extend}

The curvature term $\mu_0(1)$ is not a closed element in the Hochschild chain
complex, but it becomes closed in the cyclic complex, as shown in the following lemma.

\begin{lm}
    The element $\mu_0 \left ( 1 \right)$ defines a cocycle in $\cconnes[A]$.
\end{lm}
\begin{proof}
    Let us compute the action of the Hochschild differential on the curvature term $\mu_0(1)$. From the $A_\infty$ relations we get
    $\mu_1(\mu_0(1))=0$, so

    \begin{equation}\label{eq:dex}
      \dhoch[\mu_0(1)] = {\mu_1 \left( \mu_0(1) \right)} + (-1)^{\degr{\mu_0(1)}}
      \mu_0(1) \otimes \mu_0(1) = -\mu_0(1) \otimes \mu_0(1).
    \end{equation}

    This is clearly non-zero if $\mu_0(1) \neq 0$, so the element $\mu_0(1)$ doesn't
    define a Hochschild homology class. However, since $\mu_0(1)$ has degree one in
    $A[1]$ and we assume that $2$ is invertible, we have
    \[
      \begin{aligned}
         \left( 1 - \t \right) \left( \frac{\mu_0(1) \otimes \mu_0(1)}{2} \right) &=
         \frac{1}{2} \left( \mu_0(1) \otimes \mu_0(1) - (-1)^{\degr{\mu_0(1)} \cdot
         \degr{\mu_0(1)}} \mu_0(1) \otimes \mu_0(1) \right)
         \\
         &= \mu_0(1) \otimes \mu_0(1).
      \end{aligned}
    \]
    Therefore, in the cyclic complex $\cconnes[A]$, the class of $\mu_0 \left ( 1 \right)$ is closed.
\end{proof}

In general, the class of $\mu_0(1)$ need not be exact, but we can make it exact by adding to the cyclic complex a copy of $R$. In other words, we can extend the cyclic chain complex so
that it becomes a quotient not only of the reduced tensor algebra $\tensr{A[1]}$
but of the full tensor algebra $\tens{A[1]}$. Let us define
\[
  \ceconnes[A] := \cconnes[A] \oplus R^{*} = \choch[A] / \im (1 - \t) \oplus
  R^{*} = \tens{A[1]} / \im (1 - \t),
\]
where the action of $\t$ is extended so that it acts as identity on $R$. Note
that $R$ is possibly graded, so the definition possibly modifies every
homogeneous piece $\cconnes[A][r]$. In the simplest case, $R = k$ is a field concentrated in
degree zero and this construction only adds a copy of the base field in degree
zero. The differential on this complex will still be denoted by $\dhoch$ and is
defined on $R$ by the requirement that $\dhoch[1] = \mu_0(1)$ where $1 \in R$ is
the unit, and extended $R$-linearly. Note that $\dhoch$
does not define a differential on $\choch[A][] \oplus R = \tens{A[1]}$, as our
calculation~\eqref{eq:dex} shows, so this is a slight abuse of notation.
The homology of the extended cylic
complex $\ceconnes[A]$ with respect to $\dhoch$ will be denoted by $\hecyc[A]$
and will be called the \textbf{extended cyclic homology} of $A$. When $\mu_0(1)
= 0$, this construction is not really useful as it only adds a copy of $R$ with
zero differential.
In general, this construction immediately gives:
\begin{lm}
The element $\mu_0 \left ( 1 \right)$ defines an exact cocycle in $\ceconnes[A]$.
\end{lm}

\subsection{Normalized and reduced homologies}\label{sssec:normNred}

In order to define normalized and reduced complexes, we need the presence of a unit.

\begin{dfn}\label{dfn:unital}
 Let $(A,\mu)$ be an $\Ainf$-algebra. A strict unit for $A$ is an
  element $\overline{e} \in A[1]^{-1}$ which satisfies
  \begin{enumerate}
    \item $\mu_k(x_1, \dots, x_k) = 0$ if $k \neq 2$ and $x_i = \overline{e}$
      for some $1 \leq i \leq k$.
    \item $\mu_2(\overline{e}, x) = x$ and
      $\mu_2(x, \overline{e}) = (-1)^{\degr{x} + 1} x$, for all $x \in A[1]$.
  \end{enumerate}
\end{dfn}

Assume $(A,\mu)$ is equipped with a unit. Then the following holds:
\begin{lm}
    Let $(A,\mu,\overline{e})$ be a unital $\Ainf$-algebra.
    \begin{enumerate}
        \item\label{it:degen_hoch}
         The submodule of the Hochschild complex $\choch[A] = \tensr{A[1]}$ given by
            \begin{gather*}
              D^{*}(A) = \langle x_0 \otimes \dots \otimes x_k \, \mid \,
              x_0, \dots, x_k \in A[1], \, \exists \, 1 \leq i \leq k \textrm{ such that
              } x_i = \overline{e} \rangle
            \end{gather*}
            is a subcomplex of $\left( \choch[A], \dhoch \right)$.
        \item\label{it:degen_con}
        The submodule of Connes' complex $\cconnes[A] = \tensr{A[1]} / \im (1 - \t)$ given by
            \begin{equation*}
                E_{\lambda}^{*}(A) = \langle [x_1 \otimes \dots \otimes x_k] \, \mid \,
                x_1, \dots, x_k \in A[1], \, \exists \, 1 \leq i \leq k \textrm{ such that
            } x_i = \overline{e} \rangle
            \end{equation*}
            is a subcomplex of $\left( \cconnes[A], \dhoch \right)$.
    \end{enumerate}
\end{lm}
\begin{proof}
We shall consider
elements of $\choch[A]$ of the form $x_1 \otimes \dots \otimes x_k$ where $x_i
\in A[1]$ such that $x_i = \overline{e}$ for some $1 \leq i \leq k$, and compute
the action of the Hochschild differential on them. We have three cases:

\begin{enumerate}[label={(\alph*)}]
  \item\label{it:e_begin}
  The unit appears in the beginning of the list. Such elements can be
    written as
    \begin{equation*}
      \overline{e} \otimes l = \overline{e} \otimes \underbrace{x_1 \otimes
      \dots \otimes x_k}_{l}.
    \end{equation*}
    Then
    \begin{align*}
  	\dhoch \left( \overline{e} \otimes l \right) ={}&
	(-1)^{\degb{\overline{e}} + \degb{l_{(1)}}} \overline{e} \otimes l_{(1)} \otimes
	\mu(l_{(2)}) \otimes l_{(3)}  + \mu_2(\overline{e}, x_1) \otimes x_2 \otimes
	\dots \otimes x_k +\\
	&
 	(-1)^{\degb{x_k} \cdot \left( \degb{\overline{e}} + \sum_{i=1}^{k-1}
  	\degb{x_i} \right)} \mu_2(x_k, \overline{e}) \otimes x_1
  	\otimes \dots \otimes x_{k-1} \\
  	={}&
	(-1)^{\degb{l_{(1)}} - 1} \overline{e} \otimes l_{(1)} \otimes \mu(l_{(2)}) \otimes l_{(3)} +
  	x_1 \otimes \dots \otimes x_k +\\
	&
  	(-1)^{\degb{x_k} \cdot \left( \degb{\overline{e}} + \sum_{i=1}^{k-1}
  	\degb{x_i} \right) + \degb{x_k} + 1} x_k \otimes x_1 \otimes \dots \otimes
  	x_{k-1} \\
	={}& (-1)^{\degb{l_{(1)}} - 1} \overline{e} \otimes l_{(1)} \otimes \mu(l_{(2)})
	\otimes l_{(3)} + \left( 1 - \t \right) \left( l \right).
    \end{align*}
  \item\label{it:e_mid}
   The unit appears in the middle of the list. Such elements can be written as
    \begin{equation*}
      x_0 \otimes l \otimes \overline{e} \otimes s =
      x_0 \otimes \underbrace{x_1 \otimes \dots \otimes x_i}_{l} \otimes \overline{e} \otimes
      \underbrace{x_{i+1} \otimes \dots \otimes x_k}_{s}
    \end{equation*}
    with $0 \leq i < k$. Then  

    \begin{align*}
      \dhoch \left( x_0 \otimes l \otimes \overline{e} \otimes s \right) ={} &
      (-1)^{\degb{x_0} + \degb{l_{(1)}}} x_0 \otimes l_{(1)} \otimes \mu(l_{(2)}) \otimes l_{(3)}
      \otimes \overline{e} \otimes s +\\
      &
      (-1)^{\degb{x_0} + \degb{l} + \degb{\overline{e}} + \degb{s_{(1)}}} x_0 \otimes l
      \otimes \overline{e} \otimes s_{(1)} \otimes \mu(s_{(2)}) \otimes s_{(3)} +\\
      &
      (-1)^{\degb{s_{(2)}} \left( \degb{x_0} + \degb{l} + \degb{\overline{e}} +
      \degb{s_{(1)}} \right)} \mu(s_{(2)} \otimes x_0 \otimes l_{(1)}) \otimes l_{(2)} \otimes
      \overline{e} \otimes s_{(1)} +\\
      &
      (-1)^{\degb{x_0} + \sum_{j=1}^{i-1} \degb{x_i}} x_0 \otimes x_1 \otimes \dots
      \otimes x_{i-1} \otimes \mu_2(x_i, \overline{e}) \otimes s +\\
      &
      (-1)^{\degb{x_0} + \degb{l}} x_0 \otimes l \otimes \mu_2(\overline{e},
      x_{i+1}) \otimes x_{i+2} \otimes \dots \otimes x_k \\
      ={}&
      (-1)^{\degb{x_0} + \degb{l_{(1)}}} x_0 \otimes l_{(1)} \otimes \mu(l_{(2)}) \otimes l_{(3)}
      \otimes \overline{e} \otimes s +\\
      &
      (-1)^{\degb{x_0} + \degb{l} + \degb{s_{(1)}} - 1} x_0 \otimes l \otimes
      \overline{e} \otimes s_{(1)} \otimes \mu(s_{(2)}) \otimes s_{(3)} +\\
      &
      (-1)^{\degb{s_{(2)}} \left( \degb{x_0} + \degb{l} + \degb{s_{(1)}} - 1 \right)}
      \mu(s_{(2)} \otimes x_0 \otimes l_{(1)}) \otimes l_{(2)} \otimes \overline{e} \otimes
      s_{(1)}.
    \end{align*}
  \item\label{it:e_end}
  The unit appears at the end of the list. Such elements can be written as
    \begin{equation*}
      x_0 \otimes l \otimes \overline{e} = x_0 \otimes \underbrace{x_1 \otimes \dots
      \otimes x_k}_{l} \otimes \overline{e}
    \end{equation*}
    with $k \geq 0$. Then
    \begin{align*}
      \dhoch \left( x_0 \otimes l \otimes \overline{e} \right)
      ={}& \mu(x_0 \otimes
      l_{(1)}) \otimes l_{(2)} \otimes \overline{e} +\\
      & (-1)^{\degb{x_0} + \degb{l_{(1)}}} x_0 \otimes l_{(1)} \otimes \mu(l_{(2)})
      \otimes l_{(3)} \otimes \overline{e} +\\
      &
      (-1)^{\degb{x_0} + \degb{l} + \degb{\overline{e}}} x_0 \otimes l \otimes
      \overline{e} \otimes \mu_0(1) +\\
      &
      (-1)^{\degb{x_0} + \sum_{i=1}^{k-1} \degb{x_i}} x_0 \otimes x_1 \otimes \dots
      \otimes x_{k-1} \otimes \mu_2(x_k, \overline{e}) +\\
      &
      (-1)^{\degb{e} \cdot \left( \degb{x_0} + \degb{l} \right)}
      \mu_2(\overline{e}, x_0) \otimes l \\
      ={} &
      \mu(x_0 \otimes l_{(1)}) \otimes l_{(2)} \otimes \overline{e} +\\
      &
      (-1)^{\degb{x_0} + \degb{l_{(1)}}} x_0 \otimes l_{(1)} \otimes \mu(l_{(2)}) \otimes
      l_{(3)} \otimes \overline{e} +\\
      &
      (-1)^{\degb{x_0} + \degb{l} -1} x_0 \otimes l \otimes
      \overline{e} \otimes \mu_0(1).
    \end{align*}
\end{enumerate}
    Part~\eqref{it:degen_hoch} follows from cases~\ref{it:e_mid} and~\ref{it:e_end}, while part~\eqref{it:degen_con} follows from all three cases.
\end{proof}

The quotient complex
\[
  \chochnorm[A] = \left( \choch[A] / D^{*}(A), \dhoch \right)
\]
is called the \textbf{normalized Hochschild complex}.
In the case where $A$ is an associative algebra (with no curvature term), the complex $D^{*}(A)$ is called the \textbf{degenerate chain complex} and has a trivial homology.
The quotient
complex
 \begin{equation*}
  \crconnes[A] := \left( \ceconnes[A] / E_{\lambda}^{*}(A), \dhoch
\right)
\end{equation*}
is called the \textbf{reduced cyclic complex} and its cohomology is
denoted by $\hrconnes[A]$. The submodule $E_{\lambda}^{*}(A)$ is also a
subcomplex of the extended complex $\left( \ceconnes[A], \dhoch \right)$.
The quotient complex
\begin{equation*}
  \cerconnes[A] := \left( \ceconnes[A] / E_{\lambda}^{*}(A), \dhoch
\right)
\end{equation*}
is called the \textbf{extended and reduced cyclic complex} and its
cohomology is denoted by $\herconnes[A]$.

\section{Geometric background}\label{sec:background}

In Section ~\ref{ssec:currents},~\ref{ssec:int} we discuss currents and operations on them,
and cite needed results proven elsewhere, primarily in ~\cite{ST4}.
In Section ~\ref{ssec:qemptyset} we define the closed maps. In
Section ~\ref{ssec:not} we establish useful notation and conventions and in
~\ref{ssec:q} we define the closed-open maps and cite their properties,
proven in ~\cite{ST1}.

Throughout, we use conventions on orbifolds with corners and orientation thereof from~\cite{ST4}, see there for full detail. The only difference is that, for simplicity, in the current manuscript we write ``smooth'' instead of what was called ``strongly smooth'' in~\cite{ST4}.

\subsection{Currents}\label{ssec:currents}

Let $M$ be a compact oriented orbifold with corners. Denote by $\ccur{M}[k]$ the
space of currents of cohomological degree $k$, that is, the dual space of
differential forms $A^{\dim M - k}(M)$. Differential forms are identified as a
subspace of currents by
\begin{gather*}
  \varphi_{M} = \varphi \colon \cdiff{M}[k] \hookrightarrow \ccur{M}[k],\\
  \varphi(\eta)(\alpha):=\int_M\eta\wedge\alpha,\quad \alpha\in \cdiff{M}[\dim M
  -k].
\end{gather*}
Accordingly, for a general current $\zeta,$ we may use the notation
\begin{equation}\label{eq:diw}
  \zeta(\alpha)=\int_{M} \zeta\wedge\alpha.
\end{equation}
Note that the identification $\varphi$ depends on the orientation of $M$. Following~\cite[Section 6]{ST4}, we
have the following operations on currents:
\begin{enumerate}
  \item \textbf{Exterior derivative}. The exterior derivative $d \colon
    \ccur{M}[k] \rightarrow \ccur{M}[k+1]$ of a current $\zeta \in
    \ccur{M}[k]$ is defined by the formula
    \begin{equation*}
      d\zeta(\alpha) :=(-1)^{|\zeta| + 1}\zeta(d\alpha).
    \end{equation*}
    Clearly we have $d^2(\zeta) = 0$ so we get a cochain complex $\left(
    \ccur{M}[*], d \right)$.
    The choice of sign guarantees that if $M$ is
    closed, the definition generalizes the exterior derivative of differential
    forms. That is, we have $d\varphi(\eta) = \varphi(d\eta)$ and so $\varphi$
    is a chain map, and the chain complex $\cdiff{M}[*]$ can be identified as a
    \textbf{subcomplex} of $\ccur{M}[*]$.
  \item \textbf{Push-forward}. Given a proper morphism of orbifolds with corners
    $f \colon M \rightarrow N$, denote by $\rdim f:=\dim N-\dim M$. The
    push-forward
    \begin{equation} \label{eq:pfc}
      f_{*} \colon \ccur{M}[k] \rightarrow \ccur{N}[k - \rdim f]
    \end{equation}
    is defined by the formula
    \begin{equation*}
      (f_{*}(\zeta))(\alpha) := (-1)^{|\alpha| \cdot \rdim f} \zeta(f^{*} \alpha).
    \end{equation*}
    When the map $f$ is a relatively oriented
    submersion, $N$ is oriented and $M$ is endowed with the orientation which is
    compatible with $f$ and $N$, the choice of sign guarantees that the
    definition generalizes the push-forward of differential forms, namely, integration over the fiber. That is, we
    have $f_{*} \left( \varphi_{M} \left( \eta \right) \right) = \varphi_{N}
    \left( f_{*} \left( \eta \right) \right)$. This follows from parts \eqref{prop:pushcomp}-\eqref{prop:pushpull} of Proposition~\ref{prop:proppp} below. The push-forward commutes with
    the exterior derivative so we have $d \left( f_{*} \left( \zeta \right)
    \right) = f_{*}(d \zeta)$.
  \item \textbf{Pullback}. Given a relatively oriented
    \textbf{submersion} $f \colon M \rightarrow N$, the pullback
    \begin{equation} \label{eq:pbc}
      f^{*} \colon \ccur{N}[k] \rightarrow \ccur{M}[k]
    \end{equation}
    is defined by the formula
    \begin{equation*}
      \left( f^{*} \left( \zeta \right) \right)(\alpha) := \zeta \left( f_{*}
      \left( \alpha \right) \right).
    \end{equation*}
    When the current $\zeta$ is a differential form, $N$ is oriented and we endow $M$
    with the orientation which is compatible with $f$ and $N$, this generalizes
    the usual pullback.  That is, we have $f^{*} \left( \varphi_{N} \left( \eta
    \right) \right) = \varphi_{M} \left( f^{*} \left( \eta \right) \right)$. This follows from parts \eqref{prop:pushcomp}-\eqref{prop:pushpull} of Proposition~\ref{prop:proppp} below.
    When $M$ has no boundary, the pullback commutes with the exterior derivative
    so we have $d \left( f^{*} \left( \zeta \right) \right) = f^{*} \left( d \left( \zeta
    \right) \right)$, while if $M$ has a boundary, we have an additional term. See Proposition~\ref{stokes}.
    The pullback is functorial.  That is, if $g
    \colon N \rightarrow L$ is also a relatively oriented submersion, we have $\left( g \circ f \right)^{*}
    \left( \zeta \right) = f^{*} \left( g^{*} \left( \zeta \right) \right)$.
    This follows from Proposition~\ref{prop:proppp}\eqref{prop:pushcomp}.
  \item \textbf{Exterior product}. In general, the exterior product of two
    currents is not defined. However, given a current $\zeta \in \ccur{M}[k]$
    and a differential form $\beta \in \cdiff{M}[l]$, the exterior product
    $\zeta \wedge \beta$ is a current of degree $k + l$ defined by the formula
    \begin{equation*}
      \left( \zeta \wedge \beta \right)(\alpha) := \zeta( \beta \wedge \alpha).
    \end{equation*}
    When the current $\zeta$ is a differential form, this definition generalizes
    the usual exterior product. That is, we have $\varphi(\eta) \wedge \beta =
    \varphi(\eta \wedge \beta)$. In order to maintain compatibility with the
    usual exterior derivative, one also defines $\beta \wedge \zeta$ by
    $(-1)^{|\beta| \cdot |\zeta|} \zeta \wedge \beta$. This gives $\ccur{M}$
    the structure of a graded-symmetric dg-bimodule over $\cdiff{M}$. With the above
    definitions, the usual push-pull formula
    \begin{equation*}
      f_{*} \left( f^{*} \left( \zeta \right) \wedge \beta \right) =
      \zeta \wedge f_{*} \left( \beta \right)
    \end{equation*}
    holds also for currents as specified in Proposition~\ref{prop:proppp}\eqref{prop:pushpull}.
\end{enumerate}

We will also find it convenient to work with a modified complex of currents
where we ``kill'' a specific closed element. Let $M$ be a smooth $m$-dimensional
manifold and let $\zeta \in \ccur{M}[k]$ be a closed current of degree $k$.
Since $\zeta$ is closed, we have the following short exact sequence of
complexes:
\begin{equation*}
  0 \rightarrow \gen{\zeta} \hookrightarrow \ccur{M}[*]
  \xrightarrowdbl \ccur{M}[*] / \gen{\zeta} \rightarrow 0
\end{equation*}
Let us denote by $\ccur{M}[*][\zeta] := \ccur{M}[*] / \gen{\zeta}$ the quotient
complex and denote its cohomology by $\hcur{M}[*][\zeta]$. Note that for $j \neq
k-1,k$ we have $\hcur{M}[j][\zeta] = \hcur{M}[j]$. For $j\in \Set{k-1,k}$, the
short exact sequence of complexes gives us a long exact sequence in cohomology:
\begin{equation*}
  \dots \rightarrow 0 \rightarrow \hcur{M}[k-1] \rightarrow \hcur{M}[k-1][\zeta]
  \rightarrow \gen{\zeta} \rightarrow \hcur{M}[k] \rightarrow \hcur{M}[k][\zeta]
  \rightarrow 0 \rightarrow \dots
\end{equation*}
We have two cases:
\begin{enumerate}
  \item If $\zeta$ is exact, the inclusion map $\gen{\zeta} \rightarrow
    \hcur{M}[k]$ is the zero map and so $\hcur{M}[k][\zeta] \cong \hcur{M}[k]$
    and we get a short exact sequence
    \begin{equation*}
       0 \rightarrow \hcur{M}[k-1] \rightarrow \hcur{M}[k-1][\zeta]
       \rightarrow \gen{\zeta} \rightarrow 0.
    \end{equation*}
    Choosing a splitting for this sequence amounts to choosing a current $\mu
    \in \ccur{M}[k-1]$ with $d\mu = \zeta$ and then we get $\hcur{M}[k-1][\zeta]
    \cong \hcur{M}[k-1] \oplus \gen{\mu}$.
  \item If $\zeta$ is not exact, the inclusion map $\gen{\zeta} \rightarrow
    \hcur{M}[k]$ is injective and so $\hcur{M}[k-1][\zeta] \cong \hcur{M}[k-1]$
    and we get a short exact sequence
    \begin{equation*}
       0 \rightarrow \gen{\zeta} \rightarrow \hcur{M}[k] \rightarrow
       \hcur{M}[k][\zeta] \rightarrow 0
    \end{equation*}
    which shows that $\hcur{M}[k][\zeta] \cong \hcur{M}[k] / \gen{[\zeta]}$. In
    this case, $\hcur{M}[*][\zeta]$ is the same as
    $\hcur{M}[*]$, except the (non-trivial) cohomology class $\zeta$ is killed.
\end{enumerate}

Now, let $L \subseteq M$ be a smooth oriented submanifold of codimension $n$ and
denote by $i \colon L \rightarrow M$ the inclusion.
Let $1_L\in A^0(L;\R)$ be the constant zero form with value $1$.
The submanifold $L$
gives us a naturally associated current
\begin{equation*}
  \zeta_{L} (\alpha) = i_{*} \left( 1_L \right) (\alpha) =
  \int_{L} i^{*}(\alpha)
\end{equation*}
and in this case we will use the notation $\ccur{M}[k][\zeta_{L}] =
\ccur{M}[k][L]$ and $\hcur{M}[k][\zeta_{L}] = \hcur{M}[k][L]$. Consider
$i_*(\eqcl{L}) \in H_{m-n}(M;\R)$ the class of $L$ inside of $M$. We have two
cases
\begin{enumerate}
  \item If $i_{*}(\eqcl{L}) = 0$ (that is, $L$ is homologically trivial in $M$) we
    can choose a smooth singular chain $S$ with $\eqcl{\partial S} =
    i_{*}(\eqcl{L})$ and then
    \begin{equation*}
      \left( d \zeta_{S} \right)(\alpha) = (-1)^{n} \int_{S} d\alpha = (-1)^n
      \int_{\partial S} \alpha = (-1)^n \int_{L} i^{*}(\alpha)
    \end{equation*}
    so $d \zeta_{S} = (-1)^{n} \zeta_{L}$. Hence
    $\hcur{M}[n][L] \cong \hcur{M}[n]$ while $\hcur{M}[n-1][L] \cong
    \hcur{M}[n-1] \oplus \gen{\zeta_{S}}$.
  \item If $i_{*}(\eqcl{L}) \neq 0$ then by de Rham's theorem $\zeta_{L}$ is not
    exact and $\hcur{M}[n-1][L] \cong \hcur{M}[n-1]$ while $\hcur{M}[n][L] \cong
    \hcur{M}[n] / \gen{[\zeta_{L}]}$.
\end{enumerate}
The resulting complex $\ccur{M}[*][L]$ is dual to the complex
\begin{equation*}
  \ker \left( \zeta_{L} \right) =  \Set{\alpha \in \cdiff{M}}[\int_{L} i^{*}
  \alpha = 0] \subseteq \cdiff{M}
\end{equation*}
of differential forms on $M$ whose integral on $L$ vanishes.

\begin{rem}
In the special case when $L\subset X$ is a Lagrangian submanifold, the complex
$\ker(\zeta_L)$ is the one denoted, e.g., in~\cite{ST3} or~\cite{T19}, by $\widehat{A}^*(X,L)$, the cohomology of which was denoted by $\Hh^*(X,L;R)$. There, it was used as the complex from which interior constraints are taken for open Gromov-Witten invariants. Since $\ccur{M}[*][L]$ is dual to the complex $\ker(\zeta_L),$ it follows that $\hcur{X}[2n-*][L][] \simeq \Hh^*(X,L)^\vee.$ In the special case when $H^*(L;\R)\simeq H^*(S^n;\R)$, for $* \neq 0$ we get $\Hh^*(X,L)=H^*(X,L),$ the last expression being the standard relative cohomology. Thus, by Poincar\'e duality, for $* \neq n$ we have $\hcur{X}[*][L][] \simeq H_*(X\setminus L).$
The chain $S$ is the same chain used in~\cite{ST3} in combination with a bounding cochain to define enhanced open Gromov-Witten invariants, as explained in~\cite[Remark 4.12]{ST3}.
\end{rem}

\subsection{Integration}\label{ssec:int}

The following proposition is proved in Theorem~1 (for differential forms) and Proposition~6.1 (for currents) in~\cite{ST4}. They are concerned with integration properties of maps between smooth orbifolds with corners -- more specifically, with pull-back and push-forward properties for forms and currents.

\begin{prop}\label{prop:proppp}
Assume all maps below are relatively oriented.
\begin{enumerate}
	\item \label{normalization}
	Let $f:M\to pt$ and $\alpha\in A^m(M)\otimes R$. Then
	\[
	f_*\alpha=\begin{cases}
	\int_M\alpha,& m=\dim M,\\
	0,&\text{otherwise}.
	\end{cases}
	\]
	\item\label{prop:pushcomp}
		Let $g: P\to M$, $f:M\to N,$ be proper submersions and $\alpha\in A^*(P;R)$. Then
		\[
		\left(f_*\circ g_*\right)(\a)=(f\circ g)_*(\a).
		\]
The same formula hold for proper smooth maps $f,g,$ if $\a\in\A^*(P;R)$.
	\item\label{prop:pushpull}
		Let $f:M\to N$ be a proper submersion, $\alpha\in A^*(N;R),$ $\beta\in A^*(M;R)$. Then
		\[
		f_*(f^*\alpha\wedge\beta)=\alpha\wedge f_*\beta.
		\]d
The same formula holds if $f$ is proper but is not necessarily a submersion with $\a \in A^*(N;R)$ and $\beta \in \A^*(M;R)$, or if $f$ is a proper submersion with $\a\in \A^*(N;R)$ and $\beta\in A^*(M;R)$.
	\item\label{prop:pushfiberprod}
		Consider a pull-back diagram of smooth maps
		\[
		\xymatrix{
		{M\times_N P}\ar[r]^{\quad p}\ar[d]^{q}&
        {P}\ar[d]^{g}\\
        {M}\ar[r]^{f}&N\, ,
		}
		\]
		where $g$ is a proper submersion. Let $\alpha\in A^*(P;R).$ Then
		\[
		q_*p^*\alpha=f^*g_*\alpha.
		\]
The same holds for $\a\in \A^*(P;R)$ if $f$ is a submersion and $g$ is proper.
\end{enumerate}
\end{prop}

The following is proved in Theorem~1 (for differential forms) and Proposition~6.5 (for currents) of~\cite{ST4}. By abuse of notation, we identify a differential form with its image in currents without further comment.
If $f$ in the proposition below is not a submersion, the equality is to be understood in the space of currents dual to $A^*(N,\d N)=\{\eta \in A^*(N)\,|\,\eta|_{\d N}=0\}$.

\begin{prop}[Stokes' theorem]\label{stokes}
Let $f:M\to N$ be proper relatively oriented with $\dim M=s$, and let $\xi\in A^t(M;R)$. Then
\[
d (f_*\xi)=f_*(d \xi)+(-1)^{s+t}\big(f\big|_{\partial M}\big)_*\xi,
\]
where $\d M$ is understood as the fiberwise boundary with respect to $f$.
\end{prop}

\subsection{Closed maps}\label{ssec:qemptyset}

Let $(X,\omega)$ be a closed symplectic manifold, $J$ an almost complex structure, and $L$ a Lagrangian submanifold of $X$, satisfying the assumptions specified in Section~\ref{ssec:setting}.

For $\beta\in H_2(X;\Z)$, let $\M_{l+1}(\beta)$ be the moduli space of stable $J$-holomorphic spheres with $l+1$ marked points indexed from 0 to $l$ representing the class $\beta$, and let $ev_j=ev_j^\beta:\M_{l+1}(\beta)\to X$ be the evaluation maps. Assume that all the moduli spaces $\M_{l+1}(\beta)$ are smooth orbifolds and $ev_0$ is a submersion.

For a list $\glt=(\gamma_1,\ldots,\gamma_l)\in A^*(X;R)^{\times l}$, write for short
\[
ev^*\glt:=\bigwedge_{j=1}^l ev_j^*\gamma_j.
\]
For an ordered sublist $I\subset [l]:=(1,\ldots,l)$, write
\[
ev_I^*\glt := \bigwedge_{j\in I}ev_j^*\gamma_j.
\]
Let
\begin{equation}
\varpi: H_2(X;\Z)\to \sly
\end{equation}
be the natural map coming from the long exact sequence of the pair $(X,L)$.
Recall the relative spin structure $\s$ determines a class $w_{\s} \in H^2(X;\Z/2\Z)$ such that $w_2(TL) = i^* w_{\s}$. By abuse of notation we think of $w_{\s}$ as acting on $H_2(X;\Z)$.

As in~\cite{ST1}, we define operators
\[
\q_{\emptyset,l}^{\beta} \colon
\cdiff{X}^{\otimes l} \rightarrow \cdiff{X}
\]
by
\begin{equation*}
  \q_{\emptyset,l}^{\beta} \left( \gamma_1, \dots, \gamma_l \right) :=
(-1)^{\omega_{\s}(\beta)}
  \left(
  ev_0^{\beta} \right)_{*} \left( \bigwedge_{j=1}^l \left(
  ev_{j}^{\beta} \right)^{*} \gamma_j
  \right).
\end{equation*}
Recall also that we are working under the assumption that
$\mathcal{M}_{l+1}(\beta)$ is a smooth orbifold (without boundary) and
that the map $ev_0 \colon \mathcal{M}_{l+1}(\beta) \rightarrow X$ is a
smooth proper submersion.
The moduli spaces $\mathcal{M}_{l+1}(\beta)$ come with a natural orientation induced by the natural orientation of $X$. Therefore
we can relatively orient $ev_0$ to make it compatible with the
orientations on $X$ and $\mathcal{M}_{l+1}(\beta)$. Since the group of permutations $S_{l+1}$ acts on $\mathcal{M}_{l+1}(\beta)$ by orientation-preserving diffeomorphisms, our
assumption implies that all other evaluation maps $ev_i$ are
submersions and are also consistently oriented with the orientations on $X$ and
$\mathcal{M}_{l+1}(\beta)$. This implies that in the definition of
$\q_{\emptyset,l}^{\beta}$, we can replace \textbf{one} of the differential
forms $\gamma_i$ with a current on $X$ and obtain a current. We will always use
the first input as a current and obtain operators
\begin{equation*}
  \q_{\emptyset,l}^{\beta} \colon \ccur{X} \otimes \cdiff{X}^{\otimes(l-1)}
  \rightarrow \ccur{X}
\end{equation*}
defined by exactly the same formula. The consisent orientation guarantees that those
operators extend the usual operators on differential forms in the sense that
\begin{equation*}
  \q_{\emptyset,l}^{\beta} \left( \varphi_{X} \left( \gamma_1 \right), \gamma_2,
  \dots, \gamma_l \right) = \varphi_{X} \left( \q_{\emptyset,l}^{\beta} \left( \gamma_1,
  \dots, \gamma_l \right) \right).
\end{equation*}
Since $\mathcal{M}_{l+1}(\beta)$ has no boundary,
the operations $\q_{\emptyset,l}^{\beta}$ satisfy the property that
\begin{equation*}
  d \left( \q_{\emptyset,l} \left( \gamma_1, \dots, \gamma_l \right) \right) =
  \sum_{i=1}^{l} (-1)^{\sum_{j=1}^{i-1} \norm{\gamma_i}} \q_{\emptyset,l}
  \left( \gamma_1, \dots, \gamma_{i-1}, d\gamma_i, \gamma_{i+1},\dots,\gamma_l
  \right).
\end{equation*}
Since $X$ and $\mathcal{M}_{l+1}(\beta)$ have no boundary, one can see that this
property continues to hold even if one of the inputs $\gamma_i$ is a current -- here we use the sign conventions specified in Section~\ref{ssec:currents} and Proposition~\ref{stokes}.
The previous observations extend to the operators
\begin{equation*}
  \q_{\emptyset,l} \left( \gamma_1, \dots, \gamma_l \right) := \sum_{\beta \in
  H_2(X;\ZZ)} T^{\varpi(\beta)} \q_{\emptyset,l}^{\beta} \left(
  \gamma_1, \dots, \gamma_l \right).
\end{equation*}
Now, consider the operator $\q_{\emptyset,1}^{\gamma} \colon \cdiff{X}[*][][R]
\rightarrow \cdiff{X}[*][][R]$ given by
\begin{equation}\label{eq:defq}
  \q_{\emptyset,1}^{\gamma} \left( \alpha \right) := \sum_{l \geq 0}
  \frac{1}{l!} \q_{\emptyset,l+1} \left( \alpha, \gamma^{\otimes l} \right).
\end{equation}
Let's assume $\gamma$ has degree two and is closed. Then
$\q_{\emptyset,1}^{\gamma}$ is a homogeneous map of degree $2$ that satisfies $d
\left( \q_{\emptyset,1}^{\gamma} \left( \alpha \right) \right) =
\q_{\emptyset,1}^{\gamma} \left( d \alpha \right)$.
Again, as long as $\gamma$ is a differential form, we can extend this map to
allow $\alpha$ to be a current and obtain a map $\q_{\emptyset,1}^{\gamma}
\colon \ccur{X}[*][][R] \rightarrow \ccur{X}[*+2][][R]$ which extends the previous
map on differential forms and also satisfies $d \left( \q_{\emptyset,1}^{\gamma}
\left( \zeta \right) \right) = \q_{\emptyset,1}^{\gamma} \left( d\zeta \right)$
for all currents $\zeta \in \ccur{X}[*][][R]$.

\subsection{Auxiliary notation}\label{ssec:not}

\subsubsection{Grading}

Let $\Upsilon$ be a graded-commutative algebra with grading $|\cdot|$. Denote by $\V \cdot \V$ the grading on the shifted module $\Upsilon[1]$, that is,
\[
\Vert\a\Vert:=|\a|+1,\qquad \forall \a\in \Upsilon.
\]

\subsubsection{Permutations and signs}
For a list $\lst=(\a_1,\ldots,\a_k)\in \Upsilon^{\times k}$ and a permutation $\sigma\in S_{k}$, define the weighted permutation sign by
\begin{equation}\label{eq:sgnsigmagamma}
s_\sigma(\lst):=
\sum_{\substack{i>j\\ \sigma^{-1}(i)<\sigma^{-1}(j)}}|\a_i|\cdot|\a_j|
=
\sum_{\substack{i<j\\ \sigma(i)>\sigma(j)}}|\a_{\sigma(i)}|\cdot |\a_{\sigma(j)}|\pmod 2
\end{equation}
and the shifted weighted permutation sign by
\begin{equation}\label{eq:ssig}
s_\sigma^{[1]}(\lst):
=\sum_{\substack{i<j\\ \sigma(i)>\sigma(j)}}
(|\a_{\sigma(i)}|+1) (|\a_{\sigma(j)}|+1)
=\sum_{\substack{i<j\\ \sigma(i)>\sigma(j)}}
\V\a_{\sigma(i)}\V \V\a_{\sigma(j)}\V \pmod 2.
\end{equation}
For for a list $\alt=(\alpha_1,\ldots,\alpha_k)$ and a permutation $\sigma\in S_k$, denote by $\alt^\sigma$ the $\sigma$-reordered list:
\begin{equation}\label{eq:as}
\lst^\sigma:=(\alpha_{\sigma(1)},\ldots,\alpha_{\sigma(k)}).
\end{equation}
We will typically use this notation specifically with cyclic permutations $\sigma\in \Z/k\Z \subset S_k$.

For $I\sqcup J$ a (non-ordered) splitting of
\[
[l]:=(1,2,\ldots,l)
\]
with the induced order on $I,J,$ let $\sigma_{I\sqcup J}\in S_l$ be the permutation that reorders the concatenation $I\circ J$ back into $[l]$. In particular, for a list $\gamma=(\gamma_1,\ldots,\gamma_l)$ of differential forms, we get
\[
\bigwedge_{i \in I} \gamma_i \wedge \bigwedge_{j \in J} \gamma_j = (-1)^{s_{\sigma_{I\cup J}}(\gamma)}\bigwedge_{k \in [l]} \gamma_k,
\]
where the wedge products are taken in the order of the respective lists.

\subsubsection{Sublists and splittings}
Recall that in Section~\ref{ssec:notlist} we established the notation of round brackets to describe summation over all possible ordered splittings. For example, summation over all $3$-splittings of $\a$ looks like
\[
\a_{(1)}\otimes \a_{(2)}\otimes \a_{(3)}.
\]
Further splitting $\a_{(1)}$ into $2$ sublists is written as
$
\a_{(11)}\otimes \a_{(12)}
$
and results in a $4$-splitting
\[
\a_{(11)}\otimes \a_{(12)}\otimes \a_{(2)} \otimes \a_{(3)} = \a_{(1)} \otimes \a_{(2)} \otimes \a_{(3)} \otimes \a_{(4)}
\] of $\a$.
For a \textit{specific}, though unspecified, ordered splitting, we use angle
brackets. Thus, $\a_{(1)}\otimes \a_{(2)}\otimes \a_{(3)}$ is a sum of terms
of the form
\[
\a_{\ngl{1}}\otimes \a_{\ngl{2}}\otimes \a_{\ngl{3}}.
\]
When we need to describe the terms explicitly, we use the following notation.
For a sublist of the form $(i,i+1,\ldots,j)$ inside $[k]$, write
\[
\a_{[i:j]}:=\a_{i}\otimes \a_{i+1}\otimes \cdots \otimes\a_{j}.
\]
In particular, we can now write a specific $3$-splitting in two ways:
\[
\a_{[1:i-1]}\otimes \a_{[i:j]}\otimes \a_{[j+1:k]}
= \a_{\ngl{1}}\otimes\a_{\ngl{2}}\otimes\a_{\ngl{3}}.
\]

\subsection{Closed-open maps}\label{ssec:q}

Denote by $\M_{k+1,l}(\beta) = \M_{k+1,l}(\beta;J)$ the moduli space of $J$-holomorphic genus zero open stable maps to $(X,L)$ of degree $\beta$ with one boundary component, $k+1$ boundary marked points, and $l$ internal marked points.
Denote by
\begin{gather*}
evb_j^\beta:\M_{k+1,l}(\beta)\to L, \qquad  \qquad j=0,\ldots,k, \\
evi_j^\beta:\M_{k+1,l}(\beta) \to X, \qquad \qquad j=1,\ldots,l,
\end{gather*}
the evaluation maps given by $evb_j^\beta((\Sigma,u,\vec{z},\vec{w}))=u(z_j)$ and $evi_j^\beta((\Sigma,u,\vec{z},\vec{w}))= u(w_j).$
We may omit the superscript $\beta$ when the omission does not create ambiguity.

For lists $\alt=(\a_1,\ldots,\a_k)\in A^*(L;R)^{\times k}$ and $\glt=(\gamma_1,\ldots,\gamma_l)\in A^*(X;R)^{\times l}$, write for short
\begin{equation}\label{eq:shorthand}
evb^*\alt:=\bigwedge_{j=1}^k evb_j^*\a_j,
\quad
evi^*\glt:=\bigwedge_{j=1}^l evi_j^*\gamma_j.
\end{equation}
For permutations $\sigma\in S_k$ and $\tau\in S_l$, write
\begin{equation}\label{eq:shortsig}
evb_\sigma^*\alt:=\bigwedge_{j=1}^k evb_{\sigma(j)}^*\a_j,
\quad
evi_\tau^*\glt:=\bigwedge_{j=1}^l evi_{\tau(j)}^*\gamma_j.
\end{equation}
In particular,
\[
evb^*\alt^{\sigma}=(-1)^{s_\sigma(\alt)}evb_{\sigma^{-1}}^*\alt,
\quad
evi^*\glt^{\tau}=(-1)^{s_\tau(\glt)}evi_{\tau^{-1}}^*\glt.
\]
For ordered sublists $[i:k']:=(i,i+1,\ldots,k')\subset[k]$ and $I\subset [l]$, write
\[
evb_{[i;k']}^*\a:=\bigwedge_{j\in [i;k']}evb_j^*\a_j,
\quad
evi_I^*\glt := \bigwedge_{j\in I}evi_j^*\gamma_j.
\]
In particular,
\[
evb^*\a = evb_{[1:i-1]}^*\a \wedge evb_{[i:k']}^*\a \wedge evb_{[k'+1:k]}^*\a,
\quad
evi^*\glt = (-1)^{s_{\sigma_{I\cup J}}(\glt)} evi_I^*\glt\wedge evi_J^*\glt.
\]

For a list $a=(a_1,\ldots,a_k)\in \Z_{\ge 0}^{\times k},$ define
\[
\varepsilon(a):=1 + \sum_{j=1}^kj(a_j+1).
\]
To simplify notation in the following, we allow differential forms as input, in lieu of their degrees.
In particular, for a list $\alt =(\a_1,\ldots, \a_k)\in C^{\times k}$,
\[
\varepsilon(\alt)=1+ \sum_{j=1}^kj\V\alpha_j\V.
\]

For all $\beta\in\sly$, $k,l\ge 0$,  $(k,l,\beta) \not\in\{ (1,0,\beta_0),(0,0,\beta_0)\}$, define
\[
\q_{k,l}^\beta:C^{\otimes k}\otimes A^*(X;R)^{\otimes l} \lrarr C
\]
by
\begin{align*}
\q^{\beta}_{k,l}(\alt;\glt):=
(-1)^{\varepsilon(\alt)}
(evb_0)_* \left(evi^*\glt\wedge evb^*\alt\right).
\end{align*}
The case $\q_{0,0}^{\beta}$ is understood as $-(evb_0^{\beta})_*1.$
Define
\[
\q_{1,0}^{\beta_0}(\alpha):=d\alpha,\quad \q_{0,0}^{\beta_0}:=0.
\]
Set
\begin{align*}
\q_{k,l}:=\sum_{\beta\in\sly}T^{\beta}\q_{k,l}^{\beta}.
\end{align*}
Let $\gamma \in \mI_R A^*(X;R)$ such that $|\gamma|=2$ and $d\gamma=0$. Let $b\in \mI_R A^*(L;R)$ such that $|b|=1$.
For $k\ge 0,$ define
\begin{multline}\label{eq:defqempt}
\qbg_{k,l}(\alpha_1,\ldots,\alpha_k;\delta_1,\ldots,\delta_l):=\\
\sum_{s,t}\frac{1}{(t-l)!}\sum_{\substack{1\le i_1<\cdots\;\\ \;\cdots<i_k\le s}}\!
\q_{s,t}(b^{\otimes i_1-1}\otimes\alpha_1\otimes b^{\otimes i_2-i_1-1}\otimes \cdots\otimes b^{\otimes i_k-i_{k-1}-1}\otimes\alpha_k\otimes b^{\otimes s-i_k};\otimes_{j=1}^l\delta_j\otimes\gamma^{\otimes t-l}).
\end{multline}
These operators give rise to the $A_\infty$ structure on $C$ mentioned in the introduction. Namely,
\[
\mg_k=\q_{k,0}^{0,\gamma}: C^{\otimes k}\lrarr C
\]
satisfy
\begin{equation}\label{eq:ainfty}
\sum_{\substack{k_1+k_2=k+1\\ 1\le i\le k_1}}
(-1)^{\sum_{j=1}^{i-1}\V\a_j\V}
\mg_{k_1}(\otimes_{j=1}^{i-1}\a_j\otimes \mg_{k_2}(\otimes_{j=i}^{i+k_2-1}\a_j)\otimes \otimes_{j=i+k_2}^k\a_j) = 0
\end{equation}
for any list $\a_1,\ldots,\a_k\in C$. Equation~\eqref{eq:ainfty} is called the $A_\infty$ relations and is proved, e.g., in~\cite{ST1}. The proof of Proposition~\ref{prop:pstructure} below is analogous to that.

\section{Open-closed maps on the de-Rham complex}\label{sec:p}
We define here operators similar to $\q$, with the difference that the output point lies in the interior of the disk instead of its boundary. Then, we verify properties satisfied by the resulting operators.

\subsection{Structure}

Relabel the marked points on the space $\M_{k,l+1}(\beta)$ to get
evaluation maps $evb_j^\beta:\M_{k,l+1}(\beta)\to L$, $j=1,\ldots,k$, and $evi_j^\beta:\M_{k,l+1}(\beta)\to X$, $j=0,\ldots,l$.
We do not necessarily assume that $evi_0$ is a submersion --- in fact, typically it would be impossible for $evi_0$ to be a submersion. For example, in the case $\beta = \beta_0$ where the stable maps in $\M_{k,l+1}(\beta)$ are constant, $evi_0$ maps to $L \subset X.$ Therefore, there is no way to interpret $(evi_0)_*$ as an operation between complexes of differential forms (not without virtual techniques). Rather, the output of $(evi_0)_*$ is to be understood as a current.

For a list $a=(a_1,\ldots,a_k)\in \Z_{\ge 0}^{\times k},$ define
\[
\varepsilon_p(a):=1 + \sum_{j=1}^k(n+j)(a_j+1).
\]
To simplify notation in the following, we allow differential forms as input, in lieu of their degrees.
In particular, for a list $\alt =(\a_1,\ldots, \a_k)\in C^{\times k}$,
\[
\varepsilon_p(\alt)=1+ \sum_{j=1}^k (n+j)\V\alpha_j\V.
\]

For all $\beta\in\sly$, $k,l\ge 0$,  $(k,l,\beta) \ne (0,0,\beta_0)$, define
\[
\p_{k,l}^\beta:A^*(L;R)^{\otimes k}\otimes A^*(X;R)^{\otimes l} \lrarr \ccur{X}[*][][R],
\]
for $\alt=(\a_1,\ldots, \a_k), \glt
=(\gamma_1,\ldots, \gamma_l)$, by
\begin{align*}
\p^{\beta}_{k,l}(\alt;\glt):=
(-1)^{\varepsilon_p(\alt)}
(evi_0^\beta)_* \left(evi^*\glt\wedge evb^*\alt\right).
\end{align*}
The presence of $n$ in $\varepsilon_p(\alpha)$ (contrary to $\varepsilon(\alpha)$) is motivated by the presence of $n$ in the degree of the push-forward along $evi_0$, and thus in the degree of the operator $\p^\beta_{k,l}$, as calculated in Proposition~\ref{prop:p_deg}.

Define also $\p_{0,0}^{\beta_0}=0$
and set
\begin{align*}
\pkl:=\sum_{\beta\in\sly}T^{\beta}\pkl^{\beta}.
\end{align*}

For a list $\alt=(\alpha_1,\ldots,\alpha_k)$ and a cyclic permutation $\sigma\in \Z/k\Z$,
recall the notation of $\alt^\sigma$, $s_\sigma(\alt)$, $s_\sigma^{[1]}(\alt)$, and $s_{\sigma_{I\cup J}}(\alt)$ from Section~\ref{ssec:not}.
Additionally, if $\alt$ is a list of differential forms, abbreviate
\[
d(\a):=\sum_{j=1}^k (-1)^{|\a_{[1:j-1]}|}\a_1\otimes \cdots\otimes \a_{j-1}\otimes d\a_j\otimes\a_{j+1}\otimes\cdots\otimes\a_k.
\]

The rest of the section is devoted to the proof of the following result.

\begin{prop}\label{prop:pstructure}
Consider lists $\alt=(\alpha_1,\ldots,\alpha_k),$ $\alpha_j\in A^*(L;R)$, and $\glt=(\gamma_1,\ldots,\gamma_l),$ $\gamma_j\in A^*(X;R).$
Then
\begin{multline}\label{eq:p_rel}
d\p_{k,l}(\alt;\glt)
=
\p_{k,l}(\alt; d(\glt)) \\
\quad+
\sum_{\substack{I\sqcup J=[l]\\ \sigma\in\Z/k\Z}}
(-1)^{s_\sigma^{[1]}(\alt)+|\glt| +s_{\sigma_{I\sqcup J}}(\glt)+(n+1)(|\glt_{J}|+1)}
\p_{k_1,|I|}(\q_{k_2,|J|} ((\alt^\sigma)_{\rgl{1}};\glt_J) \otimes(\alt^\sigma)_{\rgl{2}}; \glt_{I})\\
+
\delta_{k,0}\cdot(-1)^{|\gamma|} \q_{\emptyset,l+1}(\glt \otimes i_*1_L).\notag
\end{multline}
\end{prop}

Roughly speaking, Proposition~\ref{prop:pstructure} describes the boundary of the moduli space of disks with boundary constraints in $\alpha$, interior constraints in $\gamma$, and an interior output point. The different types of boundary components are specified in the different summands on the right-hand side, as follows.
The first summand describes boundary components where one of the interior constraints (in $\gamma$) goes to its boundary (i.e., we get a constraint in $d\gamma_j$ for one of the $j$). The case where one of the boundary constraints goes to its boundary (i.e., where we get a constraint of the form $d\alpha_j$ for one of the $j$) is included in the second summand, in view of the definition $\q_{1,0}^{\beta_0}=d$.
The rest of the second summand (when $\q$ comes with degree other than $\beta_0$) describes boundary components that arise from a disk bubbling off at the boundary, with some of the boundary constraints ($(\alpha^\sigma)_{(1)}$) and some of the interior constraints ($\gamma_J$) on it.
The third summand is only relevant to the case with no boundary constraints, that is, when $k=0$; then another boundary component appears from the degeneration of the boundary of the disk to a point. In the presence of boundary marked points we would consider such a degeneration an instance of interior bubbling (from a ghost disk), which is a configuration of condimension $2$.

\smallskip
We defer the proof of Proposition~\ref{prop:pstructure} to Section~\ref{sec:str_pf} below, with the exception of the following lemma that will also be used in Section~\ref{ssec:p_prop}.

\begin{lm}\label{lm:epsilonsigma}
For $\alt=(\a_1,\ldots,\a_k)$ and $\sigma\in \Z/k\Z$, we have
\[
\varepsilon_p(\alt^\sigma)- \varepsilon_p(\alt)
=
\varepsilon(\alt^\sigma)- \varepsilon(\alt)
=
\sum_{\substack{j<i\\ \sigma(j)>\sigma(i)}}(|\alpha_{\sigma(i)}|-|\alpha_{\sigma(j)}|).
\]
\end{lm}

\begin{proof}
The first identity follows from $\V\alt^\sigma\V=\V\lst\V$. For the second identity, compute
\begin{align*}
\varepsilon(\alt^\sigma)&-\varepsilon(\alt)=
\sum_{j=1}^kj(|\alpha_{\sigma(j)}|-|\alpha_j|)\\
&=\sum_{j=1}^kj|\alpha_{\sigma(j)}|-\sum_{m=1}^k\sigma(m)|\alpha_{\sigma(m)}|\\
&=\sum_{j=1}^k|\alpha_{\sigma(j)}|(j-\sigma(j)).
\end{align*}
Let $t\in\{1,\ldots,k\}$ such that $t\equiv \sigma(1)-1\pmod k$. Since $\sigma$ is cyclic, it follows that for all $j$ we have
$\sigma(j)\equiv j+t\pmod k$. More specifically, for $j\le k-t,$ we have $\sigma(j)=j+t$, and thus $j-\sigma(j)=-t$; for $j\ge k-t+1,$ we have $\sigma(j)=j+t-k,$ and thus $j-\sigma(j)=k-t$.
Therefore, the above expression equals
\[
 \sum_{j=1}^k|\alpha_{\sigma(j)}|(j-\sigma(j))=
(k-t)\sum_{j=k-t+1}^k |\alpha_{\sigma(j)}|-t\cdot\sum_{j=1}^{k-t}|\alpha_{\sigma(j)}|\\
=\sum_{j=1}^{k-t}\sum_{i=k-t+1}^k(|\alpha_{\sigma(i)}|-|\alpha_{\sigma(j)}|).
\]
Note that
\[
\{j<i \mbox{ and } \sigma(j)>\sigma(i)\} \iff \{j\le k-t \mbox{ and } i\ge k-t+1\}.
\]
Therefore,
\[
\sum_{j=1}^{k-t}\sum_{i=k-t+1}^k(|\alpha_{\sigma(i)}|-|\alpha_{\sigma(j)}|)
= \sum_{\substack{j<i\\ \sigma(j)>\sigma(i)}}(|\alpha_{\sigma(i)}|-|\alpha_{\sigma(j)}|).
\]

\end{proof}

\subsection{Properties}\label{ssec:p_prop}

The $\p$ operators can be shown to behave very similarly to the $\q$ operators in many ways.

\subsubsection{Linearity}
\begin{prop}\label{prop:p_lin}
The $\p$ operators are $R$-multilinear in the sense that for all $a \in R$, $\alt=(\a_1,\ldots,\a_k)\in A^*(L;R)^{\otimes k}$, and $\glt=(\gamma_1,\ldots,\gamma_l)\in A^*(X;R)^{\otimes l}$, we have
\[
\qquad\p_{k,l}^\beta(\a_1,\ldots,\a_{i-1},a\cdot\a_i,\ldots,\a_k; \glt)
		=(-1)^{|a|\cdot \big(n+1+\V \a_{[1:i-1]}\V+|\glt|\big)}
		a\cdot\p_{k,l}^\beta(\alt;\glt),
\]
and if $a \in R$ we have
\[
\p_{k,l}^\beta(\alt;\gamma_1,\ldots,a\cdot\gamma_i,\ldots,\gamma_l) =(-1)^{|a|\cdot|\gamma_{[1:i-1]}|}
		a\cdot\p_{k,l}^\beta(\alt;\glt).
\]
\end{prop}

\begin{proof}

For the first identity, consider
\begin{multline*}
(evi_0)_*(evi^*\glt\wedge evb_{[1:i-1]}^*\alt\wedge evb_i^*(a\a_i)\wedge evb_{[i+1:k]}^*\alt)=\\
=
(-1)^{|a|(|\glt|+|\alt_{[1:i-1]}|)}
a\cdot(evi_0)_*(evi^*\glt\wedge evb^*\alt).
\end{multline*}
The corresponding change in $\varepsilon_p$ is
\[
\varepsilon_p(\a_1,\ldots,\a_{i-1},a\a_i,\a_{i+1},\ldots ,\a_k)
-\varepsilon_p(\a_1,\ldots,\a_k)
=(i+n)\cdot |a|.
\]
Together, this gives the required result.

Similarly, for the second identity,
\begin{multline*}
(evi_0)_*(evi_{[1:i-1]}^*\glt\wedge evi_i^*(a\gamma_i)\wedge evi_{[i+1:l]}^*\glt\wedge evb^*\alt)=\\
=
(-1)^{|a|\cdot |\glt_{[1:i-1]}|}
a\cdot(evb_0)_*(evi^*\glt\wedge evb^*\alt),
\end{multline*}
while $\varepsilon_p$ is not affected.
\end{proof}

\subsubsection{Unit}
\begin{prop}\label{prop:p_unit}
For $\a_1,\ldots,\a_k\in A^*(L;R)$ and $\glt=(\gamma_1,\ldots,\gamma_l)\in A^*(X;R)^{\otimes l}$,
\begin{equation*} \p_{k+1,l}^{\beta}(\alpha_1,\ldots,\alpha_{i-1},1_L,\alpha_{i},\ldots,\alpha_k ;\glt)=
\begin{cases}
0,&(k+1,l,\beta)\ne (1,0,\beta_0),\\
(-1)^{n+1}i_*1_L,&(k+1,l,\beta)= (1,0,\beta_0).
\end{cases}
		\end{equation*}
\end{prop}
\begin{proof}

Roughly speaking, when one of the inputs is fixed in the unit class, the operation vanishes because it factors through a moduli space of strictly lower dimension. More rigorously:

For $(k+1,l+1,\beta)\ne(1,1,\beta_0)$, consider $\pi_i:\M_{k+1,l+1}(\beta)\to\M_{k,l+1}(\beta)$, the map that forgets the $i$-th boundary marked point, shifts the degrees of the following points, and stabilizes the result. The case $(k+1,l+1,\beta)=(1,1,\beta_0)$ is exactly when stabilization is impossible; in all other cases $\pi$ is well defined.
Then, the argument in~\cite[Proposition 3.2]{ST1} with $evi_0$ instead of $evb_0$ shows that $\p_{k+1,l}^{\beta}(\alpha_1,\ldots,\alpha_{i-1},1,\alpha_{i},\ldots,\alpha_k ;\glt)=0$.

It remains to analyze the exceptional case $(k+1,l+1,\beta)=(1,1,\beta_0)$.
In this case, the evaluation $ev$ identifies $\M_{1,1}(\beta_0)$ with $L$, preserving orientation, and  $evi_0=i\circ ev$, $evb_1=ev$. So,
\[
\p^{\beta_0}_{1,0}(1_L)=(-1)^{(n+1)\cdot 1}\cdot i_*ev_*ev^*1_L=(-1)^{n+1}i_*1_L.
\]
\end{proof}

\subsubsection{Degree}
\begin{prop}\label{prop:p_deg}
For $k\ge 0$ and $\glt=(\gamma_1,\ldots,\gamma_l)$ with $|\gamma_j|=2$ for all $j$,
the map
\[
\pkl(\; ;\glt):C^{\otimes k}\lrarr \A^*(X;R)
\]
is of degree $n+1-k$.
\end{prop}

\begin{proof}
It is enough to check that, for any $\beta$, the map
\[
T^{\beta}\pkl^\beta(\; ;\glt):C^{\otimes k}\lrarr \A^*(X;R)
\]
is of degree $n+1-k$.

Recalling the degree of push-forward~\eqref{eq:pfc}, for any $\alt=(\a_1,\ldots,\a_k)$ compute
\begin{align*}
|T^{\beta}\pkl^{\beta}(\alt;\glt)|
={}& \mu(\beta)+|\alt|+|\gamma|+\rdim(evi_0)\\
={}& \mu(\beta)+|\alt|+2l+(2n-(n-3+\mu(\beta)+k +2l+2))\\
={}& |\alt|+n+1-k.
\end{align*}

\end{proof}

\subsubsection{Symmetry}
\begin{prop}[Cyclic symmetry of boundary input]\label{prop:p_sgn}
For a cyclic permutation $\sigma\in\Z/k\Z$, $\alt=(\alpha_1,\ldots,\alpha_k)$, and $\glt=(\gamma_1,\ldots,\gamma_l)$, we have
\[
\pkl(\alt;\glt)=
(-1)^{s_\sigma^{[1]}(\alt)}\pkl(\alt^\sigma;\glt)
\]
with $s_\sigma^{[1]}(\alt)$ as defined before Proposition~\ref{prop:pstructure}.
\end{prop}
\begin{proof}
By Lemma~\ref{lm:epsilonsigma},
\begin{align*}
\pkl^\beta(\alt^\sigma;\glt)&=
(-1)^{\varepsilon_p(\alt^\sigma)}
(evi_0^\beta)_* \left((evi^\beta)^*\glt\wedge  (evb^\beta)^*\alt^{\sigma}\right)\\
&=(-1)^{\varepsilon_p(\alt^\sigma)+\sum_{\substack{i<j\\ \sigma(i)>\sigma(j)}}|\a_{\sigma(i)}||\a_{\sigma(j)}|}
(evi_0^\beta)_* \left((evi^\beta)^*\glt\wedge  (evb_{\sigma^{-1}}^\beta)^*\alt\right)\\
&=(-1)^{\varepsilon_p(\alt)+\sum_{\substack{i<j\\ \sigma(i)>\sigma(j)}}(|\a_{\sigma(i)}||\a_{\sigma(j)}|+|\a_{\sigma(j)}|-|\a_{\sigma(i)}|)-sgn(\sigma)}
(evi_0^\beta)_* \left((evi^\beta)^*\glt\wedge (evb^\beta)^*\alt\right)\\
&=(-1)^{\varepsilon_p(\alt)+s_\sigma^{[1]}(\alt)}
(evi_0^\beta)_* \left((evi^\beta)^*\glt\wedge  (evb^\beta)^*\alt\right)\\
&=(-1)^{s_\sigma^{[1]}(\alt)}\pkl^\beta(\alt;\glt).
\end{align*}
\end{proof}

\begin{prop}[Symmetry of interior input]\label{prop:p_int_sgn}
For any permutation $\sigma\in S_k$ and any $\alt=(\alpha_1,\ldots,\alpha_k)$ and $\glt=(\gamma_1,\ldots,\gamma_l)$, we have
\[
\p_{k,l}(\alt;\glt)=(-1)^{s_\sigma(\gamma)} \p_{k,l}(\alt;\glt^\sigma).
\]
\end{prop}

\begin{proof}
Note that $\varepsilon_p$ does not depend on $\glt$, and permutation on the labels of interior marked points is an orientation preserving diffeomorphism of the moduli space. Therefore, a calculation analogous to the one in the proof of Proposition~\ref{prop:p_sgn} gives
\begin{align*}
\pkl^\beta(\alt;\glt^\sigma)&=
(-1)^{\varepsilon_p(\alt)}
(evi_0^\beta)_* \left((evi^\beta)^*\glt^{\sigma}\wedge (evb^\beta)^*\alt\right)\\
&=(-1)^{\varepsilon_p(\alt)+\sum_{\substack{i<j\\ \sigma(i)>\sigma(j)}}|\gamma_{\sigma(i)}||\gamma_{\sigma(j)}|}
(evi_0^\beta)_* \left((evi_{\sigma^{-1}}^\beta)^*\gamma \wedge (evb^\beta)^*\alt\right)\\
&=(-1)^{s_\sigma(\glt)}\pkl^\beta(\alt;\glt).
\end{align*}
\end{proof}

\subsubsection{Energy zero}
\begin{prop}\label{prop:p_zero}
For any $\alt=(\a_1,\ldots,\a_k)$, $\glt=(\gamma_1,\ldots,\gamma_l)$, we have
\[
\p_{k,l}^{\beta_0}(\alt;\glt)=
\begin{cases}
0,&(k,l)\ne (1,0),\\
(-1)^{(n+1)\V\a_1\V}i_*\alpha_1, &(k,l)=(1,0).
\end{cases}
\]
\end{prop}
	
\begin{proof}
For $\beta=\beta_0$ we have $evi_0=evi_j=i\circ evb_m$ for all $j,m.$ When $k\ge 1,$
\begin{align*}
\p^{\beta_0}_{k,l}(\alt;\glt)
&=	(-1)^{\varepsilon_p(\alt)} (evi_0)_*(evi_0^*(\wedge_{j=1}^l\gamma_j) \wedge evb_1^*(\wedge_{j=1}^k\alpha_k))\\	
&=
(-1)^{\varepsilon_p(\alt)}
i_*(evb_1)_*(evb_1^*(\wedge_{j=1}^li^*\gamma_j \wedge \wedge_{j=1}^k\alpha_k))\\
&=
(-1)^{\varepsilon_p(\alt)}i_*(\wedge_{j=1}^li^*\gamma_j \wedge \wedge_{j=1}^k\alpha_k\wedge(evb_1)_*1).
\end{align*}
For this to be nonzero, $\rdim(evb_1)$ has to be $0$. This happens if and only if
\[
0=n-3+\mu(\beta_0)+k+2l+2-n=k+2l-1,
\]
which implies $(k,l)=(1,0)$, $\M_{k,l+1}(\beta)=\M_{1,1}(\beta_0),$ and $\p^{\beta_0}_{1,0}(\alpha_1)
=(-1)^{(n+1)\V\a_1\V}i_*\alpha_1.$
		
When $k=0,$ all $j=0,\ldots,l,$ satisfy that $evi_j=i\circ ev$, where $ev:\M_{0,l+1}(\beta_0)\to L$ takes each map to the point that is its image. Then 		
\[
\p^{\beta_0}_{0,l}(\glt)=  (evi_0)_*(\wedge_{j=1}^levi_j^*\gamma_j)
		= i_*ev_*(ev^*(\wedge_{j=1}^li^*\gamma_j))
		= i_*\left(\wedge_{j=1}^li^*\gamma_j\wedge ev_*1\right).
\]
In order for this to be nonzero, it is necessary that
\[
0=\rdim(ev)=n-3+\mu(\beta_0)+2l+2-n=2l-1,
\]
which is impossible, $l$ being integer.
	
\end{proof}

\subsubsection{Fundamental class}
\leavevmode
Denote by $1_{X}\in A^0(X)$ the constant function with value $1$.
\begin{prop}\label{cl:pfund}
For all  $\alt=(\a_1,\ldots,\a_k)$, $\glt=(\gamma_1,\ldots,\gamma_{l-1})$, we have
\[
\pkl(\alt;1_X,\gamma_1,\ldots,\gamma_{l-1})=0.
\]
\end{prop}

\begin{proof}
As with Proposition~\ref{prop:p_unit},
whenever the forgetful map $\pi:\M_{k,l+2}(\beta)\to\M_{k,l+1}(\beta)$ is defined, we have $\pkl(\alt;1_X,\gamma_1,\ldots,\gamma_{l-1})=0$.

The only case when $\pi$ is not defined is $(k,l,\beta)=(0,0,\beta_0)$. But $\p_{0,1}^{\beta_0}\equiv 0$
by the zero energy property, Proposition~\ref{prop:p_zero}.
\end{proof}

\subsubsection{Divisor}
\begin{prop}\label{cl:pdiv}
Let $\alt=(\a_1,\ldots,\a_k)$ and $\glt=(\gamma_1,\ldots,\gamma_{l-1})$.
Assume $\gamma'\in A^2(X,L)\otimes R$ and $d\gamma' = 0$. Then
	\begin{equation}
	\pkl^{\beta}(\alt;\gamma' \otimes \glt)=
	\left(\int_\beta\gamma'\right) \cdot\p_{k,l-1}^{\beta}(\alt;\glt).
	\end{equation}
\end{prop}

The proof is the same as that of~\cite[Proposition 3.9]{ST1}, with $evi_0$ instead of $evb_0$.

\subsubsection{Top degree}\label{sssec:topdeg}

Given a homogeneous current $\alpha$ (or, as a special case, a differential form) with coefficients in $R$, denote by $\deg^d(\alpha)$ the degree of the current, ignoring the grading of $R$.
That is, for $\alpha =T^\beta t_1^{r_1}\cdots t_N^{r_N}\alpha'$ with $\alpha'\in A^j(L),$ we have $\deg^d(\alpha)=j.$

\begin{prop}\label{prop:p_no_top_deg}
Suppose $(k,l+1,\beta)\ne (1,1,\beta_0)$. Then, for all lists $\alt=(\a_1,\ldots,\a_k)$, $\glt=(\gamma_1,\ldots,\gamma_l)$, we have $\pkl^\beta(\alt;\glt)\in \A^{<2n}(X)\otimes R$.
\end{prop}
\begin{proof}
In the case $(k,l+1,\beta)=(2,1,0)$, the Energy Zero property, Proposition~\ref{prop:p_zero}, gives $\p_{2,0}^{\beta_0}(\alt)=0$. Thus, assume $(k,l+1,\beta)\not\in  \{(1,1,\beta_0), (2,1,\beta_0)\}$.

Assume without loss of generality that $\pkl^\beta(\alt;\glt)$ is homogeneous with respect to the grading $\deg^d.$ Let $evb_j^{l+1},evi_j^{l+1},$ be the evaluation maps for $\M_{k,l+1}(\beta).$ Write
\[
\xi:=(evi^{l+1})^*\glt\wedge(evb^{l+1})^*\alt,
\]
that is, $\pkl^\beta(\alt;\glt) =(-1)^{\varepsilon_p(\alt)}(evi_0^{l+1})_*\xi$. If
$
\deg^d(\pkl^\beta(\alt;\glt))=2n,
$
then
\[
2n=\deg^d(\xi)+\rdim (evi_0)
=\deg^d(\xi)+2n-\dim\M_{k,l+1}(\beta),
\]
so $\deg^d(\xi)=\dim\M_{k,l+1}(\beta)$.

On the other hand, if $\pi:\M_{k,l+1}(\beta)\to\M_{k,l}(\beta)$ is the map that forgets $w_0$, and $evb_j^l,evi_j^l,$ are the evaluation maps for $\M_{k,l}(\beta),$ then $\xi=\pi^*\xi'$ where
\[
\xi'=(evi^l)^*\glt\wedge (evb^l)^*\alt \in A^*(\M_{k,l}(\beta)).
\]
In particular
\[
\deg^d(\xi')=\deg^d(\xi)=\dim\M_{k,l+1}(\beta)>\dim\M_{k,l}(\beta).
\]
Therefore, $\xi'=0$ and so $\xi=0$.
\end{proof}

\subsubsection{Proof of \texorpdfstring{\Cref{thm:pgprop}}{Theorem 6}}

\begin{proof}[Proof of Theorem~\ref{thm:pgprop}]
Properties~\eqref{pprop1},~\eqref{pprop2}, and~\eqref{pprop3},
follow from Propositions~\ref{cl:pfund}, \ref{cl:pdiv}, and~\ref{prop:p_zero}, respectively.

\end{proof}

\subsection{Proof of the structure equations}\label{sec:str_pf}
In this section, we prove Proposition~\ref{prop:pstructure}.
We start with a series of lemmas required for the proof.

We decompose $\varepsilon(\a)$ into the part that depends on the elements $\a_j$ and the part that depends only on the number of inputs. Namely, for $k\in \Z_{\ge 0}$ and $\alt=(\a_1,\ldots,\a_k)$ we set
\begin{equation}\label{eq:epsprime}
\varepsilon'(k)=1+\sum_{j=1}^kj =\frac{k(k+1)}{2}+1,
\quad
\varepsilon''(\alt)=\sum_{j=1}^kj|\a_j|.
\end{equation}

\begin{figure}[ht]
\centering
\scalebox{.7}{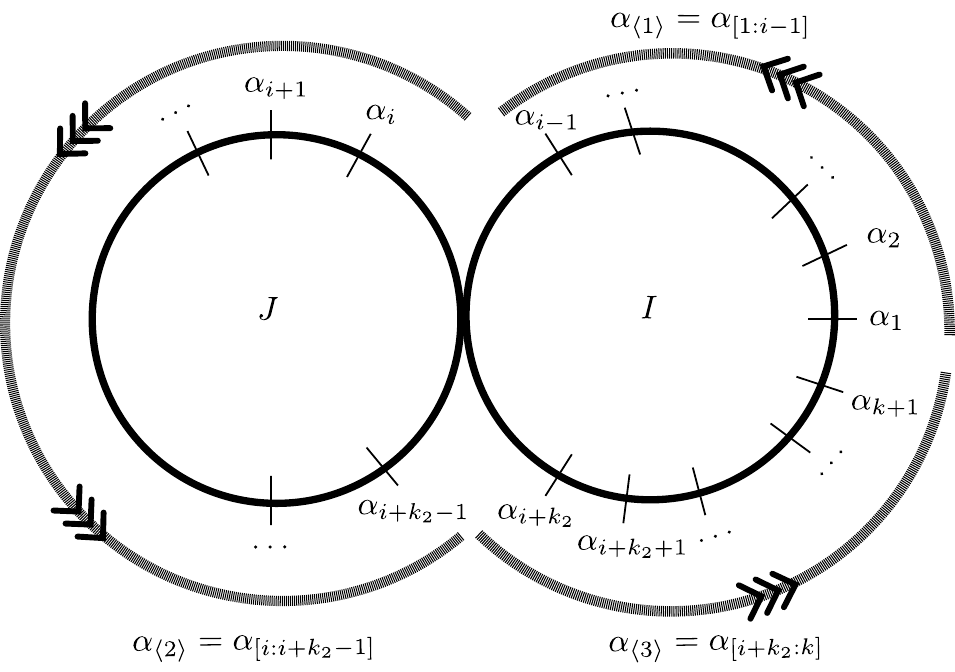}
\caption{A stable curve with two disk components, together with the distribution of inputs for each marked point.}
\label{fig:two_discs_splitting}
\end{figure}

In the next two lemmas, we use our notation for splitting lists of inputs. For the geometric meaning of
the splitting and all the parameters involved, we refer to \Cref{fig:two_discs_splitting}.
\begin{lm}[{\cite[Lemma 2.10]{ST1}}]\label{lm:sumofepsilons}
Let $\alt=(\a_1,\ldots,\a_k)\in A^*(L;R)^{\otimes k}$ and $\glt=(\gamma_1,\ldots,\gamma_l)\in A^*(X;R)^{\otimes l}.$
Fix a partition $I\sqcup J$ of $[l]$.
Take $i\in [k+1]$ and $k_2\in [0:k+1-i]$, and consider the splitting
\begin{align*}
\a
&= \a_{[1:i-1]}\otimes \a_{[i:i+k_2-1]}\otimes \a_{[i+k_2:k]}\\
&= \a_{\ngl{1}}\otimes\a_{\ngl{2}}\otimes \a_{\ngl{3}}.
\end{align*}
Set $k_1:= k+1-k_2$.
Then
\begin{enumerate}
  \item $\varepsilon'(k_1)+\varepsilon'(k_2)
      \equiv\varepsilon'(k)+ k+k_1k_2\pmod 2.$
  \item
  \begin{multline*}
\qquad\varepsilon''(\alt_{\ngl{1}},|\alt_{\ngl{2}}|+|\glt_J|+k_2, \alt_{\ngl{3}}) +\varepsilon''(\alt_{\ngl{2}})
      \equiv\\
     \equiv\varepsilon''(\alt)+ik_2+k_2|\alt_{\ngl{3}}|+ |\alt|+|\alt_{\ngl{1}}|+ i|\glt_J|\pmod 2.
  \end{multline*}
  \item
  \begin{multline*}
  \qquad \varepsilon(\alt_{\ngl{1}},|\alt_{\ngl{2}}|+|\glt_J|+k_2, \alt_{\ngl{3}})+ \varepsilon(\alt_{\ngl{2}})
\equiv\\
\equiv
\varepsilon(\alt)+|\alt|+k+|\alt_{\ngl{1}}|
+i|\glt_J|+k_2|\alt_{\ngl{3}}|+k_1k_2+ik_2
\pmod 2.
\end{multline*}
\end{enumerate}
\end{lm}

Similarly to~\eqref{eq:epsprime}, we consider the two parts of $\varepsilon_p$:
\[
\varepsilon'_p(k)= kn+\frac{k(k+1)}{2} =\varepsilon'(k)+kn-1,
\quad
\varepsilon''_p(\alt)= n|\alt|+\sum_{j=1}^kj|\a_j| =\varepsilon''(\alt)+n|\alt|.
\]

\begin{lm}\label{lm:sumofepsilonsp}
Let $\alt=(\a_1,\ldots,\a_k)\in A^*(L;R)^{\otimes k}$ and $\glt=(\gamma_1,\ldots,\gamma_l)\in A^*(X;R)^{\otimes l}.$ Fix  a splitting $\a=\a_{\ngl{1}}\otimes\a_{\ngl{2}}\otimes \a_{\ngl{3}}$ and a partition $I\sqcup J$ of $[l]$, take $i\in [k+1]$ and $k_2\in [0:k+1-i]$ such that $\alt_{\ngl{2}}=\alt_{[i:i+k_2-1]}$, and set $k_1:= k+1-k_2$.
Then
\begin{enumerate}
  \item $\varepsilon'_p(k_1)+\varepsilon'(k_2)
      \equiv\varepsilon'_p(k)+k+k_1k_2+k_2n+n\pmod 2.$
  \item
  \begin{multline*}
\qquad\varepsilon''_p(\alt_{\ngl{1}}, |\alt_{\ngl{2}}|+|\glt_{J}|+k_2, \alt_{\ngl{3}}) +\varepsilon''(\alt_{\ngl{2}})
      \equiv\\
     \equiv\varepsilon''_p(\alt)+(i+n)k_2 +k_2|\alt_{\ngl{3}}|+ |\alt|+|\alt_{\ngl{1}}|+ (i+n)|\glt_{J}|\pmod 2.
  \end{multline*}
  \item
  \begin{multline*}
  \qquad \varepsilon_p(\alt_{\ngl{1}},|\alt_{\ngl{2}}| +|\glt_{J}|+k_2, \alt_{\ngl{3}})+ \varepsilon(\alt_{\ngl{2}})
\equiv\\
\equiv
\varepsilon_p(\alt)+|\alt|+k+|\alt_{\ngl{1}}|
+(i+n)|\glt_{J}|+k_2|\alt_{\ngl{3}}|+k_1k_2+ik_2+n
\pmod 2.
\end{multline*}
\end{enumerate}
\end{lm}
\begin{proof}
We deduce the result from Lemma~\ref{lm:sumofepsilons}.
To see the first identity, compute
\begin{multline*}
\varepsilon'_p(k_1)+\varepsilon'(k_2)
=\varepsilon'(k_1)+k_1n-1+\varepsilon'(k_2)
      \equiv
\varepsilon'(k)+ k+k_1k_2+ k_1n-1
=\\
=\varepsilon'_p(k)-kn+1+k+k_1k_2+k_1n -1
\equiv
\varepsilon'_p(k)+k+k_1k_2+k_2n+n.
\end{multline*}
For the second identity,
\begin{align*}
\varepsilon''_p(&\alt_{[1:i-1]},|\alt_{[i:i+k_2-1]}| +|\glt_{J}|+k_2, \alt_{[i+k_2:k]}) +\varepsilon''(\alt_{[i:i+k_2-1]})\equiv\\
 &     \equiv
\varepsilon''(\alt)+ik_2+k_2|\alt_{[i+k_2:k]}|+ |\alt|+|\alt_{[1:i-1]}|+ i|\glt_{J}|+n|\alt|+n|\glt_{J}|+nk_2\\
&\equiv
\varepsilon''_p(\alt)-n|\alt|
+ik_2+k_2|\alt_{[i+k_2:k]}|+ |\alt|+|\alt_{[1:i-1]}|+ i|\glt_{J}|+n|\alt|+n|\glt_{J}|+nk_2\\
&\equiv
\varepsilon''_p(\alt)+
(i+n)k_2+k_2|\alt_{[i+k_2:k]}|+ |\alt|+|\alt_{[1:i-1]}|+ (i+n)|\glt_{J}|.
\end{align*}
The third identity is the sum of the first two.

\end{proof}

The next result follows from~\cite[Proposition~8.10.3]{FOOO}:
\begin{prop}\label{rem:pbdgluing} 
Fix $k,l\in \Z_{\ge 0},$ $\beta\in \sly$. Let $k_i,\beta_i,$ ($i=1,2$) be such that $k_1+k_2=k+1$ and $\beta_1+\beta_2=\beta$. Let $I\sqcup J=[l]$ be a partition of the interior labels except zero.
Let
$
B_{k_1,k_2;I,J}\subset \d\M_{k,l+1}(\beta)
$
be the boundary component where a generic point is a stable map of two disk components, the last $k_1-1$ boundary marked points and the interior marked points with indices in $I\cup\{0\}$ lie on one disk, while the rest of the marked points lie on the other. See Figure~\ref{fig:p_eq}.
Then the map
\[
\vartheta:\M_{k_1,I\cup \{0\}}(\beta_1)\prescript{}{evb_1^{\beta_1}}\times_{evb_0^{\beta_2}} \M_{k_2+1,J}(\beta_2)\stackrel{\sim}{\lrarr} B_{k_1,k_2;I,J}
\]
changes the orientation by $(-1)^{k_2k_1+k_2+n}$.
\end{prop}

\begin{figure}[ht]
\centering
\includegraphics[width=9cm]{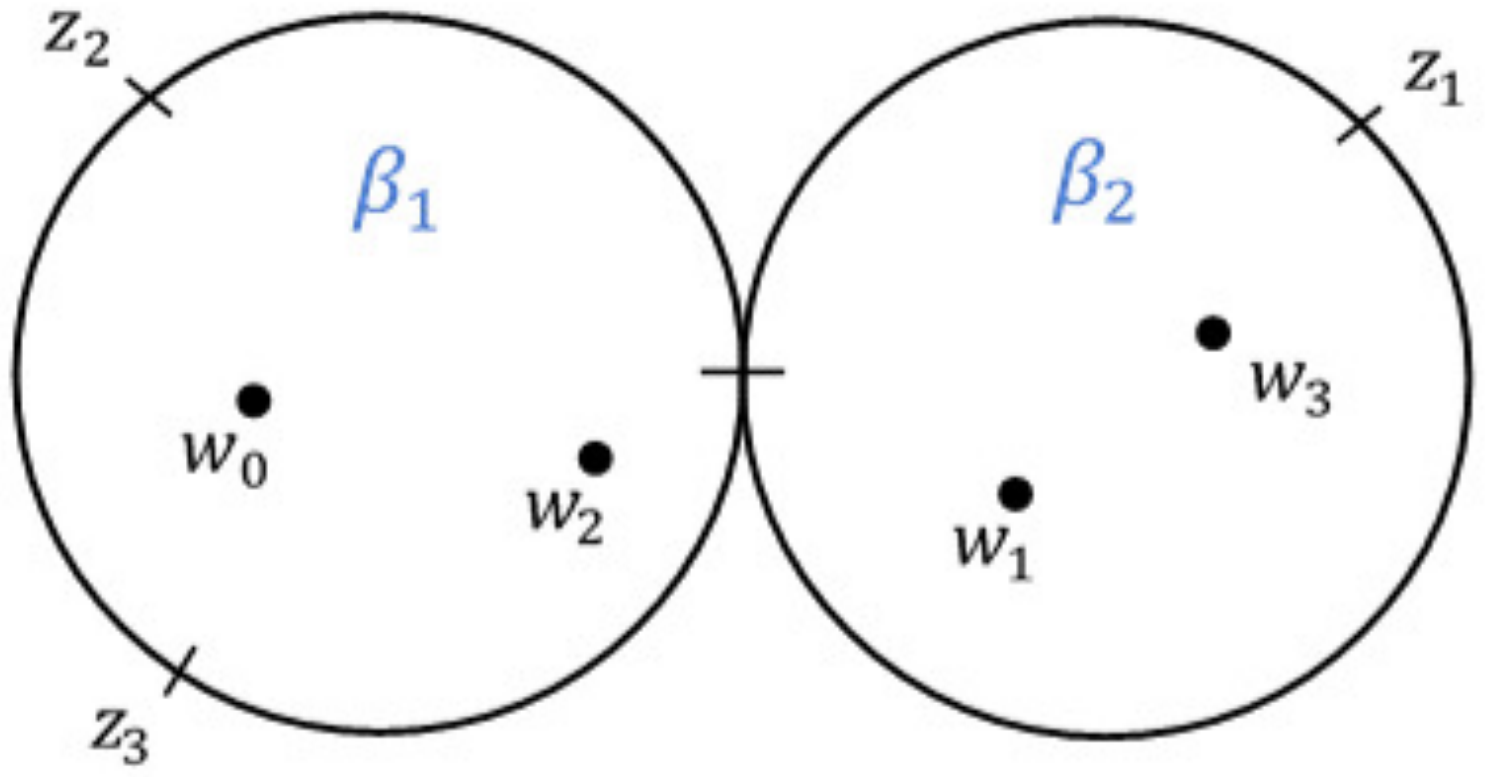}
\caption{The domain of an element of $B_{3,1;\{2\},\{1,3\}}\subset \d\M_{3,4}(\beta)$.}
\label{fig:p_eq}
\end{figure}

\begin{lm}\label{lm:pstar}
Let $k$, $l$, $\beta$, $k_i$, $\beta_i$, $I$, and $J$, be as in Lemma~\ref{rem:pbdgluing}, and let $\sigma\in \Z/k\Z$. Write $l_1:=|I|,$ $l_2:=|J|$.
Let $\alt=(\a_1,\ldots,\a_k)\in A^*(L;R)^{\otimes k}$ and $\glt=(\gamma_1,\ldots,\gamma_l)\in A^*(X;R)^{\otimes l}.$
Take a splitting $\a^\sigma=\a^\sigma_{\ngl{1}}\otimes\a^\sigma_{\ngl{2}}$
in which $\a^{\sigma}_{\ngl{j}}$ has length $k_j$ for $j=1,2$.
Let $B\subset\M_{k,l+1}(\beta)$ be the boundary component described as follows.
A generic point of $B$ is a stable map of two disk components. One of the components has on it the $k_1-1$ boundary marked points with the last indices of the list
$\left( \sigma(1), \dots, \sigma(k) \right)$, and the interior marked points with indices in $I\cup\{0\}$. The other has the other $k_2$ boundary marked points and the interior marked points with indices in $J$ on it.
Then
\[
(evi_0|_B)_*\left( evi^*\glt\wedge evb^*\alt \right)=
(-1)^*
\p_{k_1,l_1}^{\beta_1} (\q_{k_2,l_2}^{\beta_2}(\alt_{\ngl{2}}^\sigma;\glt_{J}), \alt_{\ngl{1}}^{\sigma};\glt_{I}),
\]
with
$ *=s^{[1]}_\sigma(\alt)
+\varepsilon_p(\alt)+|\alt|
+(1+n)|\glt_{J}|+k+s_{\sigma_{I\cup J}}(\glt)$.
\end{lm}

\begin{proof}
For any cyclic permutation $\sigma\in\Z/k\Z$, let
\[
\varphi_\sigma:\M_{k,l+1}(\beta)\stackrel{\sim}{\lrarr}\M_{k,l+1}(\beta)
\]
be the diffeomorphism defined by
\[
\varphi_\sigma([u,(z_j)_{j=1}^k,(w_j)_{j=0}^l])=[u,(z_{\sigma(j)})_{j=1}^k, (w_j)_{j=0}^l].
\]
The map $\varphi_\sigma$ changes orientation by $(-1)^{sgn(\sigma)}$.
Let $B_{k_1,k_2;I,J}$ be the boundary component as in Proposition~\ref{rem:pbdgluing}, and take the permutation $\sigma$ for which
\[
\varphi_{\sigma^{-1}}(B_{k_1,k_2;I,J})=B.
\]
Then
\[
\M_{k_1,I\cup\{0\}}(\beta_1) \prescript{}{evb_1^{\beta_1}}\times_{evb_0^{\beta_2}}\M_{k_2+1,J}(\beta_2) \xrightarrow[\vartheta]{\;\;\sim\;\;}B_{k_1,k_2;I,J}\xrightarrow[\varphi_{\sigma^{-1}}]{\;\;\sim\;\;}B
\]
is a diffeomorphism with a total change of orientation by the sign
\[
\delta_1:=k_2k_1+k_2+n+sgn(\sigma^{-1}).
\]

Consider the following pull-back diagram
\[
\xymatrix{
{\M_{k_1,I\cup\{0\}}(\beta_1)\times_L\M_{k_2+1,J}(\beta_2)}\ar[r]^(.65){p_2}\ar[d]^{p_1}&
{\M_{k_2+1,J}(\beta_2)}\ar[d]^{evb_0^{\beta_2}}\\
{\M_{k_1,I\cup\{0\}}(\beta_1)}\ar[r]^{\quad evb_1^{\beta_1}}&L\;.
}
\]
Write $\alt=(\a_1,\ldots,\a_k)$ and $\glt=(\gamma_1,\ldots,\gamma_l)$, and set
\[
\xi:=evi^*\glt\wedge evb^*\alt,
\]
and abbreviate $\bar{\xi}:=\vartheta^*\varphi_{\sigma^{-1}}^*\xi.$ Thus,
\begin{align*}
\bar{\xi}
&=\vartheta^*\varphi_{\sigma^{-1}}^*\Big( evi^*\gamma\wedge evb^*\alpha\Big)\\
&=\vartheta^*\Big(evi^*\gamma\wedge evb_{\sigma^{-1}}^*\alpha\Big)\\
&=(-1)^{s_\sigma(\alpha)}\cdot
\vartheta^*\Big(evi^*\gamma\wedge evb^*\alpha^{\sigma}\Big),
\end{align*}
with $s_\sigma(\alt)$ defined by~\eqref{eq:sgnsigmagamma}.
Set
\begin{gather*}
\xi_1:=\bigwedge_{j\in I}(evi_j^{\beta_1})^*\gamma_j \wedge\bigwedge_{j=2}^{k_1} (evb_j^{\beta_1})^*\alpha_{\sigma(j+k_2-1)}\in A^*(\M_{k_1,I\cup\{0\}}(\beta_1)),\\
\xi_2:=\bigwedge_{j\in J}(evi_j^{\beta_2})^*\gamma_j\wedge \bigwedge_{j=1}^{k_2} (evb_j^{\beta_2})^*\alpha_{\sigma(j)}\in A^*(\M_{k_2+1,J}(\beta_2)).
\end{gather*}
Then
\[
\vartheta^*\Big(\bigwedge_{j=1}^l evi_j^*\gamma_j\wedge \bigwedge_{j=1}^kevb_{j}^*\alpha_{\sigma(j)}\Big)=(-1)^{\delta_2}p_1^*\xi_1\wedge p_2^*\xi_2
\]
with
\[
\delta_2:=(|\alt_{\ngl{1}}^\sigma|+|\glt_J|)\cdot|\alt_{\ngl{2}}^\sigma| +s_{\sigma_{I\cup J}}(\glt).
\]
So,
\begin{align*}
(evi_0|_{B})_*{\xi}={}&
(-1)^{\delta_1}(evi_0^{\beta_1})_*(p_1)_*\bar\xi\\
={}& (-1)^{s_\sigma(\alt)+\delta_1+\delta_2}(evi_0^{\beta_1})_*(p_1)_* (p_1^*\xi_1\wedge p_2^*\xi_2)\\
={}&(-1)^{s_\sigma(\alt)+\delta_1+\delta_2}(evi_0^{\beta_1})_*(\xi_1\wedge(p_1)_*(p_2)^*\xi_2)\\
={}&(-1)^{s_\sigma(\alt)+\delta_1+\delta_2}(evi_0^{\beta_1})_* (\xi_1\wedge(evb_1^{\beta_1})^*(evb_0^{\beta_2})_*\xi_2)\\
={}&(-1)^{s_\sigma(\alt)+\delta_1+\delta_2+|(evb_0^{\beta_2})_* \xi_2|\cdot|\alt^\sigma_2|}(evi_0^{\beta_1})_*\Bigg(
\bigwedge_{j\in I}(evi_j^{\beta_1})^*\gamma_j\wedge\\
&\quad\wedge(evb_1^{\beta_1})^*(evb_0^{\beta_2})_*\xi_2
\wedge\bigwedge_{j=2}^{k_1}(evb_0^{\beta_1})^*\a_{\sigma(j+k_2-1)}
\Bigg)\\
={}&(-1)^*
\p_{k_1,l_1}^{\beta_1} (\q_{k_2,l_2}^{\beta_2}(\alt
^\sigma_{\ngl{1}};\glt_J),\alt^{\sigma}_{\ngl{2}};\glt_I)
\end{align*}
with
\[
*=s_\sigma(\alt)+\delta_1+\delta_2+|(evb_0^{\beta_2})_* \xi_2|\cdot|\alt^\sigma_{\ngl{2}}|+ \varepsilon_p((evb_0^{\beta_2})_*\xi_2,\alt^\sigma_{\ngl{2}})+ \varepsilon(\alt^\sigma_{\ngl{1}}).
\]
By Lemma~\ref{lm:epsilonsigma} and Lemma~\ref{lm:sumofepsilonsp} applied to $\alt^\sigma$ in the case $i=1,$ we get
\begin{align*}
*\equiv&\sum_{\substack{j<m\\\sigma(j)>\sigma(m)}} |\a_{\sigma(j)}|\cdot|\a_{\sigma(m)}| +k_2k_1+k_2+n+sgn(\sigma^{-1})
+(|\alt^\sigma_{\ngl{1}}|+|\glt_J|)\cdot|\alt_{\ngl{2}}^\sigma| +s_{\sigma_{I\cup J}}(\glt)+\\
&+(|\alt^\sigma_{\ngl{1}}|+|\glt_J|+k_2)\cdot|\alt^\sigma_{\ngl{2}}|
+\varepsilon_p(\alt^\sigma)+|\alt^\sigma|+k+n
+(1+n)|\glt_{J}|+k_2|\alt_{\ngl{2}}^\sigma|+k_1k_2+k_2\\
&\equiv
\sum_{\substack{j<m\\\sigma(j)>\sigma(m)}} |\a_{\sigma(j)}|\cdot|\a_{\sigma(m)}| +sgn(\sigma)
+\varepsilon_p(\alt^\sigma)+|\alt^\sigma|
+(1+n)|\glt_{J}|+k+s_{\sigma_{I\cup J}}(\glt)\\
&\equiv
\sum_{\substack{j<m\\\sigma(j)>\sigma(m)}} |\a_{\sigma(j)}|\cdot|\a_{\sigma(m)}| +sgn(\sigma)
+\sum_{\substack{i<j\\ \sigma(i)>\sigma(j)}}(|\a_{\sigma(j)}|
-|\a_{\sigma(i)}|)+\\
&+\varepsilon_p(\alt)+|\alt|
+(1+n)|\glt_{J}|+k+s_{\sigma_{I\cup J}}(\glt)\\
&\equiv
s^{[1]}_\sigma(\alt)
+\varepsilon_p(\alt)+|\alt|
+(1+n)|\glt_{J}|+k+s_{\sigma_{I\cup J}}(\glt).
\end{align*}
\end{proof}

\begin{prop}[{\cite[Proposition 2.12]{ST1}}]\label{rem:pintgluing} 
Let $l\in\Z_{\ge 0},\,\beta\in\sly$, and $\hat\beta \in H_2(X;\Z)$ with $\pr(\hat\beta) = \beta.$
Let
$
B\subset \d\M_{0,l}(\beta)
$
be the boundary component where a generic point is a sphere of class $\hat\beta$ intersecting $L$ at a marked point. Such spheres arise when the boundary of a disk collapses to a point. Equivalently, one can view this as interior bubbling from a ghost disk component. Note that the ghost disk is not stable.
Then the map
\[
\vartheta:
L\times_X\M_{l+1}(\hat\beta)\stackrel{\sim}{\lrarr}B.
\]
satisfies $sgn(\vartheta)=(-1)^{n+1+w_{\s}(\hat\beta)}$.
\end{prop}

\begin{lm}\label{lm:bbbp}
Let $B\subset\d\M_{0,l+1}(\beta)$ be a boundary component as in Proposition~\ref{rem:pintgluing}. Then
\[
(evi_0|_B)_*(\bigwedge_{j=1}^levi_j^*\gamma_j)
=
(-1)^{n+1}\q^{\hat\beta}_{\emptyset,l+1}(\gamma_1,\ldots,\gamma_l,i_*1_L).
\]
\end{lm}
\begin{proof}
Define $\xi$ on $\M_{0,l+1}(\beta)$ by
$\xi:=\bigwedge_{j=1}^l evi_j^*\gamma_j$. Let $\vartheta$ be the gluing map from Proposition~\ref{rem:pintgluing} and set $\xi'=\vartheta^*\xi.$
Consider the pullback diagram
\[
\xymatrix{
{L\times_X\M_{l+2}(\hat\beta)}\ar[r]^(.6){p_2}\ar[d]^{p_1}&
{\M_{l+2}(\hat\beta)}\ar[d]^{ev_{l+1}}\\
{L}\ar[r]^{i}&X
}
\]
and define $\xi''$ on $\M_{l+2}(\hat\beta)$ by $\xi'':=\bigwedge_{j=1}^l ev_j^*\gamma_j.$

Then
\begin{align*}
(evi_0)_*\xi
={}&(-1)^{n+1+w_s(\hat\beta)}(ev_0)_*(p_2)_*\xi'\\
={}&
(-1)^{n+1+w_s(\hat\beta)}(ev_0)_*(p_2)_*(p_2^*\xi''\wedge p_1^*1_L)\\
={}&
(-1)^{n+1+w_s(\hat\beta)}(ev_0)_*(\xi''\wedge (p_2)_*p_1^*1_L)\\
={}&
(-1)^{n+1+w_s(\hat\beta)}(ev_0)_*(\xi''\wedge ev_{l+1}^*i_*1_L)\\
={}&
(-1)^{n+1} \q^{\hat\beta}_{\emptyset,l+1}(\gamma_1,\ldots,\gamma_l, i_*1_L).
\end{align*}

\end{proof}

\begin{proof}[Proof of Proposition~\ref{prop:pstructure}]
The proof is based on Stokes' theorem, Proposition~\ref{stokes}, similarly to the proof of~\cite[Proposition 2.6]{ST1}.
Write $\alt=(\a_1,\ldots,\a_k)$ and $\glt=(\gamma_1,\ldots,\gamma_l)$ and apply Stokes to
\[
f=evi_0:\M_{k,l+1}(\beta)\to X ,\quad
\xi=evi^*\gamma\wedge evb^*\alpha.
\]

\textit{Left-hand side: $df_*\xi$.}
This contributes $(-1)^{\varepsilon_p(\alt)}d\p_{k,l}^\beta(\alt;\glt).$

\textit{Right-hand side, first summand: $f_*(d\xi)$.}
This contributes
\[
\sum_{i=1}^k
(-1)^{\varepsilon_p(\alt)+i+n+ |\glt|+|\alt_{[1:i-1]}|}\pkl^\beta(\alt_{[1:i-1]}, d\alpha_i,\alt_{[i+1:k]};\glt)
+
(-1)^{\varepsilon_p(\alt)}\pkl^\beta (\alt;d(\glt)).
\]
Denote by $\sigma_i\in S_k$ the cyclic permutation such that $\sigma_i(1)=i.$ In other words, $\sigma_i\in \Z/k\Z$ is given by adding $i-1$ modulo $k$. As in the proof of Lemma~\ref{lm:epsilonsigma}, we have
\begin{align*}
\{j<m \mbox{ and } \sigma_i(j)>\sigma_i(m)\} &\iff \{j\le k-(i-1) \mbox{ and } m\ge k-(i-1)+1\}\\
&\iff \{j\le k-i+1 \mbox{ and } m\ge k-i+2\}.
\end{align*}
Therefore,
\begin{align*}
s_{\sigma_i}^{[1]}(\alt_{[1:i-1]},d\alpha_i,\alt_{[i+1:k]})&=
\sum_{\substack{1< j\le k-i+1\\ k-i+2\le m\le k}}(|\a_{\sigma(j)}|+1)(|\a_{\sigma(m)}|+1)+\sum_{m=k-i+2}^k(|\a_{i}|+1+1)(|\a_{\sigma(m)}|+1)\\
&=s_{\sigma_i}^{[1]}(\alt)+\sum_{m=k-i+2}^k(|\a_{\sigma(m)}|+1)\\
&=s_{\sigma_i}^{[1]}(\alt)+\sum_{m=k-i+2}^k(|\a_{m+i-1-k}|+1)\\
&=s_{\sigma_i}^{[1]}(\alt)+\sum_{j=1}^{i-1}(|\a_j|+1)\\
&=s_{\sigma_i}^{[1]}(\alt)+|\alt_{[1:i-1]}|+i-1.
\end{align*}
By Proposition~\ref{prop:p_sgn} we then have
\begin{align*}
\sum_{i=1}^k(-1&)^{\varepsilon_p(\alt) +i+n+|\glt|+|\alt_{[1:i-1]}|} \pkl^\beta(\alt_{[1:i-1]}, d\alpha_i,\alt_{[i+1:k]};\glt)=\\
&=
\sum_{i=1}^k(-1)^{\varepsilon_p(\alt) +i+n+|\glt|+|\alt_{[1:i-1]}| +s_{\sigma_i}^{[1]}(\alt_{[1:i-1]},d\a_i,\alt_{[i+1:k]})} \pkl^\beta(d\a_i,\alt_{[i+1:k]},\alt_{[1:i-1]};\glt)\\
&=
\sum_{i=1}^k(-1)^{\varepsilon_p(\alt)+ n+1 + |\glt|+s_{\sigma_i}^{[1]}(\alt)} \pkl^\beta(d\a_i,\alt_3,\alt_1;\glt)\\
&=
\sum_{\sigma\in\Z/k\Z}(-1)^{\varepsilon_p(\alt)+
n+1+|\glt|+s_{\sigma}^{[1]}(\alt)} \pkl^\beta(d\a_{\sigma(1)},\a_{\sigma(2)},\ldots,\a_{\sigma(k)};\glt).
\end{align*}

\textit{Right-hand side, first type of boundary contribution:}
 Consider the contribution $(f|_{B})_*\xi$ where $B\subset \d\M_{k,l+1}(\beta)$ is as in Lemma~\ref{lm:pstar}.
Since
\[
\dim(\M_{k,l+1}(\beta))=n-3+\mu(\beta)+k+2(l+1)
\equiv k+n+1 \pmod 2
\]
and $|\xi| = |\alt| + |\glt|,$
the contribution of $(f|_B)_*\xi$ to Stokes' theorem comes with the sign $(-1)^{s+t}=(-1)^{|\alt| + |\glt| + k+n+1}.$
By Lemma~\ref{lm:pstar}, we have
\[
(f|_B)_*\xi=
(-1)^*\;
\p_{k_1,l_1}^{\beta_1} (\q_{k_2,l_2}^{\beta_2}(\alt_{\ngl{2}}^\sigma;\glt_{J}), \alt_{\ngl{1}}^{\sigma};\glt_{I}),
\]
with a total sign contribution given by
\begin{align*}
*+|\alt|&+|\glt|+k+n+1\equiv\\
\equiv{}&
s^{[1]}_\sigma(\alt)
+\varepsilon_p(\alt)+|\alt|
+(1+n)|\glt_{J}|+k+s_{\sigma_{I\cup J}}(\glt)
+|\alt|+|\glt|+k+n+1\\
\equiv{}&
s^{[1]}_\sigma(\alt)
+\varepsilon_p(\alt)
+s_{\sigma_{I\cup J}}(\glt)
+|\glt|+(n+1)(|\glt_{J}|+1) \pmod 2.
\end{align*}
Thus, the total contribution of $B$ is
\[
(-1)^{s+t}(f|_B)_*\xi=(-1)^{\varepsilon_p(\alt)+|\glt|+s_{\sigma_{I\cup J}}(\glt)+s_\sigma^{[1]}(\alt)+(n+1)(|\glt_{J}|+1)} \cdot\p_{k_1,l_1}^{\beta_1}(\q_{k_2,l_2}^{\beta_2} (\alt_{\ngl{2}}^\sigma;\glt_J),\alt_{\ngl{1}}^\sigma;\glt_I).
\]

\textit{Right-hand side, second type of boundary contribution:}
In the case $k=0$, we also take $(f|_B)_*\xi$ for
$B$ as in Proposition~\ref{rem:pintgluing}.
By Lemma~\ref{lm:bbbp}, the contribution of this component equals
\begin{align*}
(-1)^{s+t}(evi_0|_B)_*(\bigwedge_{j=1}^{l}evi_j^*\gamma_{j})=
(-1)^{|\gamma|}\q_{\emptyset,l+1}^{\hat\beta}(\gamma_1,\ldots,\gamma_l,i_*1_L).
\end{align*}

\textit{Combining contributions:}
\begin{align*}
(-1)^{\varepsilon_p(\alt)}&d\pkl^\beta(\alt;\glt)=
(-1)^{\varepsilon_p(\alt)} \pkl^\beta(\alt;d(\glt))+\\
&+ \sum_{\sigma\in\Z/k\Z} (-1)^{\varepsilon_p(\alt)+|\glt|+s_\sigma^{[1]}(\alt)+(n+1)}
\pkl^\beta(d\alpha_{\sigma(1)},\alpha_{\sigma(2)},\ldots,\alpha_{\sigma(k)}; \glt)+\\
&+\hspace{-1em} \sum_{\substack{\sigma\in\Z/k\Z,\\
k_1+k_2=k+1,\; I\sqcup J=[l] \\\beta_1+\beta_2=\beta\\ (k_2,|J|,\beta_2)\ne(1,0,\beta_0)}} \hspace{-1.5em}(-1)^{\varepsilon_p(\alt) +|\glt|+s_{\sigma_{I\cup J}}(\glt)+s_\sigma^{[1]}(\alt)+(n+1)(|\glt_{J}|+1)}
\p^{\beta_1}_{k_1,|I|}(\q_{k_2,|J|}^{\beta_2} (\alt^\sigma_{(2)};\glt_J),\alt^\sigma_{(1)};\glt_I)+\\
&+(-1)^{|\gamma|}\delta_{k,0}\cdot \q_{\emptyset,l+1}^{\hat\beta}(\gamma_1,\ldots,\gamma_l,i_*1_L).
\end{align*}
Dividing by $(-1)^{\varepsilon_p(\alt)}$, multiplying by $T^\beta$ and summing over $\beta$, we get the required equation.

\end{proof}

\subsection{Bulk and boundary deformation}

Let $\gamma \in \mI_R A^*(X;R)$ such that $|\gamma|=2$ and $d\gamma=0$. Let $b\in \mI_R A^*(L;R)$ such that $|b|=1$.
Recall the deformed closed~\eqref{eq:defq} and closed-open maps~\eqref{eq:defqempt}, and define, for $k,l\ge 0$ and $\alt=(\alpha_1,\ldots,\alpha_k)\in \big(A^*(L;R)\big)^{\otimes k}$, $\eta=(\eta_1,\ldots,\eta_l)\in \big(A^*(X;R)\big)^{\otimes l}$, the deformed open-closed maps by
\begin{equation}\label{eq:defp}
\pbg_{k,l}(\alt;\eta):=\sum_{s,t\ge 0}
\sum_{\sum_{j=0}^{k-1}i_j=s}
\frac{1}{t!}
\p_{k+s,l+t}(b^{\otimes i_0}\otimes \a_1\otimes b^{\otimes i_1}\otimes \cdots \otimes b^{\otimes i_{k-1}}\otimes \a_k ;\eta\otimes \gamma^{\otimes t}).
\end{equation}
Then a similar structure equation is satisfied by $\pbg$:
\begin{cor}
Consider lists $\alt=(\alpha_1,\ldots,\alpha_k)$, $\alpha_j\in A^*(L;R)$, and $\eta=(\eta_1,\ldots,\eta_l)$, $\eta_j\in A^*(X;R)$.
Then
\begin{multline}\label{eq:pbg_rel}
d\pbg_{k,l}(\alt;\eta)
=
\pbg_{k,l}(\alt; d(\eta))+ \\
\qquad+
\sum_{\substack{I\sqcup J=[l]\\ \sigma\in\Z/k\Z}}
(-1)^{s_\sigma^{[1]}(\alt)+|\eta| +s_{\sigma_{I\sqcup J}}(\eta)+(n+1)(|\eta_{J}|+1)}
\pbg_{k_1,|I|}(\qbg_{k_2,|J|} ((\alt^\sigma)_{(1)};\eta_J) \otimes(\alt^\sigma)_{(2)}; \eta_{I})+\\
+
\delta_{k,0}\cdot(-1)^{|\eta|}\qg_{\emptyset,l+1}(\eta\otimes i_*1_L).\notag
\end{multline}
\end{cor}
\begin{proof}
This is an immediate consequence of Proposition~\ref{prop:pstructure} and Proposition~\ref{prop:p_sgn}.
\end{proof}

\section{Open-closed maps on the Hochschild and cyclic complexes}\label{sec:OChoch}
In this section, we interpret the properties of the open-closed maps from Section~\ref{sec:p} in
terms of the maps induced on various Hochschild and cyclic chain complexes introduced in
Section~\ref{sec:Hochschild complexes}.

\subsection{A geometric realization of Hochschild homology}

We work over the field $k=\R$ and the algebra $R$ given in~\eqref{eq:R}. For the $R$-module we take the $A_\infty$-algebra
\begin{equation*}
    A = C = A^*(L;R),
\end{equation*}
so that the Hochschild cochain complex $\choch[A]$ becomes
\[
\choch[C]=\tensr{C[1]} =\bigoplus_{j=1}^\infty C[1]^{\otimes j}.
\]
We keep using the notation $|\cdot |$ for the grading of $C$ and $\V\cdot \V$ for grading of $C[1]$, and thus also of $\choch[C]$.
For lists, sublists, and splittings, recall the notation from Sections~\ref{ssec:notlist} and~\ref{ssec:not}.
The maps $\{\mg_k\}_{k=0}^\infty$ give an $A_\infty$ structure on $C$, and we denote the induced operator on $\tens{C[1]}$ by
\[
\mu=\mg=\sum_{k\ge 0} \mg_k:\tens{C[1]} \lrarr C[1].
\]
In particular, the Hochschild boundary operator
$\dhoch \colon \tensr{C[1]} \rightarrow \tensr{C[1]}$
takes the form
\begin{align*}
\dhoch[x \otimes \a]
={}& (-1)^{\V x\V + \V\a_{(1)}\V}
    x \otimes \a_{(1)} \otimes \mg(\a_{(2)}) \otimes \a_{(3)}    \\
&+ (-1)^{\V\a_{(3)}\V \cdot \left( \V x\V + \V\a_{(1)}\V + \V\a_{(2)}\V \right)}
    \mg \left( \a_{(3)} \otimes x \otimes \a_{(1)} \right) \otimes \a_{(2)},
\end{align*}
with $x\in C[1]$ and $\a=\a_1\otimes\cdots\otimes\a_{k-1}\in C[1]^{\otimes k-1}$.

The rest of the constructions from Section~\ref{sec:Hochschild complexes} follow through, giving rise to cyclic, extended, reduced, and normalized versions.

\subsection{The maps}
Recall the operators $\pg_k:C[1]^{\otimes k} \to \A^*(X;R)$ defined by~\eqref{eq:pkg} and the induced operator
\begin{equation}
  \p^{\gamma} = \sum_{k = 1}^{\infty} \p_k^{\gamma} \colon \choch[C]=\tensr{C[1]}
  \rightarrow  \ccur{X}[*][][R].
\end{equation}
This operator uses all possible moduli spaces of disks with \textbf{at least one} marked
boundary point. The results of the previous section show that $\pg$ satisfies
the following properties:

\begin{enumerate}
  \item (\textbf{Structure Equations}) By \cref{prop:pstructure}, we have
    \begin{equation} \label{eq:pg-ainf-structure}
      d \left( \pg \left( \alpha \right) \right) =
      \sum_{\sigma \in \Z/k\Z} (-1)^{s_\sigma^{[1]}(\alpha) + (n+1)} \pg \left( \mg \left(
      \alpha_{(1)}^{\sigma} \right) \otimes \alpha_{(2)}^{\sigma} \right).
    \end{equation}
  \item (\textbf{Unit}) By \cref{prop:p_unit}, we have
    \begin{align}
      &\pg \left( \alpha_1, \dots, \alpha_{i-1}, 1_{L}, \alpha_{i+1}, \dots,
      \alpha_k \right) = 0, \qquad 1 \leq i \leq k, k \geq 2,
      \label{eq:pg-unit-generic}  \\
      &\pg \left( 1_L \right) = (-1)^{n+1}i_{*} \left( 1_L \right).
      \label{eq:pg-unit-special}
    \end{align}
  \item (\textbf{Degree}) By \cref{prop:p_deg},
   the map $\pg$ is homogeneous of degree $n + 1$.
  \item (\textbf{Cyclic Symmetry}) By \cref{prop:p_sgn}, we have
    \begin{equation} \label{eq:pg-cyclic-symmetry}
      \pg \left( \alpha_1, \dots, \alpha_k \right) = (-1)^{\| \alpha_k \| \cdot
	\left( \| \alpha_1 \| + \dots + \| \alpha_{k-1} \| \right)} \pg \left(
	\alpha_k, \alpha_1, \dots, \alpha_{k-1} \right).
    \end{equation}
\end{enumerate}

We first interpret the structure equations as showing that $\pg$ defines a
chain map from the Hochschild complex of $C$ to the complex of de-Rham currents on $X$. In other words, we deduce Theorem~\ref{thm:pg-hoch} from the structure equations.

\begin{proof}[Proof of Theorem~\ref{thm:pg-hoch}]
  This will follow from \eqref{eq:pg-ainf-structure} after a sign
  calculation. Let us write $\alpha = \alpha_1 \otimes \dots \otimes \alpha_k$ and set $l
  = \alpha_2 \otimes \dots \otimes \alpha_k$ so that $\alpha = \alpha_1 \otimes
  l$. Then using the definition of the Hochschild differential and the cyclic
  symmetry of $\pg$, we have
  \begin{align*}
    \pg \left( \dhoch \left( \alpha \right) \right) ={}& (-1)^{\degs{\alpha_1} +
    \degs{l_{(1)}}} \pg \left( \alpha_1 \otimes l_{(1)} \otimes \mg \left(
    l_{(2)} \right) \otimes l_{(3)} \right) \\
    &+ (-1)^{\degs{l_{(3)}} \cdot \left( \degs{\alpha_1} + \degs{l_{(1)}} +
    \degs{l_{(2)}} \right)} \pg \left( \mg \left( l_{(3)} \otimes \alpha_1
    \otimes l_{(1)} \right) \otimes l_{(2)} \right) \\
    ={}&  (-1)^{\left( \degs{l_{(2)}} + \degs{l_{(3)}} \right) \cdot \left(
    \degs{\alpha_1} + \degs{l_{(1)}} \right)}
    \pg \left( \mg \left( l_{(2)} \right) \otimes l_{(3)} \otimes \alpha_1
    \otimes l_{(1)} \right) \\
    &+ (-1)^{\degs{l_{(3)}} \cdot \left(
    \degs{\alpha_1} + \degs{l_{(1)}} +
    \degs{l_{(2)}} \right)} \pg \left( \mg \left( l_{(3)} \otimes \alpha_1
    \otimes l_{(1)} \right) \otimes l_{(2)} \right).
  \end{align*}
  We want to compare this expression to the expression appearing in the
  structure equation \eqref{eq:pg-ainf-structure}. In the expression for $\pg
  \left( \dhoch \left( \alpha \right) \right)$ we are summing over all possible
  ways to split $l$ into three lists $l = l_{\gen{1}} \otimes l_{\gen{2}} \otimes
  l_{\gen{3}}$ while in the expression appearing in the structure equation we are
  summing over all cyclic permutations $\sigma \in \Z /k \Z$ and for each
  permutation, over all splittings of the permuted list $\alpha^{\sigma}$ into
  two lists $\alpha^{\sigma} = \alpha^{\sigma}_{\gen{1}} \otimes
  \alpha^{\sigma}_{\gen{2}}$. We need to describe a bijective correspondence
  between both descriptions and verify that the associated signs match.

  Fix some cyclic permutation $\sigma \in \Z/k\Z$ and a specific   splitting of $\alpha^{\sigma}$ into two lists $\alpha^{\sigma} =
  \alpha^{\sigma}_{\gen{1}} \otimes \alpha^{\sigma}_{\gen{2}}$ where
  \begin{equation*}
    \alpha^{\sigma}_{\gen{1}} = \alpha_{\sigma(1)} \otimes \dots \otimes
    \alpha_{\sigma(r)}, \,\,\, \alpha^{\sigma}_{\gen{2}} = \alpha_{\sigma(r+1)}
    \otimes \dots \otimes \alpha_{\sigma(k)}.
  \end{equation*}
  We have two distinct cases:
  \begin{enumerate}
    \item Assume there exists $1 \leq i \leq r$ such that $\alpha(i) = 1$ (that is,
      $\alpha_1$ appears in $\alpha^{\sigma}_{\gen{1}}$). In this case, let us set
      \begin{align*}
	\alpha^{\sigma} &= \underbrace{\alpha_{\sigma(1)} \otimes \dots \otimes
	  \alpha_{\sigma(i-1)}}_{l_{\gen{3}}} \otimes \alpha_1 \otimes
	  \underbrace{\alpha_{\sigma(i+1)} \otimes
	  \dots \otimes \alpha_{\sigma(r)}}_{l_{\gen{1}}} \otimes
	  \underbrace{\alpha_{\sigma(r+1)} \otimes \dots \otimes
	  \alpha_{\sigma(k)}}_{l_{\gen{2}}} \\
	  &= \underbrace{l_{\gen{3}} \otimes \alpha_1 \otimes
	  l_{\gen{1}}}_{\alpha^{\sigma}_{\gen{1}}} \otimes
	  \underbrace{l_{\gen{2}}}_{\alpha^{\sigma}_{\gen{2}}}.
	\end{align*}
	Note that
	\begin{equation*}
	  s_{\sigma}^{[1]}(\alpha) = \degs{l_{\gen{3}}} \cdot \left( \degs{\alpha_1} +
	  \degs{l_{\gen{1}}} + \degs{l_{\gen{2}}} \right)
	\end{equation*}
	and so
	\begin{align*}
	   (-1)^{s_\sigma^{[1]}(\alpha)}
	   &\pg \left( \mg \left( \alpha_{\gen{1}}^{\sigma} \right) \otimes
	   \alpha_{\gen{2}}^{\sigma} \right) = \\
	   (-1)^{\degs{l_{\gen{3}}} \cdot \left( \degs{\alpha_1} +
	   \degs{l_{\gen{1}}} + \degs{l_{\gen{2}}} \right)}
	   &\pg \left( \mg \left( l_{\gen{3}} \otimes
	   \alpha_1 \otimes l_{\gen{1}} \right) \otimes l_{\gen{2}} \right).
	\end{align*}
      \item Assume there exists $r+1 \leq i \leq k$ such that $\alpha(i) = 1$ (that
	is, $\alpha_1$ appears in $\alpha^{\sigma}_{\gen{2}}$). In this case, let us
	set
	\begin{align*}
	\alpha^{\sigma} &= \underbrace{\alpha_{\sigma(1)} \otimes \dots \otimes
	  \alpha_{\sigma(r)}}_{l_{\gen{2}}} \otimes
	  \underbrace{\alpha_{\sigma(r+1)} \otimes
	  \dots \otimes \alpha_{\sigma(i-1)}}_{l_{\gen{3}}} \otimes \alpha_1 \otimes
	  \underbrace{\alpha_{\sigma(i+1)} \otimes \dots \otimes
	  \alpha_{\sigma(k)}}_{l_{\gen{1}}} \\
	  &= \underbrace{l_{\gen{2}}}_{\alpha^{\sigma}_{\gen{1}}} \otimes
	  \underbrace{l_{\gen{3}} \otimes \alpha_1 \otimes
	  l_{\gen{1}}}_{\alpha^{\sigma}_{\gen{2}}}.
	\end{align*}
	Note that
	\begin{equation*}
	  s_{\sigma}^{[1]}(\alpha) = \left( \degs{l_{\gen{2}}} + \degs{l_{\gen{3}}}
	  \right) \cdot \left( \degs{\alpha_1} +\degs{l_{\gen{1}}} \right).
	\end{equation*}
	and so
	\begin{align*}
	   (-1)^{s_\sigma^{[1]}(\alpha)}
	   &\pg \left( \mg \left( \alpha_{\gen{1}}^{\sigma} \right) \otimes
	   \alpha_{\gen{2}}^{\sigma} \right) = \\
	   (-1)^{\left( \degs{l_{\gen{2}}} + \degs{l_{\gen{3}}} \right) \cdot
	   \left( \degs{\alpha_1} + \degs{l_{\gen{1}}} \right)} &\pg \left( \mg
	   \left( l_{\gen{2}} \right) \otimes l_{\gen{3}} \otimes \alpha_1
	   \otimes l_{\gen{1}} \right).	   	
	 \end{align*}
  \end{enumerate}
  This analysis shows that we indeed have a bijective correspondence with
  matching signs and hence
  \begin{equation*}
    \pg \left( \dhoch \left( \alpha \right) \right) =  \sum_{\sigma \in \Z/k\Z}
    (-1)^{s_\sigma^{[1]}(\alpha)} \pg \left( \mg \left( \alpha_{(1)}^{\sigma}
    \right) \otimes \alpha_{(2)}^{\sigma} \right).
  \end{equation*}
By the structure equations, the right-hand side equals $(-1)^{n+1}d\pg(\a)$, as desired.
\end{proof}

\begin{rem}
  In fact the calculation above shows that the Hochschild differential on
  Connes cyclic complex can be written more symmetrically as
  \begin{equation*}
    \dhoch \left( \cyccl{l} \right) = \sum_{\sigma \in \Z/k\Z}
    (-1)^{s_\sigma(l)} \cyccl{ \mu \left( l^{\sigma}_{(1)} \right) \otimes
    l^{\sigma}_{(2)}}.
  \end{equation*}
  Given a map $\psi \colon \choch[A] \rightarrow M$ which is cyclic-symmetric
  (meaning, $\psi \circ \t = \psi$), we have
  \begin{equation*}
    \psi \left( \dhoch \left( l \right) \right) = \sum_{\sigma \in \Z/k\Z}
    (-1)^{s_\sigma(l)} \psi \left( \mu \left( l^{\sigma}_{(1)} \right) \otimes
    l^{\sigma}_{(2)} \right).
  \end{equation*}
\end{rem}

Taking into account the unit property~\eqref{eq:pg-unit-generic} of $\pg$, we
immediately get that $\pg$ descends to the normalized Hochschild complex, which \textit{proves Theorem~\ref{thm:pg-normhoch}}.

Finally, the cyclic symmetry~\eqref{eq:pg-cyclic-symmetry} of $\pg$ and Theorem~\ref{thm:pg-hoch} show that $\pg$ descends to a chain map on the Connes cyclic complex, which \textit{proves Theorem~\ref{thm:pg-connes}}.

Next, we will discuss the relation between $\pg$ and the reduced cyclic complex.
Since $\pg \left( 1_L \right) = (-1)^{n+1}i_{*} \left( 1_L \right)$, the map $\pg$ does
not descend to the reduced cyclic complex without changing the codomain of
$\pg$. To fix this, we can work with the complex $\ccur{X}[*][L][R]$ which is
obtained from $\ccur{X}[*][][R]$ by quotienting out by $i_{*} \left( 1_L \right)$, as discussed in detail in Section~\ref{ssec:currents}. Then
\textit{Theorem~\ref{thm:pg-redconnes} immediately follows}.

Finally, let us try and incorporate the map $\pg_{0}$ which gives us
contributions from the moduli spaces of disks with no boundary points. Consider
the operator $\pg_{+} \colon \tens{C[1]} \rightarrow \ccur{X}[*+n+1][][R]$
defined by \begin{equation*}
  \pg_{+} = \sum_{k=0}^{\infty} \pg_k.
\end{equation*}
Using \cref{prop:pstructure}, we see that
\begin{align*}
  d \left( \pg_{+}(1) \right) &= d \left( \pg_{0}(1) \right) = d \left( \sum_{l
  \geq 0} \frac{1}{l!} \p_{0,l} \left(1;\gamma^{\otimes l} \right) \right) \\
  &=
  (-1)^{n+1}\sum_{l \geq 0} \frac{1}{l!} \left( \sum_{m + k = l} {l \choose m}  \p_1
  \left( \q_1 \left( 1 ; \gamma^{\otimes m} \right); \gamma^{\otimes k} \right)
  \right)
  +
  \sum_{l \geq 0} \frac{1}{l!} \q_{\emptyset,l+1} \left( i_{*} \left( 1_L
  \right) \otimes \gamma^{\otimes l} \right) \\
  &= (-1)^{n+1}\pg_{1} \left( \mg_{0}(1) \right) + \q_{\emptyset,1}^{\gamma} \left(
  i_{*} \left( 1_L \right) \right).
\end{align*}
Since $\pg_{+}$ is symmetric with respect to the cyclic group action, it induces
a map  on
the extended cyclic complex
\[
\pg_{+} \colon \ceconnes[C] \rightarrow \ccur{X}[*+n+1][][R].
\]
The calculation above shows that the map $\pg_{+}$
does not commute with the differentials on elements of weight zero but we have
instead
\begin{gather*} \left( \pg_{+} \circ \dhoch \right)(1) = \left( \pg_{+}
  \left( \mg_{0}(1) \right) \right) =
  (-1)^{n+1}\big(\left( d \circ \pg_{+} \right)(1) -
  \q_{\emptyset,1}^{\gamma} \left( i_{*} \left( 1_L \right) \right)\big), \\
  \left( \pg_{+} \circ \dhoch \right) \left( \alpha_1, \dots \alpha_k \right) =
  (-1)^{n+1}\left( d \circ \pg_{+} \right) \left( \alpha_1, \dots, \alpha_k \right),
  \,\,\, k \geq 1.
\end{gather*}

\begin{proof}[Proof of Theorem~\ref{thm:pg-extended}]
  Since $\mathcal{P}$ and $\pg_{+}$ act the same on all elements except those of
  weight zero, we only need to check that $\left( d \circ \mathcal{P} \right)(1)
  = \left( \mathcal{P} \circ \dhoch \right)(1)$. And indeed, since $d\eta =
  -\zeta_L  = - i_{*} \left( 1_{L} \right)$ we
  have
  \begin{align*}
    d \left( \mathcal{P}(1) \right)
    &= d \left( \pg_0(1) \right) + d \left(
    \q_{\emptyset,1}^{\gamma} \left( \eta \right) \right)
    = d \left( \pg_0(1)
    \right) + \q_{\emptyset,1}^{\gamma} \left( d \eta \right) \\
    &= d \left( \pg_0(1) \right) - \q_{\emptyset,1}^{\gamma} \left(\zeta_L \right) \\
    &= (-1)^{n+1}\pg_{1} \left( \dhoch(1) \right) \\
    &=(-1)^{n+1} \mathcal{P} \left( \dhoch(1) \right).
  \end{align*}
  By replacing the codomain $\ccur{X}[* + n + 1][][R]$ with $\ccur{X}[* + n +
  1][L][R]$, we quotient out $i_{*} \left( 1_L \right)$ and then
  $\mathcal{P}(\alpha_1, \dots, \alpha_k)$ vanishes whenver $k \geq 1$ and
  $\alpha_i = 1_L$ for some $1 \leq i \leq k$. Thus, $\mathcal{P}$ also descends
  to a map $\mathcal{P} \colon \cerconnes[C] \rightarrow \ccur{X}[*+n+1][L][R]$
  as required. Finally, if $[\zeta_1] = [\zeta_2]$ in $\hcur{X}[n-1][L][R]$ then
  we can write $\zeta_1 = \zeta_2 + d\nu$ and then since
  $\q_{\emptyset,1}^{\gamma}$ is a chain map we have
  \begin{equation*}
    \q_{\emptyset,1}^{\gamma}(\zeta_1) = \q_{\emptyset,1}^{\gamma}(\zeta_2) +
    d \left( \q_{\emptyset,1}^{\gamma}(\nu) \right)
  \end{equation*}
  so $\mathcal{P}_{\zeta_1}$ and $\mathcal{P}_{\zeta_2}$ induce the same map on
  homology.
\end{proof}

\section{Pseudoisotopy}\label{sec:isot}

In this section, we construct operators $\pt$, which can be thought of as a family of $\p$ operators, and establish their properties.
The family is parameterized by the interval $I:=[0,1]$, and we think of
the dga
\[
\mR := A^*(I;R)
\]
as the underlying ring.
We prove that the $\pt$ operators satisfy properties analogous to the $\p$ operators. We then briefly discuss the suitable extensions of Theorems~\ref{thm:pg-hoch}-\ref{thm:pg-extended}, and invariance statements as their potential consequence.

Throughout, fix a family of $\omega$-tame almost complex structures $\{J_t\}_{t\in I}$.

\subsection{Structure}
For all $\beta\in\sly$, $k,l\ge 0$,  $(k,l,\beta) \ne (0,0,\beta_0)$, define
\[
\Mt_{k,l+1}(\beta):=\{(t,u)\,|\,t\in I, u\in\M_{k+1,l}(\beta;J_t)\}.
\]
The moduli space $\Mt_{k,l+1}(\beta)$ comes with evaluation maps
\begin{gather*}
\evbt_j:\Mt_{k,l+1}(\beta)\lrarr I\times L, \quad j\in\{1,\ldots,k\},\\
\evbt_j(t,(\Sigma,u,\vec{z},\vec{w})):=(t,u(z_j)),
\end{gather*}
and
\begin{gather*}
\evit_j:\Mt_{k,l+1}(\beta)\lrarr I\times X, \quad j\in\{0,\ldots,l\},\\
\evit_j(t,(\Sigma,u,\vec{z},\vec{w})):=(t,u(w_j)).
\end{gather*}
As with the usual moduli spaces, we assume all $\Mt_{k,l+1}(\beta)$ are smooth orbifolds with corners.
In the special case when the family $\{J_t\}_{t\in I}$ is constant $J_t\equiv J$ or has the form $J_t=\varphi_t^*J$ for a family of symplectomorphisms $\varphi_t$, the smoothness of $\Mt_{k,l+1}(\beta)$ is exactly equivalent to the smoothness of $\M_{k,l+1}(\beta;J_t)$ for any one $t$ and therefore for all $t$. In general, smoothness of $\M_{k,l+1}(\beta;J_t)$ for all $t$ does not follow from smoothness of $\Mt_{k,l+1}(\beta)$. It is only imposed in Section~\ref{sssec:psp} where it is needed.

Let
\[
p:I\times L\lrarr I,\qquad p_\M: \Mt_{k,l+1}(\beta)\lrarr I,
\]
denote the projections.

Define
\[
\pt_{k,l}^{\beta}:A^*(I\times L;R)^{\otimes k}\otimes A^*(I\times X;R)^{\otimes l}\lrarr \A^*(I\times X;R)
\]
by
\begin{gather*}
\pt_{k,l}^{\beta}(\otimes_{j=1}^k\at_j;\otimes_{j=1}^l\gt_j):= (-1)^{\varepsilon_p(\at)}(\evit_0)_*(\bigwedge_{j=1}^l\evit_j^*\gt_j \wedge\bigwedge_{j=1}^k\evbt_j^*\at_j)),
\quad (k,l,\beta)\ne (0,0,\beta_0),\\
\pt_{0,0}^{\beta_0}:=0.
\end{gather*}
Denote by
\[
\pt_{k,l}:A^*(I\times L;R)^{\otimes k}\otimes A^*(I\times X;R)^{\otimes l}\lrarr \A^*(I\times X;R)
\]
the sum over $\beta$:
\begin{gather*}
\pt_{k,l}(\otimes_{j=1}^k\at_j;\otimes_{j=1}^l\gt_j):=
\sum_{\beta\in \sly}
T^{\beta}\pt_{k,l}^{\beta} (\otimes_{j=1}^k\at_j;\otimes_{j=1}^l\gt_j).
\end{gather*}

Relabelling the marked points of $\Mt_{k+1,l}(\beta)$ into $\vec{z}=(z_0,\ldots,z_k)$, $\vec{w}=(w_1,\ldots,w_l)$, define
\begin{equation}\label{eq:qt}
\qt_{k,l}^{\beta}: A^*(I\times L;R)^{\otimes k}\otimes A^*(I\times X;R)^{\otimes l}\lrarr A^*(I\times L;R)
\end{equation}
for $(k,l,\beta) \not\in\{ (1,0,\beta_0),(0,0,\beta_0)\}$ by
\[
\qt_{k,l}^{\beta}(\otimes_{j=1}^k\at_j;\otimes_{j=1}^l\gt_j):= (-1)^{\varepsilon(\at)}(\evbt_0)_*(\bigwedge_{j=1}^l\evit_j^*\gt_j \wedge\bigwedge_{j=1}^k\evbt_j^*\at_j)),
\]
and set $\qt_{1,0}^{\beta_0}(\at)=d\at$ and $\qt_{0,0}^{\beta_0}:=0$.
Denote by
\[
\qt_{k,l}:A^*(I\times L;R)^{\otimes k}\otimes A^*(I\times X;R)^{\otimes l}\lrarr A^*(I\times L;R)
\]
the sum over $\beta$:
\[
\qt_{k,l}(\otimes_{j=1}^k\at_j;\otimes_{j=1}^l\gt_j):=
\sum_{\beta\in \sly}
T^{\beta}\qt_{k,l}^{\beta}(\otimes_{j=1}^k\at_j;\otimes_{j=1}^l\gt_j).
\]

Lastly, define similar operations using spheres,
\[
\qt_{\emptyset,l}:A^*(I\times X;R)^{\otimes l}\lrarr A^*(I\times X;R),
\]
as follows. For $\beta\in H_2(X;\Z)$ let
\[
\Mt_{l+1}(\beta):=\{(t,u)\;|\,t\in I, u\in \M_{l+1}(\beta;J_t)\}.
\]
For $j=0,\ldots,l,$ let
\begin{gather*}
\evt_j^\beta:\Mt_{l+1}(\beta)\to I\times X,\\
\evt_j^\beta(t,(\Sigma,u,\vec{w})):=(t,u(w_j)),
\end{gather*}
be the evaluation maps. Assume that all the moduli spaces $\Mt_{l+1}(\beta)$ are smooth orbifolds and $\evt_0$ is a submersion. Recall that $w_{\s} \in H^2(X;\Z/2\Z)$ is the class with $w_2(TL) = i^* w_{\s}$ determined by the relative spin structure $\s$. Set
\begin{gather*}
\qt_{\emptyset,l}^\beta(\gt_1,\ldots,\gt_l):=
(-1)^{w_\s(\beta)}
(\evt_0^\beta)_*(\bigwedge_{j=1}^l(\evt_j^\beta)^*\gt_j)
\quad \mbox{for $(l,\beta)\ne (1,0),(0,0)$},\\
\qt_{\emptyset,1}^0:= 0,\qquad \qt_{\emptyset,0}^0:= 0,\\
\qt_{\emptyset,l}(\gt_1,\ldots,\gt_l):=
\sum_{\beta\in H_2(X)}T^{\pr(\beta)}\qt_{\emptyset,l}^\beta(\gt_1,\ldots,\gt_l).
\end{gather*}

\begin{prop}\label{prop:ptstr}
Consider lists $\at=(\at_1,\ldots,\at_k),$ $\at_j\in A^*(I\times L;R)$, and $\gt=(\gt_1,\ldots,\gt_l),$ $\gt_j\in A^*(I\times X;R).$
Then
\begin{multline}
d\pt_{k,l}(\at;\gt)
=
\pt_{k,l}(\at; d(\gt))+ \\
+
\sum_{\substack{I\sqcup J=[l]\\ \sigma\in\Z/k\Z}}
(-1)^{s_\sigma^{[1]}(\at)+|\gt| +s_{\sigma_{I\sqcup J}}(\gt)+(n+1)(|\gt_J|+1)}
\pt_{k_1,|I|}(\qt_{k_2,|J|} ((\at^\sigma)_{(1)};\gt_J) \otimes(\at^\sigma)_{(2)}; \gt_{I})\\
+
\delta_{k,0}\cdot(-1)^{|\gt|} \qt_{\emptyset,l+1}(\gt\otimes i_*1).\notag
\end{multline}
\end{prop}

\begin{proof}
The proof is analogous to Proposition~\ref{prop:pstructure}, with the following two differences:
The gluing sign from Proposition~\ref{rem:pbdgluing} becomes $(k_1k_2+k_2+n)+1,$
and the contribution of $s=\dim M$ to the sign of Proposition~\ref{stokes} becomes $\dim \Mt_{k,l+1}(\beta) = \dim \M_{k,l+1}(\beta) + 1$, so the total computation of the sign results in the same value.
\end{proof}

\subsection{Properties}
The properties formulated for the $\p$-operators can be equally well formulated for the $\pt$-operators, with similar proofs. Below we discuss them explicitly, and add one that is specific to the pseudoisotopy context.

\subsubsection{Linearity}
\begin{prop}\label{prop:pt_lin}
The operations $\pt$ are $\mR$-multilinear in the sense that for all $f\in \mR,$
\[
\ptkl^\beta(\at_1,\ldots,\at_{i-1},f.\at_i,\ldots,\at_k;\etat)
	=(-1)^{|f|\cdot\big(n+1+\V\at_{[1:i-1]}\V+|\etat|\big)}
	f.\ptkl^\beta(\at_1,\ldots,\at_k;\etat),
\]
and for $f\in A^*(I;R),$ we have
\[
\ptkl^\beta(\lstt;\gt_1,\ldots,f.\gt_i,\ldots,\gt_l)
=(-1)^{|f|\cdot|\gt_{[1:i-1]}|}
		f.\ptkl^\beta(\lstt;\gt_1,\ldots,\gt_l).
\]
\end{prop}

\begin{proof}

Let $p_X:I\times X\to I$ be the projection.

For the first identity, consider
\begin{align*}
(\evit_0)_*(\evit^*\etat\wedge \evbt_{[1:i-1]}^* & \lstt  \wedge \evbt_i^*(p^*f\wedge\at_i)\wedge\evbt_{[i+1:k]}^*\lstt)=\\
&=(\evit_0)_*(\evit^*\etat\wedge\evbt_{[1:i-1]}^*\lstt\wedge(p\circ\evbt_i)^*f\wedge\evbt_{[i:k]}^*\lstt)\\
&=(-1)^{|f|\cdot\big(|\lstt_{[1:i-1]}|+|\etat|\big)} (\evit_0)_*((p\circ\evbt_i)^*f\wedge \evit^*\etat \wedge\evbt^*\lstt)\\
&=(-1)^{|f|\cdot\big(|\lstt_{[1:i-1]}|+|\etat|\big)} (\evit_0)_*((p_X\circ\evit_0)^*f\wedge\evit^*\etat\wedge\evbt^*\lstt)\\
&=(-1)^{|f|\cdot\big(|\lstt_{[1:i-1]}|+|\etat|\big)} (p_X^*f)\wedge(\evit_0)_*(\evit^*\etat\wedge\evbt^*\lstt).
\end{align*}
The result follows from adding the change in $\varepsilon_p$:
\[
\varepsilon_p(\at_1,\ldots,\at_{i-1},f.\at_i,\ldots,\at_k) -\varepsilon_p(\at_1,\ldots,\at_k)
=
(n+i)|f|.
\]

The second equality follows from
\begin{align*}
(\evit_0)_*(\evit_{[1:i-1]}^*\etat \wedge  \evit_i^*(p_X^*f & \wedge\gt_i)\wedge\evit_{[i+1;l]}^*\etat \wedge\evbt^*\lstt)=\\
=&(-1)^{|f|\cdot|\etat_{[1:i-1]}|} (\evit_0)_*(\evit_i^*p_X^*f\wedge\evit^*\etat \wedge\evbt^*\lstt)\\
=&(-1)^{|f|\cdot|\etat_{[1:i-1]}|} (\evit_0)_*((p_X\circ \evit_j)^*f\wedge\evit^*\etat\wedge\evbt^*\lstt)\\
=&(-1)^{|f|\cdot|\etat_{[1:i-1]}|}(\evit_0)_*((p_X\circ \evit_0)^*f\wedge\evit^*\etat\wedge\evbt^*\lstt)\\
=&(-1)^{|f|\cdot|\etat_{[1:i-1]}|}(p_X^*f) \wedge(\evit_0)_*(\wedge\evit^*\etat\wedge\evbt^*\lstt),
\end{align*}
while $\varepsilon_p$ is not affected.

\end{proof}

\subsubsection{Pseudoisotopy}\label{sssec:psp}
For $t\in I$ and $M=pt,L,X,$ denote by $j_t:M\hookrightarrow I\times M$ the inclusion $x\mapsto (t,x)$. Assume the moduli spaces $\M_{k,l+1}(\beta;J_t)$ are smooth. Denote by $\pkl^t$ the $\p$-operators associated to the complex structure $J_t$.
\begin{prop}\label{lm:pseudo}
For $\at_1,\ldots,\at_k\in A^*(I\times L;R)$ and $\gt=(\gt_1,\ldots,\gt_l)\in A^*(I\times X;R)^{\otimes l}$,
and $t\in I$, we have
\[
j_t^*\pt_{k,l}(\at;\gt)=
\pkl^t(j_t^*\at;j_t^*\gt).
\]
\end{prop}

The proof is verbatim as the analogous proof for $\qt$, given in~\cite[Proposition 4.7]{ST1}, but with $\evit_0, evi_0,$ instead of $\evbt_0, evb_0,$ respectively.

\subsubsection{Unit}
\leavevmode
Denote by $1_{I\times L}$ the constant function on $A^*(I\times L)$.
\begin{prop}
For $\at_1,\ldots,\at_k\in A^*(I\times L;R)$ and $\gt=(\gt_1,\ldots,\gt_l)\in A^*(I\times X;R)^{\otimes l}$,
\begin{equation*}
	\pt_{\,k+1\!,l}^{\beta}(\at_1,\ldots,\at_{i-1},1_{I\times L},\at_{i}, \ldots,\at_k ;\otimes_{r=1}^l\gt_r)=
\begin{cases}
0,&(k+1,l,\beta)\ne (1,0,\beta_0),\\
(-1)^{n+1}i_*{1_{I\times L}},&(k+1,l,\beta)= (1,0,\beta_0).
\end{cases}
		\end{equation*}
\end{prop}

\begin{proof}
Repeat the proof of Proposition~\ref{prop:p_unit} with $\Mt$, $\evit_j$, $\evbt_j$, and $\pt$, instead of $\M$, $evi_j$, $evb_j$, and $\p$, respectively. In the case $(k,l,\beta)=(2,0,\beta),$ the map $\evit_0$ gives an orientation preserving identification of $\Mt_{1,1}(\beta_0)$ with $I\times L$, and the rest of the computation is again the same.

\end{proof}

\subsubsection{Degree}
\begin{prop}
For $k\ge 0$ and $\gt=(\gt_1,\ldots,\gt_l)$ with $|\gt_j|=2$ for all $j$,  the map
\[
\ptkl(\; ;\gt):A^*(I\times L;R)^{\otimes k}\lrarr \A^*(I\times X;R)
\]
is of degree $n+1-k$.
\end{prop}

\begin{proof}
Note that $\rdim(evb_0)=\rdim(\evbt_0)$. Therefore, the proof of Proposition~\ref{prop:p_deg} is valid verbatim in our case, with $\p$ replaced by $\pt$ and $evi_0$ by $\evit_0$.

\end{proof}

\subsubsection{Symmetry}
The proofs of the following propositions are the same as Propositions~\ref{prop:p_sgn} and~\ref{prop:p_int_sgn}, respectively, with $\pt, \evit_j, \evbt_j,$ instead of $\p, evi_j, evb_j,$ respectively.

\begin{prop}[Cyclic symmetry of boundary input]\label{cl:bdsym}
For a cyclic permutation $\sigma\in\Z/k\Z$, $\at=(\at_1,\ldots,\at_k)$, and $\gt=(\gt_1,\ldots,\gt_l)$, we have
\[
\ptkl(\at;\gt)=
(-1)^{s_\sigma^{[1]}(\at)}\ptkl(\at^\sigma;\gt)
\]
with $s_\sigma^{[1]}(\lstt^\sigma)$ as defined before Proposition~\ref{prop:pstructure}.
\end{prop}

\begin{prop}[Symmetry of interior input]\label{prop:intsym}
For any permutation $\sigma\in S_k$ and any $\at=(\at_1,\ldots,\at_k)$ and $\gt=(\gt_1,\ldots,\gt_l)$, we have
\[
\ptkl(\at;\gt)=(-1)^{s_\sigma(\gt)}\pkl(\at;\gt^\sigma).
\]
\end{prop}

\subsubsection{Energy zero}
\begin{prop}
For any $\at=(\at_1,\ldots,\at_k)$, $\gt=(\gt_1,\ldots,\gt_l)$, we have
	\[
	\pt_{k,l}^{\beta_0}(\at_1,\ldots,\at_k;\gt_1,\ldots,\gt_l)=
	\begin{cases}
	0,&(k,l)\ne (1,0),\\
	(-1)^{(n+1)\V\at_1\V}i_*\at_1, &(k,l)=(1,0).
	\end{cases}
	\]
\end{prop}

\begin{proof}
Note that $\rdim(evi_j)=\rdim(\evit_j)$ and $\rdim(evb_j)=\rdim(\evbt_j)$ for any $j$. Therefore the proof of Proposition~\ref{prop:p_zero} is valid verbatim in our case, with $\p$ replaced by $\pt$ everywhere.

\end{proof}

\subsubsection{Fundamental class}
\leavevmode
Denote by $1_{I\times X}\in A^0(I\times X)$ the constant function with value $1$.
\begin{prop}
For all  $\at=(\at_1,\ldots,\at_k)$, $\gt=(\gt_1,\ldots,\gt_{l-1})$, we have
\[
\pkl(\at;1_{I\times X},\gt_1,\ldots,\gt_{l-1})=0.
\]
\end{prop}

The proof is similar to that of Proposition~\ref{cl:pfund}.

\subsubsection{Divisor}
\begin{prop}
Let $\at=(\at_1,\ldots,\at_k)$ and $\gt=(\gt_1,\ldots,\gt_{l-1})$.
Assume $\gt'\in A^2(I\times X, I\times L)\otimes R$ and $d\gt' = 0$. Then
	\begin{equation}
	\ptkl^{\beta}(\lstt;\gt'\otimes \gt)=
	\left(\int_\beta\gt'\right) \cdot\pt_{k,l-1}^{\beta}(\at;\gt).
	\end{equation}
\end{prop}

As with~\cite[Proposition 4.16]{ST1}, the proof of Proposition~\ref{cl:pdiv} holds with $\Mt, \evit_j, \evbt_j,$ and $\pt,$ instead of $\M, evi_j, evb_j,$ and $\p$.

\subsubsection{Top degree}

In this section, we use the notation introduced in Section~\ref{sssec:topdeg}.

\begin{prop}
Suppose $(k,l+1,\beta)\ne (1,1,\beta_0)$. Then, for all lists $\at=(\at_1,\ldots,\at_k)$, $\gt=(\gt_1,\ldots,\gt_l)$, we have $\ptkl^\beta(\at;\gt)\in\A^{<2n+1}(I\times X)\otimes R$.
\end{prop}

\begin{proof}
Follow the proof of Proposition~\ref{prop:p_no_top_deg} with $\p$ replaced by $\pt$ and $evi_0$ by $\evit_0$. In this case, $\rdim \evit_0=\dim\Mt_{k,l+1}(\beta)-2n-1$, so the assumption $\deg^d (\ptkl^\beta(\lstt;\etat))=2n+1$ is what implies $\deg^d(\xi)=\dim\Mt_{k,l+1}(\beta)$. The rest of the proof is then valid.

\end{proof}

\subsubsection{Bulk and boundary deformation}

Let $\gt \in \mI_R A^*(I\times X;R)$ such that $|\gt|=2$ and $d\gt=0$. Let $\bt\in \mI_R A^*(I\times L;R)$ such that $|\bt|=1$.
For $k,l\ge 0$ and $\at=(\at_1,\ldots,\at_k)\in \big(A^*(I\times L;R)\big)^{\otimes k}$, $\etat=(\etat_1,\ldots,\etat_l)\in \big(A^*(I\times X;R)\big)^{\otimes l}$, define the deformed maps by
\begin{equation*}
\ptbg_{k,l}(\at;\etat):=\sum_{s,t\ge 0}
\sum_{\sum_{j=0}^{k-1}i_j=s}
\frac{1}{t!}
\ptbg_{k+s,l+t}(\bt^{\otimes i_0}\otimes \at_1\otimes \bt^{\otimes i_1}\otimes \cdots \otimes \bt^{\otimes i_{k-1}}\otimes \at_k ;\etat\otimes \gt^{\otimes t}).
\end{equation*}
Again, these deformed operators satisfy a structure equation similar to that satisfied by $\pt$, that follows from Propositions~\ref{prop:ptstr} and~\ref{prop:intsym}:
\begin{cor}
Consider lists $\at=(\at_1,\ldots,\at_k)$, $\at_j\in A^*(I\times L;R)$, and $\etat=(\etat_1,\ldots,\etat_l)$, $\etat_j\in A^*(I\times X;R)$.
Then
\begin{multline*}
d\ptbg_{k,l}(\at;\etat)
=
\ptbg_{k,l}(\at; d(\etat))+ \\
\qquad+
\sum_{\substack{I\sqcup J=[l]\\ \sigma\in\Z/k\Z}}
(-1)^{s_\sigma^{[1]}(\at)+|\etat| +s_{\sigma_{I\sqcup J}}(\etat)+(n+1)(|\etat_{J}|+1)}
\ptbg_{k_1,|I|}(\qtbg_{k_2,|J|} ((\at^\sigma)_{(1)};\etat_J) \otimes(\at^\sigma)_{(2)}; \etat_{I})+\\
+
\delta_{k,0}\cdot(-1)^{|\etat|}\qt^{\gt}_{\emptyset,l+1}(\etat\otimes i_*1_L).\notag
\end{multline*}
\end{cor}

\subsection{Open-closed maps over pseudoisotopies}\label{ssec:ocps}

The definition of Hochschild and cyclic chain complexes and their various
versions, as discussed in \Cref{sec:Hochschild complexes}, can be extended to
$A_{\infty}$-algebras over a coefficient ring which is a differential graded-commutative
$k$-algebra. See \cite{GitermanSolomon2} for details.
The $\qt$ operators defined in~\eqref{eq:qt} endow $\mC:=A^*(I\times L;R)$ with the structure of an $A_\infty$-algebra over the dga $\mR = A^*(I;R)$, as discussed in~\cite{ST1}.
Using the notions above, one can immediately deduce analogs of Theorems 1-5 for the map
$\ptg$.
To avoid a lengthy digression into the technical machinery required to handle a base dga, we omit the details here.

The primary application of such constructions would be in proving invariance statements. For example, if $\gamma$ and $\gamma'$ represent the same cohomology class in $X$, we can fix $J_t\equiv J$, take $\xi$ such that $d\xi=\gamma'-\gamma$, and set $\gt: = \gamma+t(\gamma' - \gamma)+dt\wedge \xi$. Then, by Proposition~\ref{lm:pseudo}, the maps $\tilde{\p}^{\tilde{\gamma}}$ give a pseudo-isotopy between $\pg$ and $\p^{\gamma'}$, and the same is true for the induced open-closed maps on chain level. If we show that the open-closed maps on homology are pseudo-isotopy invariant, this will prove invariance of the construction under a variation of $\gamma$ within cohomology class. Note that in this case the regularity assumptions for $\Mt_{k,l+1}(\beta)$ hold if and only if they hold for $\M_{k,l+1}(\beta)$. Other invariance questions include the variation of $\omega$ and of $J$. Cf. Theorems~1 and~4 in~\cite{ST2}. We do not verify here whether or not the open-closed maps actually are invariant under pseudo-isotopy; a very similar but not identical calculation is done in Section 3.1 of~\cite{T19}.

\bibliography{../../bibliography_exp}
\bibliographystyle{../../amsabbrvcnobysame}

\end{document}